\documentclass[a4paper,10pt]{article}
\usepackage{mathtext}
\usepackage[T1,T2A]{fontenc}
\usepackage[cp1251]{inputenc}
\usepackage[english]{babel}
\usepackage{amsmath}
\usepackage{amsfonts}
\usepackage{amssymb}
\usepackage{mathrsfs}
\usepackage{amsthm}
\usepackage{enumerate}
\usepackage{graphicx}

\usepackage{pb-diagram}
\usepackage{tikz-cd}

\usepackage{euscript}

\newtheorem{Le}{Lemma}[section]
\newtheorem{Def}[Le]{Definition}
\newtheorem{St}[Le]{Proposition}
\newtheorem{Th}{Theorem}[section]
\newtheorem{Cor}[Le]{Corollary}
\newtheorem{Rem}[Le]{Remark}

\newtheorem{Ex}[Le]{Example}
\numberwithin{equation}{subsection}

\newcommand{\R}{\mathbb{R}}
\newcommand{\Co}{\mathbb{C}}
\newcommand{\N}{\mathbb{N}}
\newcommand{\Z}{\mathbb{Z}}

\newcommand{\eps}{\varepsilon}

\newcommand{\eq}[1]{\begin{equation}{#1}\end{equation}}
\newcommand{\mlt}[1]{\begin{multline}{#1}\end{multline}}
\newcommand{\alg}[1]{\begin{align}{#1}\end{align}}

\newcommand{\set}[2]{\{{#1}\mid{#2}\}}
\newcommand{\sset}[2]{\bigg\{{#1}\,\bigg|\;{#2}\bigg\}}
\newcommand{\Set}[2]{\Big\{{#1}\,\Big|\;{#2}\Big\}}
\newcommand{\scalprod}[2]{\langle{#1},{#2}\rangle}
\newcommand{\fdot}{\,\cdot\,}
\newcommand{\twodots}{\, .. \,}
\newcommand{\Eref}[1]{\stackrel{#1}{=}}
\newcommand{\Eeqref}[1]{\stackrel{\scriptscriptstyle{\eqref{#1}}}{=}}
\newcommand{\EeqrefTwo}[2]{\stackrel{\scriptscriptstyle{\genfrac{}{}{0pt}{-2}{\eqref{#1}}{\eqref{#2}}}}{=}}
\newcommand{\Leqref}[1]{\stackrel{\scriptscriptstyle{\eqref{#1}}}{\leq}}
\newcommand{\Lseqref}[1]{\stackrel{\scriptscriptstyle{\eqref{#1}}}{\lesssim}}

\newcommand{\LseqrefTwo}[2]{\stackrel{\scriptscriptstyle{\genfrac{}{}{0pt}{-2}{\eqref{#1}}{\eqref{#2}}}}{\lesssim}}
\newcommand{\Lref}[1]{\stackrel{#1}{\leq}}
\newcommand{\Lsref}[1]{\stackrel{#1}{\lesssim}}

\newcommand{\heat}{\HHeat}
\newcommand{\Heat}{{\boldsymbol \HHeat}}
\DeclareMathOperator{\HHeat}{H}
\DeclareMathOperator{\Dil}{Dil}
\DeclareMathOperator{\DIL}{DIL}
\DeclareMathOperator{\I}{I}
\DeclareMathOperator{\loc}{loc}
\DeclareMathOperator{\CO}{Co}
\DeclareMathOperator{\Fl}{Fl}
\DeclareMathOperator{\BV}{BV}
\DeclareMathOperator{\curl}{curl}
\DeclareMathOperator{\sign}{sign}
\newcommand{\Disp}{\mathbb{D}}
\newcommand{\E}{\mathbb{E}}
\newcommand{\mass}{{\bf \mathrm{m}}}
\newcommand{\m}{\mathfrak{m}}

\newcommand{\GenTree}{\EuScript{T}}

\DeclareMathOperator{\Image}{Im}
\DeclareMathOperator{\Lip}{Lip}

\newcommand{\W}{\mathcal{W}}
\newcommand{\WW}{\mathbb{W}}
\newcommand{\Sw}{\mathcal{S}}
\newcommand{\M}{\mathbb{M}}
\newcommand{\Max}{\mathrm{M}}
\DeclareMathOperator{\s}{s}
\DeclareMathOperator{\supp}{supp}
\newcommand{\one}{\vec{\bf 1}}

\DeclareMathOperator{\Smooth}{\mathfrak{S}}
\newcommand{\Tree}{\mathcal{T}}
\newcommand{\TS}{\boldsymbol{T}}
\newcommand{\TREE}{\mathfrak{T}}

\newcommand{\tf}{\tilde{f}}

\newcommand{\GGamma}{\boldsymbol{\Gamma}}
\newcommand{\vn}{\vec{n}}

\newcommand{\ldH}{\underline{\dim}_{\mathrm{H}}}

\newcommand{\dr}{d^{\vec{r}}}
\newcommand{\codr}{\partial^{\vec{r}}}

\title{Anisotropic Bourgain--Brezis inequalities}
\author{Dmitriy Stolyarov\thanks{Supported by the Russian Science Foundation grant N 24-71-10011.}
}
\begin{document}
\maketitle

\begin{abstract}
We provide an adjustment of the Hardy--Littlewood--Sobolev inequality for~$p=1$ to the anisotropic setting. Several examples of anisotropic Bourgain--Brezis inequalities are obtained as corollaries of the main theorem. 
\end{abstract}

\tableofcontents

\section{Introduction}\label{S1}

\subsection{Classical theory}\label{s11}

The classical Hardy--Littlewood--Sobolev inequality says the Riesz potential~$\I_\alpha$ maps~$L_p(\R^d)$ to~$L_q(\R^d)$ continuously whenever~$1 < p < q < \infty$ and the parameters satisfy the homogeneity condition~$1/p - 1/q = \alpha /d$.  In other words,
\eq{
\|\I_\alpha f\|_{L_q} \lesssim \|f\|_{L_p}.
}
The sign~`$\lesssim$' hides a multiplicative constant independent of the choice of~$f$. The inequality was introduced by Sobolev in~\cite{Sobolev1938} to prove what is now called the Sobolev embedding theorem~$W_p^l \hookrightarrow L_q$,~$1 < p < q < \infty$, and~$1/p - 1/q = l/d$. Later, Gagliardo~\cite{Gagliardo1959} and Nirenberg~\cite{Nirenberg1959} showed that the embedding theorem holds true at the endpoint~$p=1$.  This case is important for the study of functions of bounded variation.  The Gagliardo--Nirenberg--Sobolev embedding admits a Lorentz space refinement, as was shown by Alvino~\cite{Alvino1977}. This hints that while the Hardy--Littlewood--Sobolev inequality fails at the endpoint~$p=1$, there should be some modification that holds. The development of the so-called Bourgain--Brezis inequalities, e.g., in~\cite{BourgainBrezis2003},~\cite{BourgainBrezis2004},~\cite{BourgainBrezis2007},~\cite{CVSYu2017},~\cite{GRvS2019},~\cite{GRV2024},~\cite{HRS2023},~\cite{HernandezSpector2024},~\cite{LanzaniStein2005},~\cite{Mazya2010},~\cite{SpectorVanSchaftingen2019},~\cite{VanSchaftingen2004},~\cite{VanSchaftingen2008},~\cite{VanSchaftingen2013}, emphasized the existence of such a modification. The reader may find more historical information in the surveys~\cite{Spector2020},~\cite{VanSchaftingen2014}, and the lecture notes~\cite{VanSchaftingen2024}.

By~$\Sw'(\R^d,\R^\ell)$ we mean the space of~$\R^\ell$-valued tempered distributions.
\begin{Th}[Theorem~$1$ in~\cite{Stolyarov2022}]\label{OldMainTheorem}
Let~$\W$ be a closed translation and dilation invariant subspace of~$\Sw'(\R^d,\R^\ell)$. The inequality
\eq{
\|\I_\alpha f\|_{L_{q}} \lesssim \|f\|_{L_1},\qquad f\in \W, \quad \alpha \in (0,d),\quad \text{and}\quad  q= \frac{d}{d-\alpha} \in (1,\infty),
}
holds true if and only if~$\W$ does not contain distributions of the type~$a\otimes \delta_0$, where~$a\in \R^\ell \setminus \{0\}$ and~$\delta_0$ is the Dirac delta at the origin.
\end{Th}
See~\cite{Stolyarov2022} for applications and historical remarks concerning Bourgain--Brezis inequalities. The standpoint of~\cite{Stolyarov2022} was that the theme of Sobolev embeddings and Hardy--Littlewood--Sobolev inequality is a phenomenon in harmonic analysis. The aim of the present paper is twofold. 

First, we  extend the theory of~\cite{Stolyarov2022}, in particular, Theorem~\ref{OldMainTheorem}, to the anisotropic setting, where the homogeneity with respect to different coordinates is different. The classical Sobolev embeddings have their anisotropic counterparts, see the monographs~\cite{BIN1979} and~\cite{Triebel2006}. As for the Bourgain--Brezis inequalities, seemingly, not much has been done (however, see~\cite{KMS2015} and~\cite{Stolyarov2021bibis}). The anisotropic setting shows the limitation of the classical methods such as the isoperimetry, the co-area formulas, or integration by parts in the spirit of Gagliardo and Nirenberg. Our aim is to show that the Bourgain--Brezis inequalities in the anisotropic setting are amenable to the harmonic analysis approach. 

Second, the proof of Theorem~\ref{OldMainTheorem} presented in~\cite{Stolyarov2022} is long and involved. While we do not give an essentially new argument, we present several shortcuts that were unnoticed in~\cite{Stolyarov2022}. We also consider many examples and explanations that show that some other natural simplifications are impossible. In the forthcoming subsection, we introduce the anisotropic formalism and state the results. After that, in Subsection~\ref{s13'}, we will provide the plan of the paper.


\subsection{Basics of anisotropic formalism}\label{s12}

The classical theory of anisotropic Sobolev spaces is described in Chapter~$3$ of~\cite{BIN1979}. The reader may find a more Fourier analytic approach to this subject in Section~$5$ of~\cite{Triebel2006}.

We call a vector~$a\in \R^d$  with positive coordinates and such that~$\sum_{j=1}^d a_j = d$ an \emph{anisotropy}\footnote{Sometimes it is called a \emph{pattern of homogeneity}.}. Consider the group of affine transforms
\eq{
\Dil_t\colon \R^d \to \R^d,\ t > 0;\qquad \Dil_t(x) = \big(t^{a_1}x_1,t^{a_2}x_2,\ldots,t^{a_d}x_d\big),\ x\in \R^d.
}
These transforms are anisotropic versions of the Euclidean dilations~$x\mapsto t x$,~$x\in \mathbb{R}^d$. Let~$m \in \R$. A function~$\Phi\colon \R^d \setminus \{0\}\to \R^\ell$,~$\ell \in \N$, is called~$m$-\emph{homogeneous} with respect to~$a$, provided
\eq{
\Phi(\Dil_t(x)) = t^{m}\Phi(x),\qquad t > 0,\ x\in \R^d\setminus \{0\}.
}
\begin{Ex}
Let~$d=2$ and~$a = (4/3,2/3)$. The function~$f(x,y) = x$ is~$4/3$-homogeneous, the function~$f(x,y) = y$ is~$2/3$-homogeneous, and the function~$xy$ is~$2$-homogeneous. The function~$f(x,y) = x+y$ is not homogeneous with respect to the chosen anisotropy, however, the function~$x+y^2$ is~$4/3$-homogeneous.
\end{Ex}
We fix an anisotropy and call functions simply homogeneous. Each anisotropy generates a function~$\rho\colon \R^d \to \R$ that replaces the isotropic Euclidean norm. It is defined implicitly:
\eq{\label{AnisotropicNorm}
\sum\limits_{j=1}^d \frac{x_j^2}{\rho^{2a_j}(x)} = 1,\quad x\in \R^d \setminus \{0\};\qquad \rho(0) = 0.
}
Note that~$\rho$ is a continuous, even, and~$1$-homogeneous function. The definition~\eqref{AnisotropicNorm} may be rephrased as~$\Dil_{1/\rho(x)}(x) \in S^{d-1}$, where the latter symbol denotes the unit sphere in~$\R^{d}$. This, in particular, yields
\eq{
\rho(x) \asymp \Big(\sum\limits_{j=1}^d |x_j|^{2/a_j}\Big)^\frac12.
}
Here and in what follows, the notation~$A \asymp B$ means~$A \lesssim B$ and~$B \lesssim A$.

We will also use two types of dilations of functions: the one that preserves the integral,
\eq{\label{L1Dilation}
\Dil_t[f](x) = t^{-d} f\big(\Dil_{t^{-1}} x\big) = t^{-d} f\big( t^{-a_1}x_1, t^{-a_2}x_2,\ldots, t^{-a_d}x_d\big),\qquad x\in \R^d,\ t > 0,
}
and the one that preserves the values of functions:
\eq{\label{LinftyDilation}
\Dil^t[f](x) = f\big(\Dil_{t^{-1}} x\big) =  f\big( t^{-a_1}x_1, t^{-a_2}x_2,\ldots, t^{-a_d}x_d\big),\qquad x\in \R^d,\ t > 0.
}
We will use dilations frequently and the author finds Fig.~\ref{Figure0} helpful, for example, to distinguish~$\Dil_t$ from~$\Dil_{1/t}$.
\begin{figure}[h!]
\centerline{
\includegraphics[height=5cm]{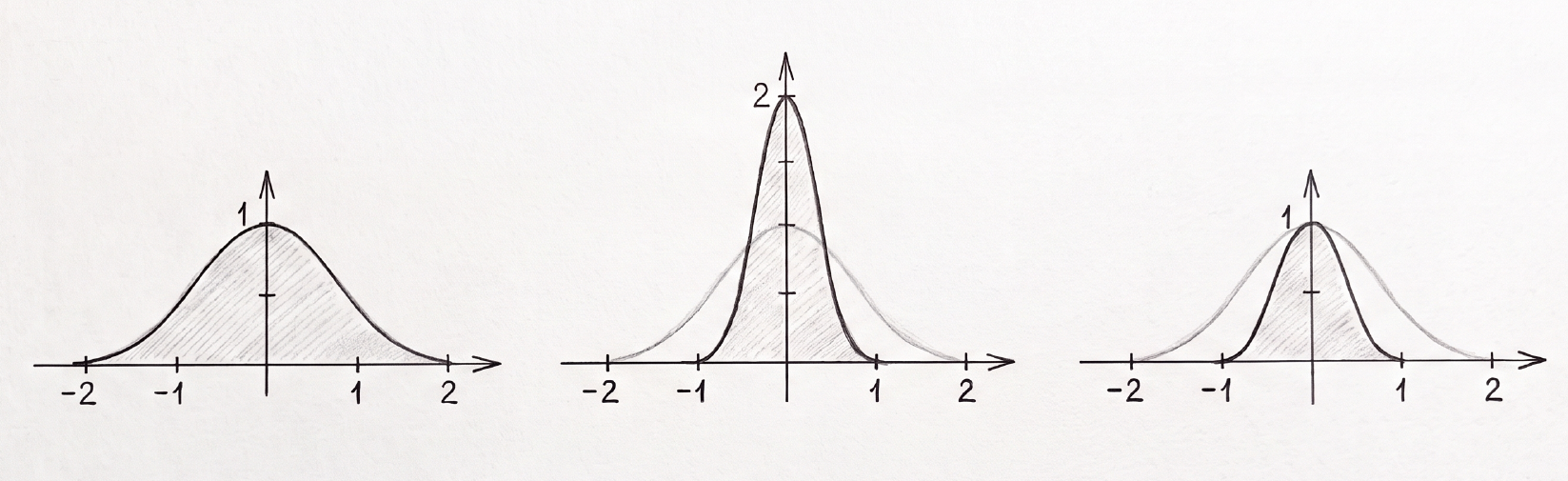}
}
\caption{A function~$f$ with its dilations~$\Dil_{1/2} f$ and~$\Dil^{1/2} f$.}
\label{Figure0}
\end{figure}

We will also apply dilations of the first type to measures:
\eq{
\Dil_t[\mu](A) = \mu \big(\Dil_{t^{-1}}A\big),\qquad \Dil_{s}A = \set{\Dil_s x}{x\in A},\quad s,t > 0,
}
here~$\mu$ is a measure and~$A\subset \R^d$ is a Borel set with finite~$\mu$-variation. Note that the definitions agree in the case where~$\mu$ is absolutely continuous: 
\mlt{
\Dil_t[\mu](A) =  \mu \big(\Dil_{t^{-1}}A\big) = \int\limits_{\Dil_{t^{-1}} A} f(y)\,dy = \int\limits_A f\big( \Dil_{t^{-1}} x\big) \,d (\Dil_{t^{-1}} x)\\
 \Eref{\scriptscriptstyle \sum a_j = d} t^{-d} \int\limits_{A} f\big( \Dil_{t^{-1}} x\big)\,dx = \int\limits_A \Dil_{t} f (x)\,dx,\qquad d\mu(x) = f(x)\,dx.
}

We will be using the following normalization of the Fourier transform:
\eq{
\hat{f}(\xi) = \int\limits_{\R^d}f(x)e^{-2\pi i \scalprod{x}{\xi}}\,dx,\qquad \xi \in \R^d,\quad f\in L_1(\R^d).
}
Consider an anisotropic version of the Riesz potential~$\I_\alpha$,~$\alpha \in (0,d)$, defined as the Fourier multiplier
\eq{\label{RieszPotentialDefinition}
\I_\alpha f  = \Big[(\rho(\fdot))^{-\alpha} \hat{f}(\fdot)\Big]\check{\phantom{\Big |}},\qquad f\in L_1(\R^d).
}
For the properties of the classical isotropic Riesz potentials, see Section~$1.2$ in~\cite{AdamsHedberg1996}. The heuristic meaning of the isotropic Riesz potential is that it is the most natural integral operator of order~$-\alpha$; the term `natural' refers to being rotation invariant. In the anisotropic case, there is no rotational invariance, and apparently, no comparably canonical unique operator. One may replace~$\rho^{-\alpha}$ in~\eqref{RieszPotentialDefinition} with any other reasonable~$(-\alpha)$-homogeneous function. What is important is the homogeneity property
\eq{\label{RieszPotentialHomogeneity}
\Dil_\lambda \big[\I_\alpha f\big] = \lambda^{-\alpha} \I_\alpha\big[\Dil_\lambda [f]\big],\qquad \lambda > 0,
}
which may be derived from
\eq{\label{DilationFourier}
\Dil_\lambda \hat{g} = (\Dil^{\lambda^{-1}} \!\! g)\hat{\phantom{i}}
}
as follows:
\eq{
\I_\alpha [\Dil_\lambda f] = \Big((\Dil_\lambda f)\!\hat{\phantom{I}}\rho^{-\alpha}\Big)\!\!\check{\phantom{\Big|}} = \Big(\Dil^{\lambda^{-1}}\! \big[\hat{f}\,\big]\rho^{-\alpha}\Big)\!\!\check{\phantom{\Big|}} = \lambda^\alpha \Big(\Dil^{\lambda^{-1}}\! \big[\hat{f}\,\rho^{-\alpha}\big]\Big)\!\!\check{\phantom{\Big|}} =\lambda^\alpha \Dil_\lambda\big[\I_\alpha f\big].
}

The family of operators~\eqref{RieszPotentialDefinition} satisfies the semigroup property
\eq{\label{SemigroupForRieszPotentials}
\I_{\beta} = \I_{\beta - \alpha}\circ \I_{\alpha},\qquad 0 < \alpha < \beta < d. 
}
By homogeneity, there exists an~$(\alpha -d)$-homogeneous function~$K_\alpha\colon \R^d\setminus \{0\}\to \R$ such that
\eq{
\I_\alpha f = K_\alpha*f.
} 
Seemingly, there is no concise formula for the kernel~$K_\alpha$ in our generality. In particular, the kernel~$K_\alpha$ need not be positive\footnote{The following explanation was suggested by Nikita Dobronravov. Take~$d=2$,~$a_1 = \eps$,~$a_2 = 2-\eps$, and~$\alpha = \eps/2$. As~$\eps\to 0$, the function~$\rho^{-\alpha}$ then converges to~\eq{\begin{cases} 1, &|\xi_1| \leq 1;\\ |\xi_1|^{-\frac12}, &\text{otherwise}\end{cases},} and the Fourier transform of the latter function attains values of both signs.}. Note that, similarly to the isotropic case,
\eq{\label{DivergenceOfKernel}
\int\limits_{\R^d}|K_\alpha(x)|^{\frac{d}{d-\alpha}}\,dx  = +\infty.
}
To justify this, we use the polar change of variables formula (see~$(2.6)$ in~\cite{Stolyarov2021bibis} or Section~$4.1$ in~\cite{BIN1979}; the corresponding formula is listed as~$(4)$ in that section):
\eq{
\int\limits_{\R^d}|K_\alpha(x)|^{\frac{d}{d-\alpha}}\,dx = \int\limits_{\R_+} r^{d-1} \int\limits_{S^{d-1}}\big|K_\alpha(\Dil_{r}(\zeta))\big|^{\frac{d}{d-\alpha}} \Big(\sum\limits_{j=1}^d a_j \zeta_j^2\Big)\,d\sigma(\zeta)\,dr,
}
where~$\sigma$ denotes the natural surface measure on the unit sphere. Using the homogeneity of the function~$K_\alpha$, we arrive at the divergent integral~$\int_{\R_+} dr/r$.

Since in the anisotropic setting different coordinates in~$\R^d$ have different scalings, it is also natural to consider~$L_p$-spaces that take this feature into account. To this end, let~$\vec{p}\in \R^d$ be a vector with~$p_j \geq 1$ for every~$j$. Consider the norm
\eq{
\|f\|_{L_{\vec{p}}(\R^d)} = \bigg(\int\limits_\R\bigg(\int\limits_{\R}\ldots \int\limits_{\R}\bigg(\int\limits_\R |f(x)|^{p_1}\,dx_1\bigg)^{\frac{p_2}{p_1}}\,dx_2 \ldots \,dx_{d-1}\bigg)^{\frac{p_d}{p_{d-1}}}\,dx_d\bigg)^{\frac{1}{p_d}}.
} 
One may show that this quantity defines a norm and a complete function space (see, e.g., Section~$1.1$ in~\cite{BIN1979}). In the case~$\vec{p} = (p,p,\ldots,p)$,~$p \in [1,\infty)$, we obtain the standard~$L_p$ norm. Now we are ready to formulate the anisotropic version of the classical Hardy--Littlewood--Sobolev inequality.
\begin{Th}[Hardy--Littlewood--Sobolev inequality, anisotropic form]\label{AnisotropicHLS}
Let~$\vec{p}$ and~$\vec{q}$ be vectors in~$\R^d$ such that
\eq{
\forall j\quad 1< p_j < q_j < \infty\qquad \text{and} \qquad \sum\limits_{j=1}^d \Big(\frac{1}{p_j} - \frac{1}{q_j}\Big)a_j = \alpha \in (0,d).
} 
The operator~$\I_\alpha$ maps~$L_{\vec{p}}$ to~$L_{\vec{q}}$ continuously.
\end{Th}
In other words,
\eq{\label{HLSFormula}
\|\I_\alpha f\|_{L_{\vec{q}}}\lesssim \|f\|_{L_{\vec{p}}}.
}
For the proof, see Subsection~$2.20$ in~\cite{BIN1979}. It is interesting that the most popular proofs of the classical isotropic Hardy--Littlewood--Sobolev inequality, the one based on the maximal function estimate (see p.~$354$ in~\cite{Stein1993}), and another based on interpolation of the endpoint weak-types (see, e.g., p.38 in~\cite{Peetre1976} or the original paper~\cite{ONeil1963}), seemingly, do not work well in the anisotropic situation. At least, they require significant modifications. 

We will write many inequalities in the style of~\eqref{HLSFormula}. We always assume that the inequality is true in the case where the right hand side (the one that bounds the quantity on the left) is infinite.

We are mostly interested in the case~$\vec{p} = (1,1,\ldots, 1)$, where Theorem~\ref{AnisotropicHLS} fails. Let us justify that failure on the example where all~$q_j$ are also equal: The inequality
\eq{\label{WrongHLS}
\|\I_{\alpha} f\|_{L_{d/(d-\alpha)}}\lesssim \|f\|_{L_1}
}
is false. To see this, we need a simple folklore lemma. By the symbol~$\M(\R^d;\R^\ell)$ we denote the space of charges of bounded variation with values in~$\R^\ell$; by a charge we mean a vector-valued or simply signed~$\sigma$-additive set function, while measures are always non-negative. The norm of a charge is its total variation.
\begin{Le}\label{SmoothingLemma}
Let~$\W$ be a translation invariant closed linear subspace of~$\Sw'(\R^d;\R^\ell)$. Assume the inequality
\eq{\label{SmoothingLemmaAssumption}
\|\I_{\alpha} f\|_{L_{d/(d-\alpha)}}\lesssim \|f\|_{L_1},\qquad f\in \W\cap L_1(\R^d;\R^\ell),
}
holds true with a uniform constant. Then, the inequality
\eq{\label{SmoothingLemmaConclusion}
\|\I_{\alpha} \mu\|_{L_{d/(d-\alpha)}}\lesssim \|\mu\|_{\M},\qquad \mu\in \W\cap \M(\R^d;\R^\ell),
}
also holds with the same constant.
\end{Le}
\begin{Rem}
We do not require any sort of dilation invariance here.
\end{Rem}
\begin{proof}[Proof of Lemma~\ref{SmoothingLemma}]
Let~$\{\varphi_n\}_n$ be a non-negative smooth approximation of the identity:~$\varphi_n(x) = n^d \varphi(n x)$, where~$n\in\N$ and~$\varphi$ is a smooth compactly supported non-negative function with unit integral. Pick some~$\mu \in \W\cap \M(\R^d;\R^\ell)$ and construct the approximations~$\mu_n = \mu*\varphi_n$. By translation invariance and the closedness of~$\W$, we have\footnote{See Proposition~\ref{ConvolutionTranslationInvariant} in the appendix for this folklore statement; similar principles are discussed, for example, in the classical paper~\cite{Schwartz1957}, see p. 8 and 9 of that paper.}~$\mu_n\in \W$. Moreover,~$\|\mu_n\|_{L_1} \leq \|\mu\|_{\M}$. Therefore, by our assumption~\eqref{SmoothingLemmaAssumption},
\eq{
\|\I_\alpha [\mu_n]\|_{L_{d/(d-\alpha)}} \lesssim \|\mu_n\|_{L_1}\leq \|\mu\|_{\M}.
}
It remains to note that~$\I_\alpha [\mu_n]\to \I_\alpha [\mu]$ as tempered distributions, which, together with the above inequality yields~\eqref{SmoothingLemmaConclusion}.
\end{proof}
Now we can disprove~\eqref{WrongHLS} by relying on the lemma above. In this case, we set~$\ell = 1$ and~$\W = \Sw'(\R^d)$. If~\eqref{WrongHLS} holds true, then, by Lemma~\ref{SmoothingLemma},~$\I_\alpha [\delta_0] \in L_{d/(d-\alpha)}$, where~$\delta_0$ is the Dirac delta, as well. By definition,~$\I_\alpha [\delta_0] = K_\alpha$. The~$L_{d/(d-\alpha)}$ norm of the latter function is infinite by~\eqref{DivergenceOfKernel}. This is a contradiction.

We are ready to formulate a preliminary version of our main result. Here and in what follows, the term 'vectorial delta measure' means a distribution of the form~$a\otimes \delta_0$, where~$a\in \R^\ell$ is a non-zero vector. It is clear from the reasoning above that if~$\W$ contains a vectorial delta measure, then~$\I_\alpha$ cannot map~$\W \cap L_1$ to~$L_{d/(d-\alpha)}$ continuously.

\begin{Th}\label{MainTheoremLorentzScale}
Let~$\W$ be a closed translation invariant linear subspace of~$\Sw'(\R^d; \R^\ell)$. Assume~$\W$ is also invariant under the dilations~$\Dil_t$. If~$\W$ does not contain vectorial delta measures, then~$\I_\alpha$ maps~$\W \cap L_1$ to~$L_{d/(d-\alpha)}$ continuously, whenever~$\alpha \in (0,d)$.
\end{Th}

For the convenience of notation, we will always assume
\eq{\label{DefOfq}
q = \frac{d}{d-\alpha},\qquad \text{which is the same as}\qquad \alpha = \frac{q-1}{q}d.
}
With the help of Theorem~\ref{AnisotropicHLS}, we may extend our result to the setting of anisotropic~$L_{\vec{p}}$ spaces.
\begin{Cor}\label{MixedSummabilityCorollary}
Let~$\W$ be a closed translation invariant linear subspace of~$\Sw'(\R^d; \R^\ell)$. Assume~$\W$ is also invariant under the dilations~$\Dil_t$. If~$\W$ does not contain vectorial delta measures, then~$\I_\alpha$ maps~$\W \cap L_1$ to~$L_{\vec{q}}$ continuously whenever
\eq{\label{HomogeneityForAnisotropicLp}
\forall j\quad q_j\in (1,\infty)\qquad \text{and}\quad \sum\limits_{j=1}^d\frac{a_j}{q_j} = d-\alpha.
}
\end{Cor}
\begin{proof}
Let~$\beta$ be a small positive number such that~$\beta < \alpha$ and also~$d/(d-\beta) < q_j$ for every~$j$. Then, by Theorem~\ref{MainTheoremLorentzScale} with~$\beta$ in the role of~$\alpha$,~$\I_\beta$ maps~$\W\cap L_1$ to~$L_{d/(d-\beta)}$ continuously. Consider the vector
\eq{
\vec{p} = \Big(\frac{d}{d-\beta},\frac{d}{d-\beta},\ldots,\frac{d}{d-\beta}\Big).
}
Then,~$L_{d/(d-\beta)} = L_{\vec{p}}$. By Theorem~\ref{AnisotropicHLS}, the latter space is continuously mapped by~$\I_{\alpha - \beta}$ to~$L_{\vec{q}}$ since (recall~$\sum_j a_j = d$)
\eq{
\sum\limits_{j=1}^d \Big(\frac{1}{p_j} - \frac{1}{q_j}\Big)a_j = \sum\limits_{j=1}^d \Big(\frac{d-\beta}{d} - \frac{1}{q_j}\Big)a_j \Eeqref{HomogeneityForAnisotropicLp} (d-\beta) - (d-\alpha) = \alpha - \beta.
}
It remains to use~\eqref{SemigroupForRieszPotentials}.
\end{proof}
The proof above might be summarized by the diagram:
\eq{
\begin{diagram}
\node{\W\cap L_1}
\arrow[2]{e,t,..}{\I_\alpha}
\arrow{se,b}{\I_\beta}
\node[2]{L_{\vec{q}}}\\
\node[2]{L_{d/(d-\beta)}}
\arrow{ne,b}{\I_{\alpha - \beta}}
\end{diagram}.
}

Theorem~\ref{MainTheoremLorentzScale} may be sharpened further by replacing the Lebesgue space~$L_q$ on the left hand side by a smaller Lorentz space~$L_{q,1}$, smaller anisotropic homogeneous Besov\footnote{In the literature this space is usually denoted by~$B_{q,1}^0$; we prefer the notation~$B_{q}^{0,1}$ since Besov spaces are interpolation spaces with respect to the smoothness, not to the summability parameter. The use of Besov--Lorentz spaces~$B_{q,1}^{0,1}$ may serve as yet another justification of consistency of this notation.} space~$\dot{B}_{q}^{0,1}$, or with even narrower Besov--Lorentz space. 
Mixed-norm Lorentz spaces~$L^{\vec{q},\vec{r}}$, as well as Besov spaces built on them, have been studied in the literature; see, for instance,~\cite{Fernandez1977},~\cite{Mandel2023}, and~\cite{WHWY2024}. We do not pursue such refinements here. 
\begin{Th}\label{MainTheoremBesovLorentzScale}
Let~$\W$ be a closed translation invariant linear subspace of~$\Sw'(\R^d; \R^\ell)$. Assume~$\W$ is also invariant under the dilations~$\Dil_t$. If~$\W$ does not contain vectorial delta measures, then~$\I_\alpha$ maps~$\W \cap L_1$ to~$\dot{B}_{q}^{0,1}$,~$q = d/(d-\alpha)$, continuously, whenever~$\alpha \in (0,d)$.
\end{Th}
\begin{Cor}\label{LorentzCorollary}
Let~$\W$ be a closed translation invariant linear subspace of~$\Sw'(\R^d; \R^\ell)$. Assume~$\W$ is also invariant under the dilations~$\Dil_t$. If~$\W$ does not contain vectorial delta measures, then~$\I_\alpha$ maps~$\W \cap L_1$ to~$L_{q,1}$,~$q = d/(d-\alpha)$, continuously, whenever~$\alpha \in (0,d)$.
\end{Cor}

The definitions of anisotropic Besov-type spaces are quite natural, and we provide them and a toolkit for these spaces in Subsection~\ref{Besov--Lorentz} of the appendix.

\begin{proof}[Derivation of Corollary~\ref{LorentzCorollary} from Theorem~\ref{MainTheoremBesovLorentzScale}]
This is similar to derivation of Corollary~\ref{MixedSummabilityCorollary} from Theorem~\ref{MainTheoremLorentzScale}. Fix~$\gamma \in (0,\alpha)$ and set~$r = d/(d-\gamma)$. Then, Theorem~\ref{MainTheoremBesovLorentzScale} implies~$\I_\gamma \colon \W\cap L_1 \to \dot{B}_r^{0,1}$. According to Lemma~\ref{FromBesovToLorentzBesov}, this yields~$\I_{\alpha}$ maps~$\W \cap L_1$ to~$\dot{B}_{q,1}^{0,1}$, which, by~\eqref{BesovToLorentz}, is continuously embedded into~$L_{q,1}$.
\end{proof}
\begin{Rem}
Theorem~\ref{MainTheoremLorentzScale} is also true in the limiting case~$\alpha = d$ and~$q=\infty$, if we define the Riesz potential with~$\alpha = d$ properly. In the language of Besov spaces, we will obtain a slightly sharper inequality~$\W \cap L_1 \hookrightarrow \dot{B}_{\infty}^{-d,1}$. This statement is derived from Theorem~\ref{MainTheoremBesovLorentzScale} in the same way as Corollary~\ref{LorentzCorollary}, see Remark~\ref{LInftyRemark} in the appendix. Note, however, that the Calder\'on--Zygmund operators do not act continuously on the space~$L_\infty$. Therefore, one may replace~$\I_d$ with another operator of the same homogeneity, and the new inequality is not equivalent to the old one. One may raise the question about description of homogeneous of order~$-d$ operators that map a constrained space~$\W$ to~$L_{\infty}$. The answer to this question indeed depends on more delicate cancellation properties of the kernel and the space. It was given in~\cite{Stolyarov2021bibis} based on earlier work of Raita in~\cite{Raita2019}.
\end{Rem}

\subsection{Plan of the paper}\label{s13'}
The heat extension played an important role in~\cite{Stolyarov2022}. To adjust it to the anisotropic setting, we need to consider multiparametric heat extension where the time parameter is a vector in~$\R^d$ with positive coordinates rather than a scalar. In other words, we consider heat extension with respect to each of the one-dimensional variables~$x_1,x_2,\ldots,x_d$. Section~\ref{S2} contains the study of the multiparametric heat extension. We discuss uniqueness and simple monotonicity properties in Subsection~\ref{s21}. The main result is Proposition~\ref{MonotoniityFormula}, which provides a form of control of the growth of the~$L_p$ norm of the extension for positive functions and measures. The extensions generated by delta measures provide the maximal possible growth of the~$L_p$-norms among all non-negative measures.  We provide an elementary proof that is new even in the isotropic case. Subsection~\ref{s22} describes a convenient way to split the function~$f$ into pieces~$f_k$ in such a way that~$\|f\|_{L_1}$ splits in a controlled way. Here the heat extension is also useful, and the splitting is, in fact, a version of an anisotropic Littlewood--Paley decomposition. We also perform further splitting that decomposes the quantity~$\|f_k\|_{L_1}$ into a sum of suitably localized weighted norms. Each weight naturally corresponds to a pair~$(k,j)$ called atom; here~$k \geq 0$ and~$j\in \Z^d$.  Subsection~\ref{s23} contains the main definition of convex and flat atoms. Convex atoms are easier to deal with and we collect the estimates corresponding to them in Proposition~\ref{ConvexSummation}.

Section~\ref{S3} contains a strengthening of Proposition~\ref{MonotoniityFormula}, which says that if a measure is somehow separated from the collection of delta measures, then the~$L_p$-norm of its heat extension grows strictly slower than that of a delta measure. The rigorous form of this principle is formulated in Proposition~\ref{StrengtheningOfSemiinvariant}. A similar principle also played a pivotal role in~\cite{Stolyarov2022}. We manage to reduce the anisotropic case to the isotropic one via multiparametric heat extensions. This reduction is not immediate and occupies Subsection~\ref{s31}. Subsection~\ref{s32} is devoted to yet another elaboration of Proposition~\ref{StrengtheningOfSemiinvariant} that is more convenient for application to the functions~$f_k$ obtained in Subsection~\ref{s22}.

We present the main body of the proof of Theorem~\ref{MainTheoremBesovLorentzScale} in Section~\ref{S4}. Subsection~\ref{s41} contains a compactness argument that allows to pass from the assumption that~$f_k$ is a non-negative measure separated from the cone of delta measures in Proposition~\ref{StrengtheningOfSemiinvariant} to the condition that some atom~$(k,j)$ is flat. Here we also need a certain concentration assumption on the atom in question. The formal statement is given in Corollary~\ref{CubeEstimateCorollary}, which concludes a series of similar theorems and propositions. The combinatorial counterpart concludes the proof and  is presented in Subsection~\ref{s42}. In fact, it is quite similar to the one presented in~\cite{Stolyarov2022}, which, in its turn, models the argument for a related discrete problem from~\cite{ASW2021}. After the proof of the main theorem is finished, we provide a reflection and several suggestions for further research in Subsection~\ref{s43}.

The paper is supplemented with a large appendix that contains the proofs of technical statements, surveys folklore facts about Besov--Lorentz spaces, and provides several explanations why expected simplifications of the proof are impossible. We also provide a separate subsection in the appendix where we explain how Theorem~\ref{MainTheoremBesovLorentzScale} implies the already known and new inequalities for differential operators.


\section{Multiparametric heat extensions}\label{S2}

\subsection{Basic properties}\label{s21}
Let~$\vec{t} = (t_1,t_2,\ldots,t_d)$ be a vector with positive coordinates. Let~$f$ be a summable function of~$d$ variables. Define the function~$\Heat[f](\fdot;\vec{t}\,)\colon \R^d \to \R$ by the formula
\eq{
\Heat[f](x,\vec{t}) = \Big(\prod\limits_{j=1}^d(4\pi t_j)\Big)^{-\frac12}\int\limits_{\R^d}f(x-y)e^{-\sum_{1}^d\frac{y_j^2}{4t_j}}\,dy.
}
The operator~$f\mapsto \Heat[f](\fdot; \vec{t}\,)$ may be extended to~$f\in \Sw'(\R^d)$ in the usual way. We may also apply it coordinatewise to functions and distributions taking values in Euclidean spaces. We list simple properties of the constructed function without proof. 

The function~$\Heat[f]$ is a solution to the heat equation
\eq{\label{MultiparametricHeatEquations}
\frac{\partial \Heat[f]}{\partial t_j} = \frac{\partial^2 \Heat[f]}{\partial x_j^2}
}
for any~$j$. What is more,~$\Heat[f](x;\vec{t}\,)\to f(x)$ as~$\vec{t} \to 0$, provided~$f$ is continuous at~$x$. We call~$\Heat[f]$ the \emph{multiparametric heat extension} of~$f$. The multiparametric heat extension also satisfies the \emph{semigroup} property
\eq{\label{SemigroupProperty}
\Heat[f](x; \vec{t} + \vec{s}) = \Heat\big[\Heat[f](\fdot,\vec{t}\,)\big](x,\vec{s}),\qquad x\in \R^d.
}
The operator~$f\mapsto \Heat[f](\fdot;\vec{t}\,)$ is a Fourier multiplier:
\eq{
\mathcal{F}\big[\Heat[f](\fdot;\vec{t}\,)\big](\xi) = e^{-4\pi^2\sum_{1}^dt_j\xi^2_j}\hat{f}(\xi),\qquad \xi \in \R^d;
}
both symbols~$\mathcal{F}$ and~$\hat{\phantom{o}}$ denote the Fourier transform. The latter formula also allows to consider the case where some of~$t_j$ are equal to zero. 

The classical heat extension of a function or a distribution, that is,
\eq{
\heat[f](x,t) = (4\pi t)^{-\frac{d}{2}} \int\limits_{\R^d} f(x-y)e^{-\frac{|y|^2}{4t}}\,dy,\qquad x\in \R^d, t > 0,
}
may be restored from the multiparametric heat extension via the formula
\eq{\label{ClassicalHeat}
\heat[f](x,t) = \Heat[f](x,t,t,\ldots, t).
}
One may proceed in the reverse direction and construct~$\Heat[f]$ from~$\heat[f]$ since the latter extension defines~$f$. We will use this principle later.

We have two families of dilations,~\eqref{L1Dilation} and~\eqref{LinftyDilation}. It will be convenient to use yet another dilation operator:
\eq{\label{DILDef}
\DIL_\lambda [G](x,\vec{t}) = \lambda^{-d} G(\Dil_{\lambda^{-1}} x, \Dil_{\lambda^{-2}} \vec{t}\,),\qquad \lambda > 0,\ x\in\R^d,\ \vec{t} \in (\R_+)^d.
}
Note that it preserves neither the~$L_1$ nor the~$L_\infty$ norm of~$G$; it preserves the~$L_1$ norm of~$G$ in the~$x$ variable.
\begin{Le}\label{SeveralDilations}
For any function~$f\in \Sw'(\R^d)$ and any~$\lambda > 0$, we have
\eq{\label{SeveralDilationsFormula}
\Heat \big[\Dil_\lambda f\big] = \DIL_\lambda\big[\Heat [f]\big].
}
\end{Le}
See Subsection~\ref{AppendixA} of the appendix for the proof.

By a \emph{weight} we mean a non-negative locally summable function. A weight~$w$ defines the weighted Lebesgue space via formula 
\eq{
\|f\|_{L_q(w)} = \Big(\int\limits_{\R^d}|f(x)|^qw(x)\,dx\Big)^\frac{1}{q}.
}
Two lemmas below are given without proofs since they are direct generalizations of Lemmas~$1$ and~$2$ in~\cite{Stolyarov2022} (the reasonings work verbatim).
\begin{Le}\label{BasicMonotonicityLemma}
Let~$w$ be a weight, let~$g\in L_{1,\loc}\cap \Sw'(\R^d;\R^\ell)$, and let~$p \geq 1$. Then,
\eq{\label{BasicMonotonicity}
\big\|\Heat[g](\fdot;\vec{t}\,)\big\|_{L_p(w)}\leq \|g\|_{L_p(\Heat[w](\fdot;\vec{t}\,))},\qquad \vec{t} \in (\R_+)^d.
}
\end{Le}
\begin{Le}\label{ZeroFlatnessLemma}
Assume~$p=1$,~$t_j > 0$ for all~$j$, the inequality~\eqref{BasicMonotonicity} turns into equality with both sides being finite quantities, and~$w$ is almost everywhere positive. Then, there exists~$a\in \R^\ell$ and~$h \in \Sw'(\R^d)$,~$h \geq 0$, such that~$g = a\otimes h$.
\end{Le}
We also need the dilation properties of the weighted norms:
\eq{\label{eq217}
\big\|\Dil_\lambda[f]\big\|_{L_q(\Dil^\lambda[w])} = \lambda^{-\frac{q-1}{q}d}\|f\|_{L_q(w)} \Eeqref{DefOfq} \lambda^{-\alpha}\|f\|_{L_q(w)},\qquad \lambda > 0.
} 
Note that we apply dilations that preserve the values, not the integral, to the weight.

Let~$\one$ denote the vector~$(1,1,\ldots,1)$.
\begin{St}\label{MonotoniityFormula}
Let~$\mu$ be a measure, let~$w$ be a weight. Then,
\eq{\label{MonotoniityFormulaFormula}
\big\|\Heat[\mu](\fdot; \vec{t}\,)\big\|_{L_q(\Heat[w](\fdot; \frac{\one - \vec{t}}{q}))}\leq \Big(\prod\limits_{j=1}^dt_j\Big)^{-\frac{q-1}{2q}}\big\|\Heat[\mu](\fdot;\one\,)\big\|_{L_q(w)},\qquad \forall j\quad t_j\in [0,1],
}
provided the quantity on the right hand side is finite.
\end{St}
The isotropic version of this proposition was justified in~\cite{Stolyarov2022} by a tricky method borrowed from~\cite{BCT2006}. Though that method will still be needed to prove a strengthening of Proposition~\ref{MonotoniityFormula}, Proposition~\ref{StrengtheningOfSemiinvariant} below, we prefer to provide an elementary proof as well. Here it is.
\begin{proof}[Proof of Proposition~\ref{MonotoniityFormula}]
Without loss of generality, we may assume~$\mu$ is a finite measure with compact support. We raise the inequality to the power~$q$:
\eq{\label{eq218}
\int\limits_{\R^d}\Big(\Heat[\mu](x; \vec{t}\,)\Big)^q\Heat[w]\Big(x; \frac{\one - \vec{t}}{q}\Big)\,dx \leq \Big(\prod\limits_{j=1}^dt_j\Big)^{-\frac{q-1}{2}}\int\limits_{\R^d}\Big(\Heat[\mu](x;\one\,)\Big)^q w(x)\,dx.
}
This inequality is linear with respect to~$w$. Thus, it suffices to test it against~$w = \delta_{y}$ for some~$y\in \R^d$; by translation invariance, we may assume~$y=0$ without loss of generality\footnote{To formalize this principle, we may go backwards. If~\eqref{eq2111} is true, then 
\eq{
\int\limits_{\R^d}\Big(\Heat[\mu](x; \vec{t}\,)\Big)^q \Big(\frac{4\pi}{q}\Big)^{-\frac{d}{2}}\Big(\prod\limits_{j=1}^d(1-t_j)\Big)^{-\frac12} e^{-q\sum_{1}^d\frac{(x_j-y_j)^2}{4(1-t_j)}}\,dx \leq \Big(\prod\limits_{j=1}^dt_j\Big)^{-\frac{q-1}{2}} \Big(\Heat[\mu](y;\one\,)\Big)^q
}
is also true. If we multiply this inequality by~$w(y)$ and integrate with respect to~$y$, we obtain~\eqref{eq218}. A similar deduction of~\eqref{eq2111} from~\eqref{eq2114} involves Minkowski's inequality as an additional ingredient.}. In such a case,
\eq{
\Heat[w]\Big(x; \frac{\one - \vec{t}}{q}\Big) = \Big(\frac{4\pi}{q}\Big)^{-\frac{d}{2}}\Big(\prod\limits_{j=1}^d(1-t_j)\Big)^{-\frac12} e^{-q\sum_{1}^d\frac{x_j^2}{4(1-t_j)}},\qquad w = \delta_y,
}
and we arrive at
\eq{
\int\limits_{\R^d}\Big(\Heat[\mu](x; \vec{t}\,)\Big)^q \Big(\frac{4\pi}{q}\Big)^{-\frac{d}{2}}\Big(\prod\limits_{j=1}^d(1-t_j)\Big)^{-\frac12} e^{-q\sum_{1}^d\frac{x_j^2}{4(1-t_j)}}\,dx \leq \Big(\prod\limits_{j=1}^dt_j\Big)^{-\frac{q-1}{2}} \Big(\Heat[\mu](0;\one\,)\Big)^q,
}
which may be rewritten as
\mlt{\label{eq2111}
\bigg(\int\limits_{\R^d}\Big(\Heat[\mu](x; \vec{t}\,)\Big)^q \Big(\frac{4\pi}{q}\Big)^{-\frac{d}{2}}\Big(\prod\limits_{j=1}^d(1-t_j)\Big)^{-\frac12} e^{-q\sum_{1}^d\frac{x_j^2}{4(1-t_j)}}\,dx\bigg)^{1/q}\\
 \leq \Big(\prod\limits_{j=1}^dt_j\Big)^{-\frac{q-1}{2q}} (4\pi)^{-\frac{d}{2}}\int\limits_{\R^d} e^{-\frac{|x|^2}{4}}\,d\mu(x).
}
This is a bound of a convex functional of~$\mu$ with a linear functional of~$\mu$. For such estimates on the cone of measures, we may restrict our attention to the case~$\mu = \delta_{z}$ for some~$z \in \R^d$. In this case,
\eq{
\Big(\Heat[\mu](x; \vec{t}\,)\Big)^q = (4\pi)^{-\frac{dq}{2}}\Big(\prod\limits_{j=1}^dt_j\Big)^{-\frac{q}{2}} e^{-\sum_{1}^d \frac{q|x_j-z_j|^2}{4t_j}},\qquad \int\limits_{\R^d} e^{-\frac{|x|^2}{4}}\,d\mu(x) = e^{-\frac{|z|^2}{4}},\quad \mu = \delta_z,
}
and we arrive at
\mlt{\label{eq2114}
\bigg(\int\limits_{\R^d}(4\pi)^{-\frac{dq}{2}}\Big(\prod\limits_{j=1}^dt_j\Big)^{-\frac{q}{2}} e^{-\sum_{1}^d \frac{q|x_j-z_j|^2}{4t_j}} \Big(\frac{4\pi}{q}\Big)^{-\frac{d}{2}}\Big(\prod\limits_{j=1}^d(1-t_j)\Big)^{-\frac12} e^{-q\sum_{1}^d\frac{x_j^2}{4(1-t_j)}}\,dx\bigg)^{1/q}\\ \leq  \Big(\prod\limits_{j=1}^dt_j\Big)^{-\frac{q-1}{2q}} (4\pi)^{-\frac{d}{2}} e^{-\frac{|z|^2}{4}}.
}
We will shortly show this inequality is, in fact, an identity. Note that the variables separate, and it suffices to establish a one-dimensional identity
\eq{
(4\pi)^{-\frac12}q^{\frac12}t^{-\frac12}(1-t)^{-\frac12} \int\limits_{\R} e^{-\frac{q}{4}(\frac{(x-z)^2}{t} + \frac{x^2}{1-t} - z^2)}\,dx = 1,\qquad z\in\R.
}
This follows from the fact
\eq{
\frac{(x-z)^2}{t} + \frac{x^2}{1-t} - z^2 = \frac{1}{t(1-t)}(x-(1-t)z)^2.
}
\end{proof}
\begin{Rem}\label{MeasureRemark}
The proof says we may slightly generalize Proposition~\ref{MonotoniityFormula} and assume~$w$ is a measure. What is crucial is that~$\mu$ and~$w$ are non-negative. Note that we do not postulate any sort of finiteness of these measures, the only condition is that the right hand side of~\eqref{MonotoniityFormulaFormula} is finite.
\end{Rem}
Proposition~\ref{MonotoniityFormula} has a useful reformulation, which is merely a translation into the PDE language. This reformulation is based upon a representation formula for the multilinear heat equation.
\begin{Le}\label{WidderLemma}
Let~$u\colon \R^{d}\times [0,T]^d\to \R$ be a non-negative solution to the multiparametric heat equation~\eqref{MultiparametricHeatEquations}. Then, there exists a measure~$\mu$ on~$\R^d$ such that
\eq{\label{eq2120}
u(x,\vec{t}\,) = \Big(\prod\limits_{j=1}^d(4\pi t_j)\Big)^{-\frac12}\int\limits_{\R^d}e^{-\sum_{1}^d\frac{(x_j -y_j)^2}{4t_j}}\,d\mu(y)
}
and these integrals converge for all~$x\in \R^d$ and~$t_j\in (0,T]$. The integral above defines the solution to~\eqref{MultiparametricHeatEquations} as long as it converges for all~$(x,\vec{t}\,) \in \R^d \times (0,T]^d$.
\end{Le}
\begin{proof}
The case~$d=1$ of the ordinary heat equation was obtained by Widder in~\cite{Widder1944}; the case of the classical heat equation and arbitrary dimension is completely similar (we may formally cite~\cite{Aronson1968} where the case of a general parabolic equation is considered). We omit the proof of the second part of the lemma since this is an exercise in standard calculus techniques. 

To prove the first part, consider the function~$U\colon \R^d \times [0,T]\to \R$ defined by~$U(x,t) = u(x,t,t,\ldots, t)$. Then,~$U$ is the solution to the ordinary heat equation on its domain; thus, by Widder's theorem, there exists a measure~$\mu$ on~$\R^d$ such that
\eq{
U(x,t) = \big(4\pi t\big)^{-\frac{d}{2}}\int\limits_{\R^d} e^{-\frac{|x-y|^2}{4t}}\,d\mu(y);\qquad x\in \R^d, t\in (0,T].
}
The integrals~\eqref{eq2120} automatically converge and define a solution~$u'$ to~\eqref{MultiparametricHeatEquations}; this is a standard calculus exercise again. What remains to justify is the coincidence of~$u$ and~$u'$. Fix some~$\vec{t} \in (0,T]^d$, without loss of generality, assume~$t_1 \leq t_2 \leq t_3 \leq \ldots \leq t_d$. We know~$u(\fdot, t_1,t_1,\ldots, t_1) = u'(\fdot, t_1,t_1,\ldots, t_1)$ for any~$t_1 \in [0,T]$. By the uniqueness result for positive solutions to the heat equation (Theorem~$5$ in~\cite{Widder1944}),~$u(\fdot, t_1,t_1,\ldots,t_1, t_d) = u'(\fdot, t_1,t_1,\ldots, t_1,t_d)$. Reasoning in the same manner for the other coordinates,  we obtain the desired coincidence of~$u$ and~$u'$.
\end{proof}
\begin{Cor}\label{PDECorollary}
Let~$u\colon \R^d\times [0,1]^d\to \R$ be a non-negative solution to~\eqref{MultiparametricHeatEquations}. Let~$v\colon \R^d\times [0,1]^d\to \R$ be a non-negative solution to 
\eq{\label{eq2115}
-q\frac{\partial v}{\partial t_j} = \frac{\partial^2 v}{\partial x_j^2},\qquad j\in [1\twodots d].
}
Then,
\eq{
\int\limits_{\R^d}u^q(x,\vec{t}\,)v(x,\vec{t}\,)\,dx \leq \Big(\prod\limits_{j=1}^d t_j\Big)^{-\frac{q-1}{2}}\int\limits_{\R^d}u^q(x,\one\,)v(x,\one\,)\,dx
}
for any~$\vec{t}\in [0,1]^d$.
\end{Cor}

\begin{proof}
By Lemma~\ref{WidderLemma}, there exist measures~$\mu$ and~$w$ such that
\alg{
u(x,\vec{t}\,) &= \Heat[\mu](\fdot; \vec{t}\,),& \quad &x\in \R^d, \vec{t}\in [0,1]^d;&\\
v(x,\vec{t}\,) &= \Heat[w]\Big(\fdot; \frac{\one - \vec{t}}{q}\Big),&\quad &x\in \R^d, \vec{t} \in [0,1]^d,&
}
and, in the light of Remark~\ref{MeasureRemark}, the corollary reduces to Proposition~\ref{MonotoniityFormula}.
\end{proof}

Using the dilations~$(x,\theta)\mapsto (\sqrt{s} x, s \theta)$,~$x\in\R$,~$\theta>0$, in each of the coordinates, we obtain a slightly more general version.
\begin{St}\label{RescaledPDECorollary}
Let~$\vec{s}\in (\R_+)^d$. Let~$u\colon \R^d\times \prod_{j}[0,s_j]\to \R$ be a non-negative solution to~\eqref{MultiparametricHeatEquations}. Let~$v$ be a non-negative solution to~\eqref{eq2115} on~$\R^d\times \prod_{j}[0,s_j]$. Then,
\eq{
\int\limits_{\R^d}u^q(x,\vec{t}\,)v(x,\vec{t}\,)\,dx \leq \Big(\prod\limits_{j=1}^d \frac{s_j}{t_j}\Big)^{\frac{q-1}{2}}\int\limits_{\R^d}u^q(x,\vec{s}\,)v(x,\vec{s}\,)\,dx,
}
whenever~$t_j \in (0,s_j]$ for every~$j = 1,2,\ldots, d$.
\end{St}
\begin{proof}
Define the functions~$\tilde{u}$ and~$\tilde{v}$ on the domain~$\R^d \times [0,1]^d$:
\alg{
\tilde{u}(x,\vec{\theta}) &= u\Big(\sqrt{s_1}x_1,\sqrt{s_2}x_2,\ldots, \sqrt{s_d}x_d,s_1\theta_1, s_2\theta_2,\ldots,s_d\theta_d\Big);\\
\tilde{v}(x,\vec{\theta}) &= v\Big(\sqrt{s_1}x_1,\sqrt{s_2}x_2,\ldots, \sqrt{s_d}x_d,s_1\theta_1, s_2\theta_2,\ldots,s_d\theta_d\Big).
}
By dilation invariance, they solve the same partial differential equations as~$u$ and~$v$ do. We apply Corollary~\ref{PDECorollary} to them and obtain the desired result by choosing~$\theta_j = t_j/s_j$ for every~$j$.
\end{proof}


\subsection{Anisotropic Littlewood--Paley decomposition}\label{s22}
Pick some large number~$A$. There will be further clarifications what we mean by `large', for now we assume that at least~$A > 2$. Let~$f$ be a summable function on~$\R^d$. Consider the functions
\eq{
f_k = \Heat[f]\big(\fdot; A^{-2ka_1}, A^{-2ka_2},\ldots, A^{-2ka_d}\big),\qquad k\in \Z.
}

\begin{Cor}\label{DilationPartsCorollary}
For any~$k, m \in \mathbb{Z}$, we have
\eq{
\Big(\Dil_{A^m}[f]\Big)_k = \Dil_{A^m}[f_{k+m}].
}
\end{Cor}
\begin{proof}
Pick some~$x\in \R^d$ and compute
\mlt{
\big(\Dil_{A^m}[f]\big)_k(x)\\ 
= \Heat\big[\Dil_{A^m} f\big]\big(x; A^{-2ka_1}, A^{-2ka_2},\ldots, A^{-2ka_d}\big)
 \Eref{\text{\tiny Lem.~\ref{SeveralDilations}}} \DIL_{A^m}\big[\Heat[f]\big] \big(x; A^{-2ka_1}, A^{-2ka_2},\ldots, A^{-2ka_d}\big)\\
 = A^{-md} \Heat[f]\big(\Dil_{A^{-m}}x; A^{-2(k+m)a_1}, A^{-2(k+m)a_2},\ldots, A^{-2(k+m)a_d}\big) = \Dil_{A^m}[f_{k+m}](x).
}
\end{proof}

The functions~$f_k$ are convenient for expressing the anisotropic Riesz potential defined in~\eqref{RieszPotentialDefinition}.
\begin{Le}\label{TriangleInequalityLemma}
For any~$q \in (1,\infty)$ and~$\alpha\in(0,d)$, the inequality
\eq{\label{RieszPotentialDiscretizationFormula}
\|\I_\alpha f\|_{L_{q}}\lesssim \sum\limits_{k \in \Z}A^{-\alpha k}\|f_k\|_{L_{q}}
}
holds true with a constant independent of~$f$.
\end{Le}
\begin{proof}
Let~$\psi$ be a Schwartz function whose Fourier transform is compactly supported and is equal to~$1$ in a neighborhood of the origin. Define the functions~$\psi_k$ by the rule
\eq{
\psi_k(x) = \Dil_{A^{-k}}[\psi],\qquad k\in\mathbb{Z},\ x\in \R^d.
}
Let us first prove the inequality
\eq{\label{eq224}
\|g\|_{L_{q}} \lesssim \sum\limits_{k\in\Z}\big\|g*(\psi_k - \psi_{k-1})\big\|_{L_{q}}.
}
For that, we recall the limit relations 
\eq{\label{LimitRelationsLorentz}
g*\psi_k \longrightarrow g \quad \text{in}\ L_{q},\quad k\to \infty;\qquad\qquad g*\psi_k \longrightarrow 0 \quad \text{in}\ L_{q},\quad k\to -\infty,
}
leading to the representation
\eq{
g= \sum_k g*(\psi_k - \psi_{k-1}),\qquad \text{the series converges in~$L_{q}$,}
}
which, in its turn, implies~\eqref{eq224} via the triangle inequality 
in~$L_{q}$. 

Thus, it remains to show
\eq{\label{IndividualBound}
\big\|\I_\alpha [f]*(\psi_k - \psi_{k-1})\big\|_{L_{q}} \lesssim A^{-\alpha k}\|f_k\|_{L_{q}}
}
for any~$k\in \Z$ with a uniform constant. We see that, by using dilations (namely, we rely upon Corollary~\ref{DilationPartsCorollary},~\eqref{RieszPotentialHomogeneity}, and~\eqref{eq217}), this inequality reduces to the case~$k=0$. The function~$\I_\alpha [f]*(\psi_0 - \psi_{-1})$ is obtained from~$f_0$ by application of the Fourier multiplier with the symbol
\eq{
e^{4\pi^2|\xi|^2}\frac{\hat{\psi}(\xi) - \Dil^{A^{-1}}[\hat{\psi}](\xi)}{(\rho(\xi))^{\alpha}},\qquad \xi \in \R^d,
}
see formula~\eqref{DilationFourier}. This symbol is a compactly supported smooth function, and therefore, the kernel of the Fourier multiplier in question is a summable function, which yields~\eqref{IndividualBound}.
\end{proof}
\begin{Rem}
While the multiplicative constant in~\eqref{RieszPotentialDiscretizationFormula} is independent of~$f$, it might depend on~$A$.
\end{Rem}
\begin{Rem}\label{DiscretizationRemark}
The bound~\eqref{IndividualBound} leads to the inequality
\eq{
\|\I_\alpha f\|_{\dot{B}_{q}^{0,1}} \lesssim \sum\limits_{k \in \Z}A^{-\alpha k}\|f_k\|_{L_{q}},
}
see~\eqref{DefOfBesovNorm}. In particular, both Theorems~\ref{MainTheoremLorentzScale} and~\ref{MainTheoremBesovLorentzScale} reduce to the bound
\eq{
\sum\limits_{k \in \Z}A^{-\alpha k}\|f_k\|_{L_{q}} \lesssim \|f\|_{L_1},\qquad f \in \W,
}
provided~$\W$ meets the requirements of those theorems; recall that the parameters satisfy~\eqref{DefOfq}.
\end{Rem}

Since we will be using induction on scales in our proof, it is convenient to have some basic scale. The following lemma provides us with such a scale.
\begin{Le}\label{TruncationLemma}
Let~$\W$ be a closed translation invariant linear subspace of~$\Sw'(\R^d; \R^\ell)$. Assume~$\W$ is also invariant under the dilations~$\Dil_t$. If the estimate
\eq{\label{TruncatedSumEstimate}
 \sum\limits_{k = 0}^\infty A^{-\alpha k}\|f_k\|_{L_{q}} \lesssim \|f\|_{L_1},\qquad f \in \W,
}
holds true for all~$f\in \W \cap L_1$ with a uniform constant, then the inequality
\eq{
\sum\limits_{k \in \Z}A^{-\alpha k}\|f_k\|_{L_{q}} \lesssim \|f\|_{L_1},\qquad f \in \W,
}
is also true.
\end{Le}
\begin{proof}
Assume~\eqref{TruncatedSumEstimate} holds true. Then, given any~$N \in \mathbb{N}$, the estimate
\eq{\label{TruncationN}
 \sum\limits_{k = -N}^\infty A^{-\alpha k}\|f_k\|_{L_{q}} \lesssim \|f\|_{L_1},\qquad f \in \W,
}
holds as well. To show this, we plug~$\tilde{f} = \Dil_{A^{-N}}[f]$ instead of~$f$ into~\eqref{TruncatedSumEstimate}. The~$L_1$ norms on the right hand side are the same. The quantities on the left hand side are also the same:
\mlt{
 \sum\limits_{k = 0}^\infty A^{-\alpha k}\|\tilde{f}_k\|_{L_{q}} =  \sum\limits_{k = 0}^\infty A^{-\alpha k}\Big\|\big(\Dil_{A^{-N}}f\big)_k\Big\|_{L_{q}} \\ 
 \Eref{\text{\tiny Cor.~\ref{DilationPartsCorollary}}} \sum\limits_{k = 0}^\infty A^{-\alpha k}\Big\|\Dil_{A^{-N}}\big[f_{k-N}\big]\Big\|_{L_{q}} \Eeqref{eq217} \sum\limits_{k=0}^\infty A^{-\alpha(k-N)} \|f_{k-N}\|_{L_q} =  \sum\limits_{k = -N}^\infty A^{-\alpha k}\|f_k\|_{L_{q}}.
} 
Thus, we have obtained~\eqref{TruncationN} with the multiplicative constant independent of~$N$ (it is the same as in~\eqref{TruncatedSumEstimate}). The desired bound follows by passing to the limit as~$N\to\infty$.
\end{proof}

Now we wish to link the functions~$f_k$ to the~$L_1$ norm of~$f$ more directly. We see that~$f_0$ is a smooth function and that~$f_k \to f$ in~$L_1(\R^d, \R^\ell)$ as~$k\to \infty$. Note that by~\eqref{SemigroupProperty}
\eq{\label{ReproducingThefk}
f_k = \Heat[f_m]\big(\fdot; A^{-2ka_1} - A^{-2ma_1}, A^{-2ka_2}-A^{-2ma_2},\ldots, A^{-2ka_d}-A^{-2ma_d}\big),\qquad m \geq k.
}
This implies via Lemma~\ref{BasicMonotonicityLemma} with constant weight that
\eq{
\|f_k\|_{L_1} \leq \|f_m\|_{L_1},\qquad m \geq k.
}
Therefore, we may represent
\eq{\label{eq2215}
\|f\|_{L_1} = \|f_0\|_{L_1} + \sum\limits_{k=0}^\infty\Big(\|f_{k+1}\|_{L_1} - \|f_k\|_{L_1}\Big),
}
and each term in the series is non-negative. Introduce the technical parameters
\eq{\label{DefinitionOfK}
K = \big\lceil \max_{i,j} \frac{a_i}{a_j}\big\rceil+1, \quad L = K+1;
}
here we use the notation~$\lceil x \rceil$ for the ceiling of a real number~$x$, which is the smallest possible integer number that is larger than or equal to~$x$. In the classical isotropic case~$a_i=1$ we have~$K=2$ and~$L=3$.

For technical purposes, we will use the inequality
\eq{\label{Telescopic}
\sum\limits_{k=0}^\infty\Big(\|f_{k+L}\|_{L_1} - \|f_k\|_{L_1}\Big)\leq L \|f\|_{L_1}
}
instead of~\eqref{eq2215}. Note that this inequality also holds true for vector-valued functions~$f$. We need to decompose the quantities~$\|f_{k+L}\|_{L_1} - \|f_k\|_{L_1}$ further. An informal principle says that the function~$f_k$ behaves like a function on the lattice~$\prod_{j=1}^d \big(A^{-ka_j}\Z\big)$. Let~$w$ be a weight such that
\eq{\label{SumOneForWeight}
\sum_{j\in \Z^d} w(x-j) = 1\qquad \text{for any}\quad x\in \R^d.
} 
There will be further requirements on~$w$, for now we assume it is smooth and satisfies the bound
\eq{\label{eq2216}
w(x) \geq C (1+|x|)^{-\theta}
}
for some~$\theta > d$ and~$C > 0$. The reader may look up the formula for~$w$ in~\eqref{OurWeights} below for the final choice of~$w$; before that choice we prefer to reason for more general weights.
Set 
\eq{\label{DefOfWeights}
w_{k,j}(x) = w\Big( \Dil_{A^k}x- j\Big),\quad j\in \Z^d, k \in \Z.
}
In other words,
\eq{
w_{0,j}(x) = w(x-j),\qquad j \in \Z^d, \quad \text{and}\quad w_{k,j} = \Dil^{A^{-k}}[w_{0,j}],\qquad k\in \Z.
}
Note that~$\sum_jw_{k,j} = 1$ for any~$k$. Then,
\mlt{\label{DefOfTildeWeights}
\|f_{k+L}\|_{L_1} - \|f_k\|_{L_1} = \sum\limits_{j\in\Z^d}\bigg(\|f_{k+L}\|_{L_1(\tilde{w}_{k,j})} - \|f_k\|_{L_1(w_{k,j})}\bigg),\quad \text{where}\\
\tilde{w}_{k,j} = \Heat [w_{k,j}](\fdot; A^{-2ka_1} - A^{-2(k+L)a_1}, A^{-2ka_2}-A^{-2(k+L)a_2},\ldots, A^{-2ka_d}-A^{-2(k+L)a_d}),
}
since~$\sum_j \tilde{w}_{k,j} = 1$ as well. By Lemma~\ref{BasicMonotonicityLemma} and~\eqref{ReproducingThefk}, each summand in this sum is non-negative.


\subsection{Convex and flat atoms}\label{s23}
An atom is a pair~$(k,j)$,~$k\geq 0$ and~$j\in \Z^d$. Each atom has a parallelepiped
\eq{
Q_{k,j} = \Set{x\in \R^d}{\big|\Dil_{A^k} x - j\big|_{\ell_{\infty}^d} \leq 1/2},\qquad |y|_{\ell_{\infty}^d} = \sup_i |y_i|, \ y= (y_1,y_2,\ldots, y_d)\in \R^d,
}
associated with it. In the classical isotropic case~$a=(1,1,\ldots,1)$ these parallelepipeds are, in fact, cubes. If~$A$ is an odd integer, any two cubes are either disjoint up to a  set of measure zero or one contains the other. This defines a tree-like structure on the set of these cubes ($A$-adic cubes) in a natural way: We join the two cubes with side lengths~$A^{-k}$ and~$A^{-k-1}$ by an edge if the former contains the latter. In the general anisotropic setting, this might not be the case: Though the parallelepipeds~$Q_{k,j}$ tile~$\R^d$ when~$k$ is fixed, in general, they do not form any tree-like structure; see Fig.~\ref{Figure1} for an example. If~$a$ has rational coordinates, one may choose~$A$ such that the collection of~$A$-adic parallelepipeds form a tree according to the aforementioned principle, see Fig.~\ref{Figure1} again.

\begin{figure}[h!]
\centerline{
\includegraphics[height=7cm]{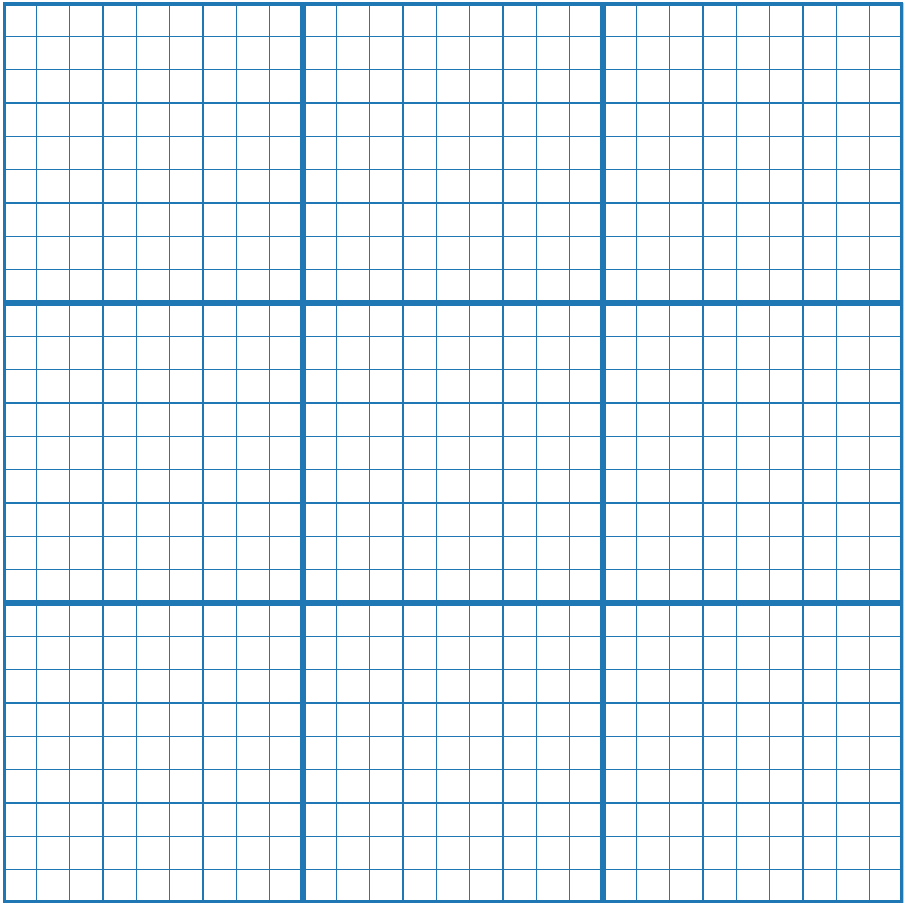}\hspace{30pt}
\includegraphics[height=7cm]{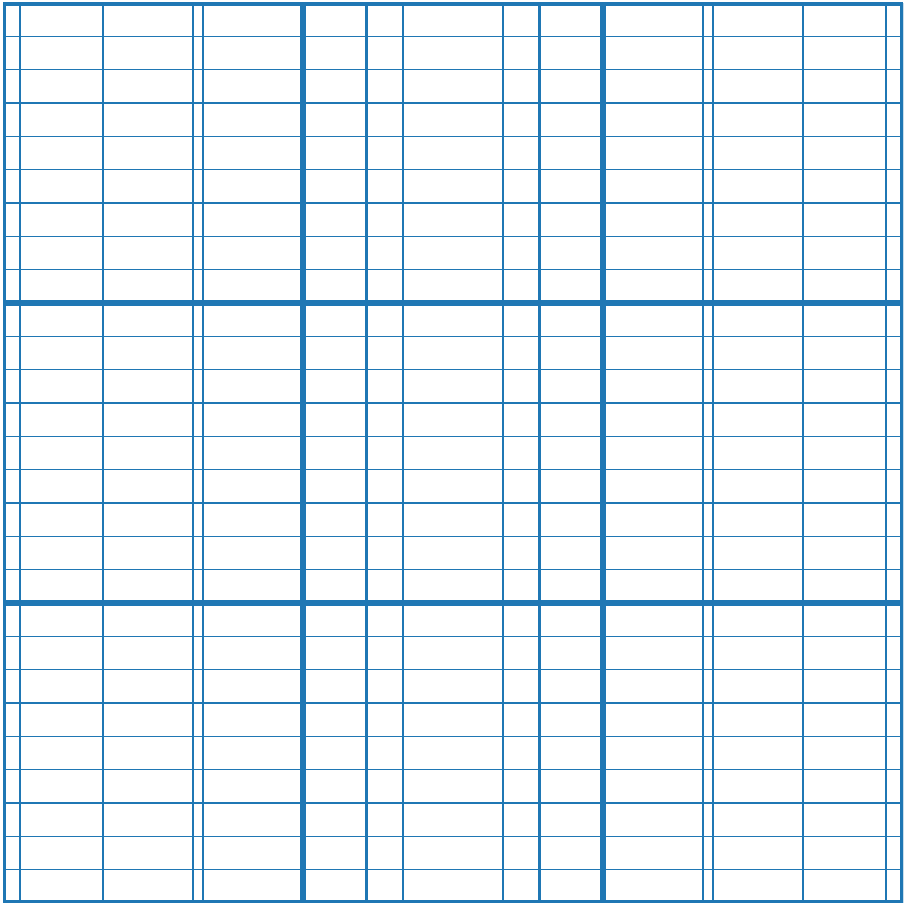}}
\caption{Classical~$3$-adic squares,~$a = (1,1)$,~$A = 3$, and anisotropic rectangles,~$a = (2/3, 4/3)$,~$A = 3^{3/4}$; each rectangle~$Q_{k,j}$ is tiled by the parallelepipeds of~$(k+2)$th generation in this case. If we choose~$A = 3^{3/2}$ for the second anisotropy, the~$A$-adic parallelepipeds form a tree-like structure.}
\label{Figure1}
\end{figure}

We will still need the tree structure and define it in the most natural way. 
\begin{Def}\label{TreeStructure}
Let~$(k,j)$ be an atom and let~$k \geq 1$. If~$Q_{k,j}$ is contained in some~$Q_{(k-1),i}$, then~$(k-1,i)$ is the parent of~$Q_{k,j}$. If this is not the case, we choose some~$Q_{k-1,i}$ that intersects~$Q_{k,j}$ to be the parent of the latter atom. Denote the obtained graph by~$\GenTree$.
\end{Def}
The next definition goes back to~\cite{ASW2021} and lies at the heart of the method.
\begin{Def}\label{Convex/flat}
Fix~$\eps \in (0,1/2)$. An atom~$(k,j)$,~$k \in \N \cup \{0\}, j\in \Z^d$, is called~$\eps$-convex, provided
\eq{
\|f_{k+L}\|_{L_1(\tilde{w}_{k,j})} \geq (1+\eps) \|f_k\|_{L_1(w_{k,j})},\qquad \text{where}\ \tilde{w}_{k,j}\ \text{is defined in~\eqref{DefOfTildeWeights}}.
}
Otherwise, the pair~$(k,j)$ is an~$\eps$-flat atom. The set of convex atoms is denoted by~$\CO$ and the set of flat atoms is~$\Fl$.
\end{Def}
Here~$\eps$ is a small number to be chosen later. In all our reasonings it is fixed. Convex atoms are easier to deal with, and the estimate for them does not require any constraint on~$f$. 
\begin{St}\label{ConvexSummation}
For any~$A > 2$ and~$\eps \in (0,\frac12)$, the estimate
\eq{
\sum\limits_{(k,j)\in\CO}A^{-\alpha k}\|f_k\|_{L_{q}(Q_{k,j})} \lesssim \|f\|_{L_1}
}
holds true uniformly in~$f$.
\end{St}
Of course, the constants in the inequality may depend on~$A$ and~$\eps$. Recall the relation~\eqref{DefOfq} on~$q$ and~$\alpha$. The remaining part of the subsection is occupied by the proof of Proposition~\ref{ConvexSummation}. At the very end we provide further explanations and a plan for the proof of Theorems~\ref{MainTheoremLorentzScale} and~\ref{MainTheoremBesovLorentzScale}.

We start the proof of Proposition~\ref{ConvexSummation} with three weighted lemmas, which will be also useful further.
The proofs are placed in Subsection~\ref{SWeights} of the appendix since they are standard and resemble the corresponding proofs in~\cite{Stolyarov2022}. We start with two lemmas that postulate the principle `if a weight is smooth, heating does not change it much'. 
\begin{Le}\label{Lemma41}
Assume the weight~$G$ satisfies the bound
\eq{
G(x) \leq C (1+|x|)^{-\theta},\qquad x\in \R^d.
}
Then, 
\eq{
\Heat[G](x,\vec{t}\,) \leq \tilde{C}(1+|x|)^{-\theta},\qquad \forall j\in [1\twodots d] \quad t_j \in [0,2],\  x\in \R^d,
}
and the constant~$\tilde{C}$ depends on~$d, \theta$ and~$C$ only. 
\end{Le}
\begin{Le}\label{Lemma42}
Assume the weight~$G$ satisfies the bound
\eq{\label{LowerPowerBound}
G(x) \geq c (1+|x|)^{-\theta},\qquad x\in \R^d.
}
Then, 
\eq{
\Heat[G](x,\vec{t}\,) \geq \tilde{c}(1+|x|)^{-\theta},\qquad  \forall j\in [1\twodots d] \quad t_j \in [0,2],\ x\in \R^d,
}
and the constant~$\tilde{c}$ depends on~$d, \theta$ and~$c$ only. 
\end{Le}

The next lemma provides us with the basic~$L_1 \to L_q$ bound. 
\begin{Le}\label{Lemma44}
Let~$u$ and~$v$ be two weights that satisfy the bounds
\eq{\label{ConditionsOnuv}
v(x) \leq C_v(1+|x|)^{-\theta_v};\qquad u(x) \geq c_u(1+|x|)^{-\theta_u}.
}
Assume also
\eq{\label{ThetavThetau}
\theta_v \geq q\theta_u.
} 
Then,
\eq{\label{eq416}
\big\|\Heat[f](\fdot, \vec{t}\,)\big\|_{L_q(v)} \lesssim \|f\|_{L_1(u)}, \qquad \forall j \in [1\twodots d] \quad t_j\in \Big[\frac12,2\Big].
}
\end{Le}
The three lemmas have rescaled versions stated in three corollaries below. The derivations are straightforward applications of Lemma~\ref{SeveralDilations} and~\eqref{eq217}.
\begin{Cor}\label{Lemma41Cor}
Let~$k  \geq 0$,~$y\in \R^d$. Assume the weight~$G$ satisfies the bound
\eq{
G(x) \leq C (1+ |\Dil_{A^k}x-y|)^{-\theta},\qquad x\in \R^d.
}
Then, 
\eq{
\Heat[G](x,\vec{t}\,) \leq \tilde{C}(1+|\Dil_{A^k}x-y|)^{-\theta},\qquad \forall j\in [1\twodots d] \quad t_j \in [0,2A^{-2ka_j}],\  x\in \R^d,
}
and the constant~$\tilde{C}$ depends on~$d, \theta$ and~$C$ only. 
\end{Cor}
\begin{Cor}\label{Lemma42Cor}
Let~$k  \geq 0$,~$y\in \R^d$. Assume the weight~$G$ satisfies the bound
\eq{
G(x) \geq c (1+|\Dil_{A^k}x-y|)^{-\theta},\qquad x\in \R^d.
}
Then, 
\eq{
\Heat[G](x,\vec{t}\,) \geq \tilde{c}(1+|\Dil_{A^k}x-y|)^{-\theta},\qquad  \forall j\in [1\twodots d] \quad t_j \in [0,2A^{-2ka_j}],\ x\in \R^d,
}
and the constant~$\tilde{c}$ depends on~$d, \theta$ and~$c$ only. 
\end{Cor}
\begin{Cor}\label{Cor44}
Let~$k  \geq 0$,~$y\in \R^d$.  Assume~$u$ and~$v$ are weights that satisfy
\eq{
v(x) \leq C_v(1+|\Dil_{A^k}x - y|)^{-\theta_v};\qquad u(x) \geq c_u(1+|\Dil_{A^k}x - y|)^{-\theta_u}.
}
Assume also~\eqref{ThetavThetau}. Then, 
\eq{\label{eq417}
A^{-\alpha k}\big\|\Heat[f](\fdot, \vec{t}\,)\big\|_{L_q(v)} \lesssim \|f\|_{L_1(u)}, \qquad \forall j \in [1\twodots d] \quad t_j \in \Big[\frac{A^{-2ka_j}}{2},2A^{-2ka_j}\Big],
}
where~$\alpha$ is defined by the usual homogeneity relation~\eqref{DefOfq}.
\end{Cor}

\begin{proof}[Proof of Proposition~\ref{ConvexSummation}]
Recall
\eq{
\tilde{w}_{k,j} = \Heat [w_{k,j}](\fdot; A^{-2ka_1} - A^{-2(k+L)a_1}, A^{-2ka_2}-A^{-2(k+L)a_2},\ldots, A^{-2ka_d}-A^{-2(k+L)a_d}).
}
Note that~$a_j L  > 1$ for any~$j$, which implies
\eq{
A^{-2ka_j} - A^{-2(k+L)a_j} \in \Big[\frac{A^{-2ka_j}}{2},2A^{-2ka_j}\Big],\qquad \text{for any}\ j \in [1\twodots d].
}
By Corollary~\ref{Lemma42Cor} and the assumption~\eqref{eq2216},
\eq{
\tilde{w}_{k,j}(x) \gtrsim (1+|\Dil_{A^k} x - j|)^{-\theta}.
}
Therefore, we may choose~$y = j$ and~$u = \tilde{w}_{k,j}$ in Corollary~\ref{Cor44}. Choosing~$\chi_{Q_{k,j}}$ as~$v$ and using the convexity of~$(k,j)$, we get:
\eq{
A^{-\alpha k}\|f_k\|_{L_{q}(Q_{k,j})} \lesssim \|f_{k+L}\|_{L_1(\tilde{w}_{k,j})}\leq \frac{1+\eps}{\eps}\Big( \|f_{k+L}\|_{L_1(\tilde{w}_{k,j})} - \|f_k\|_{L_1(w_{k,j})}\Big),\quad (k,j)\in \CO.
}
It remains to sum over all~$(k,j) \in \CO$ and use~\eqref{Telescopic}:
\mlt{
\sum\limits_{(k,j)\in\CO}A^{-\alpha k}\|f_k\|_{L_{q}(Q_{k,j})}  \lesssim \sum\limits_{(k,j)\in\CO}\Big( \|f_{k+L}\|_{L_1(\tilde{w}_{k,j})} - \|f_k\|_{L_1(w_{k,j})}\Big)\\
\leq \sum\limits_{k=0}^\infty\sum\limits_{j\in\Z^d}\Big( \|f_{k+L}\|_{L_1(\tilde{w}_{k,j})} - \|f_k\|_{L_1(w_{k,j})}\Big) \lesssim \|f\|_{L_1}.
}
\end{proof}
It is convenient to define the sets~$\Omega_k$ as
\eq{\label{OmegakDefinition}
\Omega_k = \bigcup_{j\colon (k,j)\in \CO} Q_{k,j}.
}
\begin{Cor}\label{ConvexAtomsEstimates}
For any~$A > 2$ and~$\eps \in (0,\frac12)$, the estimate
\eq{\label{ConvexAtomsEstimateFormula}
\sum\limits_{k \geq 0}A^{-\alpha k}\|f_k\|_{L_{q}(\Omega_k)} \lesssim \|f\|_{L_1}
}
holds true uniformly in~$f$.
\end{Cor}
\begin{proof}
By the triangle inequality,
\eq{
\|f_k\|_{L_{q}(\Omega_k)}  \leq \sum\limits_{j\colon (k,j)\in \CO}\|f_k\|_{L_{q}(Q_{k,j})},
}
and we arrive at the statement of Proposition~\ref{ConvexSummation}.
\end{proof}
We conclude the section with a brief summary of our progress towards Theorems~\ref{MainTheoremLorentzScale} and~\ref{MainTheoremBesovLorentzScale}. Remark~\ref{DiscretizationRemark} and Lemma~\ref{TruncationLemma} reduce those theorems to the bound
\eq{
\sum\limits_{k \geq 0}A^{-\alpha k}\|f_k\|_{L_{q}} \lesssim \|f\|_{L_1},\qquad f \in \W,
}
provided~$\W$ meets the requirements of those theorems. We have proved a simpler bound~\eqref{ConvexAtomsEstimateFormula}. The sets~$\Omega_k$ in that formula were constructed from~$\eps$-convex atoms. By the triangle inequality, it remains to justify
\eq{\label{SumOverFlatAtoms}
\sum\limits_{k \geq 0} A^{-\alpha k} \|f_k\|_{L_{q}(\!\!\!\bigcup\limits_{(k,j)\in\Fl}\!\!\! Q_{k,j})} \lesssim \|f\|_{L_1},\qquad f\in \W.
}
This bound is more demanding than~\eqref{ConvexAtomsEstimateFormula}. First, we have not used the space~$\W$ yet, and it will play the pivotal role in the bound for flat atoms. Second, we cannot bound them individually as we did with convex atoms in Proposition~\ref{ConvexSummation}, i.e., the bound
\eq{
\sum\limits_{k \geq 0} A^{-\alpha k} \sum\limits_{j\colon (k,j)\in \Fl} \|f_k\|_{L_{q}(Q_{k,j})} \lesssim \|f\|_{L_1},\qquad f\in \W,
}
might be false in general, see Subsection~\ref{NecessityForNonSplit} in the appendix. We will need to split them into groups related to certain trees. However, before that we wish to make the assumptions that~$f\in \W$ and that~$\W$ does not contain vectorial delta measures, quantitative. The forthcoming section is devoted to that topic.


\section{Stronger monotonicity formula}\label{S3}

\subsection{Reduction to the isotropic case}\label{s31}
The target of this section is to obtain an improvement of Proposition~\ref{MonotoniityFormula} in the case where both~$\mu$ and~$w$ are somehow separated from the set of delta measures. In~\cite{Stolyarov2022}, the separability of~$\mu$ was expressed via invariant cones of measures, a notion related to tangent measures. We provide its anisotropic analog. For anisotropic analogs of tangent measures, see, e.g.,~\cite{Mattila2022}.
\begin{Def}\label{InvariantConeOfMeasuresDefinition}
A set~$\M\subset \Sw'(\R^d)$ is called an invariant cone of measures, provided:
\begin{enumerate}[1)]
\item Any element~$\mu \in\M$ is a measure\textup, i.e., a non-negative distribution\textup;
\item The set~$\M$ is closed in the topology of~$\Sw'(\R^d)$\textup;
\item The set~$\M$ is invariant with respect to the dilations~$\Dil_t$\textup;
\item The set~$\M$ is translation invariant\textup;
\item The set~$\M$ is a cone in the sense that~$c\mu\in\M$ provided~$c \geq 0$ and~$\mu \in \M$. 
\end{enumerate}
\end{Def}
We also need the definition of the smoothness function.
\begin{Def}
Let~$w$ be a weight. Its smoothness function~$\s[w]\colon \mathbb{R}_+\to[1,\infty]$ is defined as follows\textup:
\eq{
\s[w](\zeta) = \sup\Set{\frac{w(x)}{w(y)}}{|x-y| \leq \zeta,\ x,y\in\R^d}.
}
\end{Def}
Note that the smoothness function is defined pointwise. Therefore, we usually compute~$\s[w]$ for continuous or piecewise continuous weights~$w$. The smoothness function~$\s[w]$ is non-decreasing. We will also use two simple properties: If~$\Phi \geq 0$, then
\eq{\label{ConvolutionAndSmoothness}
\s[w*\Phi] \leq \s[w],
}
and if~$t \in (0,1)$, then
\eq{\label{DilationAndSmoothness}
\s\!\big[\Dil^{t^{-1}}[w]\big] \leq \s[w].
} 
Note that
\eq{\label{TypicalWeightSmoothness}
\s[(1+|\cdot|)^{-\theta}](\zeta) = (1+\zeta)^{\theta},\qquad \zeta \geq 0.
}
This identity may be derived from the elementary inequality~\eqref{eqE01} in the appendix.
We will often consider weights~$w$ that satisfy the smoothness bound of the form
\eq{\label{StandardBoundOnWeight}
\s[w](\zeta) \leq C_w(1+\zeta)^{\theta_w},\qquad \zeta \geq 0,
}
and need notation for them.
\begin{Def}\label{SmoothnessOfWeightsDef}
Let~$\theta \geq 0$ and~$C \geq 1$ be given. Denote the set of weights~$w$ obeying the bound~\eqref{StandardBoundOnWeight} with~$C_w = C$ and~$\theta_w = \theta$ by~$\Smooth(\theta,C)$.
\end{Def}
The two statements below are variations on Lemmas~\ref{Lemma41},~\ref{Lemma42}, and~\ref{Lemma44}. The proofs are postponed to Subsection~\ref{SWeights} of the appendix.
\begin{Le}\label{Lemma43}
Assume~$G \in \Smooth(\theta,C)$ is a weight. Then,
\eq{\label{HeatingSmoothWeightFormula}
\Heat[G](x,\vec{t}\,)\asymp G(x), \qquad \forall j\in [1\twodots d]\quad t_j \in  [0,2].
}
\end{Le}
The lemma above says that the constants in~\eqref{HeatingSmoothWeightFormula} depend on the parameters~$d,\theta$, and~$C$ only. 
\begin{Le}\label{Lemma47}
Let~$u$ and~$v$ be two weights. Let~$v\in \Smooth(\theta,C)$ and~$v(x) \lesssim u^q(x)$ for all~$x\in \R^d$. Then,
\eq{
\|\Heat[f](\fdot, \vec{t}\,)\|_{L_q(v)}\lesssim \|f\|_{L_1(u)},\qquad \forall j \in [1\twodots d]\qquad t_j \in \Big[\frac12,2\Big].
}
\end{Le}

\begin{St}\label{StrengtheningOfSemiinvariant}
Let~$\M$ be an invariant cone of measures that does not contain delta measures.  Let~$w \in \Smooth(\theta, C)$ be a weight. Then, there exists a number~$\nu > 0$,~$\nu = \nu(\M, \theta, C)$, such that
\eq{\label{StrengtheningOfSemiinvariantFormula}
\Big\|\Heat[\mu]\big(\fdot;t^{a_1},t^{a_2},\ldots, t^{a_d}\big)\Big\|_{L_q\big(\Heat[w](\fdot;\frac{1-t^{a_1}}{q},\frac{1-t^{a_2}}{q},\ldots, \frac{1-t^{a_d}}{q})\big)} \leq t^{-\frac{\alpha}{2} + \nu}\|\Heat[\mu](\fdot;\one)\|_{L_q(w)}
}
for any~$t \in (0,1)$.
\end{St}
Let us fix~$\mu$ and~$w$ and use the notation
\eq{\label{SpDefinition}
S_q[\mu,w](t) = t^{\frac{d(q-1)}{2}}\Big\|\Heat[\mu]\big(\fdot;t^{a_1},t^{a_2},\ldots, t^{a_d}\big)\Big\|^q_{L_q\big(\Heat[w](\fdot;\frac{1-t^{a_1}}{q},\frac{1-t^{a_2}}{q},\ldots, \frac{1-t^{a_d}}{q})\big)}
}
for brevity; sometimes we will abbreviate~$S_q[\mu,w](t)$ as~$S_q(t)$ or even simply as~$S_q$. Then,~\eqref{StrengtheningOfSemiinvariantFormula} reduces to~$S_q(t) \leq t^{\nu q}S_q(1)$. Note that unlike inequalities at the end of Subsection~\ref{s21}, the time parameter~$t$ is one-dimensional here.
\begin{Le}\label{DilationOfInvariantCone}
Let~$\M$ be an invariant cone of measures. Let~$\theta > 0$ and~$C \geq 1$ be fixed. If there exists~$\nu > 0$ such that for any choice of~$\mu\in \M$ and~$w\in \Smooth(\theta, C)$, the bound 
\eq{\label{DilationOfInvariantConeFormula}
\frac{\partial S_q[\mu,w]}{\partial t}\Big|_{t=1} \geq \nu q S_q[\mu,w](1)
}
holds true, then
\eq{
S_q[\mu,w](t) \leq t^{\nu q}S_q[\mu,w](1)
}
holds as well.
\end{Le}
\begin{proof}
The desired inequality~$S_q(t) \leq t^{\nu q}S_q(1)$ follows from
\eq{\label{eq316}
\frac{\partial S_q}{\partial t}(t) \geq \frac{\nu q S_q(t)}{t},\qquad t \in (0,1),
}
via integration:
\eq{
\log S_q(t) = \log S_q(1) - \int\limits_t^1\frac{S_q'(\theta)}{S_q(\theta)}\,d\theta \leq \log S_q(1) - \int\limits_t^1 \frac{\nu q\,d\theta}{\theta} = \log S_q(1) + \nu q \log t.
}
The inequality~\eqref{eq316} is derived from~\eqref{DilationOfInvariantConeFormula} via a scaling argument. To this end, let
\eq{
U(x,s) = \Heat[\mu]\big(x;s^{a_1},s^{a_2},\ldots, s^{a_d}\big);\quad V(x,s) = \Heat[w]\Big(x;\frac{1-s^{a_1}}{q},\frac{1-s^{a_2}}{q},\ldots, \frac{1-s^{a_d}}{q}\Big).
}
Consider the functions~$\tilde{U}$ and~$\tilde{V}$ given  by
\eq{\label{eq3115}
\tilde{U}(x,\theta) = U(\Dil_{t}x, t^2\theta);\qquad \tilde{V}(x,\theta) = V(\Dil_{t}x, t^2\theta),\qquad x\in \R^d,\ \theta \in [0,1],
}
similarly to~\eqref{DILDef}. Let~$\tilde{\mu}$ be the dilation of~$\mu$:~$\tilde{\mu} = t^{-d}\Dil_{t^{-1}} [\mu]$. Then,~$\tilde{U}(x,\theta) = \Heat[\tilde{\mu}](x, \theta)$ by Lemma~\ref{SeveralDilations}, and, moreover,
\mlt{\label{DefinitionOfTildeV}
\tilde{V}(x,\theta) = \Heat[\tilde{w}]\Big(x;\frac{1-\theta^{a_1}}{q},\frac{1-\theta^{a_2}}{q},\ldots, \frac{1-\theta^{a_d}}{q}\Big),\quad \text{where}\\
 \tilde{w}(x)= \Heat[w]\Big(\Dil_{t}x;\frac{1-t^{2a_1}}{q},\frac{1-t^{2a_2}}{q},\ldots, \frac{1-t^{2a_d}}{q}\Big).
}
Let us justify the latter claim. First, by~\eqref{eq3115},
\eq{\label{eq3117}
\tilde{V}(x,\theta) =  \Heat[w]\Big(\Dil_{t} x, \frac{1 - t^{2a_1}\theta^{a_1}}{q}, \frac{1 - t^{2a_2}\theta^{a_2}}{q},\ldots, \frac{1 - t^{2a_d}\theta^{a_d}}{q}\Big).
}
Second,
\mlt{
\Heat[\tilde{w}]\Big(x;\frac{1-\theta^{a_1}}{q},\frac{1-\theta^{a_2}}{q},\ldots, \frac{1-\theta^{a_d}}{q}\Big)\\
 = t^{-d}\Heat\bigg[\Dil_{t^{-1}}\Heat[w]\Big(\fdot;\frac{1-t^{2a_1}}{q},\frac{1-t^{2a_2}}{q},\ldots, \frac{1-t^{2a_d}}{q}\Big)\bigg] \Big(x ,\frac{1 - \theta^{a_1}}{q},\frac{1 - \theta^{a_2}}{q}, \ldots, \frac{1 - \theta^{a_d}}{q}\Big)\\
 \Eref{\text{\tiny Lem.\ref{SeveralDilations}}} \Heat\bigg[\Heat[w]\Big(\fdot;\frac{1-t^{2a_1}}{q},\frac{1-t^{2a_2}}{q},\ldots, \frac{1-t^{2a_d}}{q}\Big)\bigg] \Big(\Dil_{t}x ,\frac{t^{2a_1} - t^{2a_1}\theta^{a_1}}{q},\frac{t^{2a_2} - t^{2a_2}\theta^{a_2}}{q}, \ldots, \frac{t^{2a_d} - t^{2a_d}\theta^{a_d}}{q}\Big)\\
 \Eeqref{SemigroupProperty}\Heat[w]\Big(\Dil_{t} x, \frac{1 - t^{2a_1}\theta^{a_1}}{q}, \frac{1 - t^{2a_2}\theta^{a_2}}{q},\ldots, \frac{1 - t^{2a_d}\theta^{a_d}}{q}\Big),
}
which coincides with~\eqref{eq3117}.

The crucial observation is that~$\tilde{w}\in \Smooth(\theta, C)$ as well. This may be justified as follows:
\mlt{
\s[\tilde{w}](\zeta) = \s\Big[\Heat[w]\Big(\Dil_{t}x;\frac{1-t^{2a_1}}{q},\frac{1-t^{2a_2}}{q},\ldots, \frac{1-t^{2a_d}}{q}\Big)
\Big](\zeta)\\
 \Leqref{DilationAndSmoothness}\s\Big[\Heat[w]\Big(x;\frac{1-t^{2a_1}}{q},\frac{1-t^{2a_2}}{q},\ldots, \frac{1-t^{2a_d}}{q}\Big)
\Big](\zeta) \Leqref{ConvolutionAndSmoothness} \s[w](\zeta).
}

Thus, if we denote
\eq{
T_q(\theta) = \theta^{\frac{d(q-1)}{2}}\int\limits_{\R^d}\tilde{U}^q(x,\theta)\tilde{V}(x,\theta)\,dx,
}
then application of~\eqref{DilationOfInvariantConeFormula} to the measure~$\tilde{\mu}$ in the role of~$\mu$ and the weight~$\tilde{w}$ in the role of~$w$ yields
\eq{\label{eq3112}
\frac{\partial T_q}{\partial \theta}(1) \geq \nu q T_q(1).
}
By substitution,
\mlt{
T_q(\theta) = \theta^{\frac{d(q-1)}{2}}\int\limits_{\R^d}U^q\big(\Dil_{t}x,t^2\theta\big) V\big(\Dil_{t}x,t^2\theta\big)\,dx \\
= t^{-d} \theta^{\frac{d(q-1)}{2}}\int\limits_{\R^d}U^q\big(y,t^2\theta\big) V\big(y,t^2\theta\big)\,dy = t^{-dq}S_q(t^2\theta),
}
and we see that~\eqref{eq3112} implies~\eqref{eq316}.
\end{proof}

Now we return to the proof of Proposition~\ref{StrengtheningOfSemiinvariant} and aim at establishing~\eqref{DilationOfInvariantConeFormula}. For that we will need multiparametric heat extensions and consider time parameters~$\vec{t}$ of arbitrary form, not only of the type~$\Dil_{t^{-1}}\one$. Let~$u\colon \R^d\times [0,1]^d\to \R_+$ be a solution to the multiparametric heat equation~\eqref{MultiparametricHeatEquations}. Let~$v\colon \R^d \times [0,1]^d \to \R_+$ be a solution of the rescaled backwards multiparametric heat equation~\eqref{eq2115}.
Compose the quantity
\eq{
Q_q(\vec{t}\,) = \Big(\prod\limits_{j=1}^dt_j\Big)^{\frac{q-1}{2}}\int\limits_{\R^d}u^q(x,\vec{t}\,)v(x,\vec{t}\,)\,dx.
}
\begin{Le}\label{PartialDerivativeBCT}
For any~$j$,
\mlt{\label{LongPartialDerivative}
\frac{\partial Q_q}{\partial t_j} = \frac{q-1}{8}(4\pi)^{-d}\Big(\prod\limits_{j=1}^d t_j\Big)^{\frac{q-3}{2}}t_j^{-2}\times \\
\times\int\limits_{\R^d}\bigg(\int\limits_{\R^{2d}}|y_j - z_j|^2e^{-\sum_{1}^d\frac{|x_i-y_i|^2+|x_i - z_i|^2}{4t_i}}u(y,0)u(z,0)\,dy\,dz\bigg)u^{q-2}(x,\vec{t}\,)v(x,\vec{t}\,)\,dx.
}
\end{Le}
Before proving Lemma~\ref{PartialDerivativeBCT}, we recall a useful formula for the particular case~$d=1$ from~\cite{Stolyarov2022}. For that, fix~$d=1$ and define the measure~$\mu_{x,t}$ on~$\R$ as the absolutely continuous measure with the density
\eq{\label{eq3123}
d\mu_{x,t}(y) = e^{-\frac{|x-y|^2}{4t}} u(y,0)\,dy.
}
The paper~\cite{Stolyarov2022} also used the notation~$\mu = u(\fdot,0)$. Note that
\eq{\label{eq3124}
\mu_{x,t}(\R) = \sqrt{4\pi t}\,u(x,t).
}
The Euclidean space equipped with the probability measure~$(\mu_{x,t}(\R))^{-1}\,\mu_{x,t}$ becomes a probability space. The function~$y$ may be treated as the vectorial random variable~$Y_x$ on this probability space. Formula~$(2.12)$ in~\cite{Stolyarov2022} says:
\eq{
\frac{\partial Q_q}{\partial t} = \frac{q-1}{4} (4\pi)^{-\frac{q}{2}} t^{-\frac52}\int\limits_{\R}\Disp(x- Y_x)(\mu_{x,t}(\R))^qv(x,t)\,dx.
}
It will be convenient to transform this formula using the identity
\eq{\label{DispersionFormula}
\Disp \zeta = \frac12 \E |\zeta_1 - \zeta_2|^2,\qquad \text{where}\ \zeta_1 \text{ and } \zeta_2 \text{ are independent copies of }\zeta.
}
Thus,
\mlt{\label{eq3127}
\frac{\partial Q_q}{\partial t} = \frac{q-1}{8} (4\pi)^{-\frac{q}{2}} t^{-\frac52}\int\limits_{\R}\bigg(\iint\limits_{\R^2} |y-z|^2\,d\mu_{x,t}(y)\,d\mu_{x,t}(z)\bigg)(\mu_{x,t}(\R))^{q-2}v(x,t)\,dx\\
\EeqrefTwo{eq3123}{eq3124} \frac{q-1}{8} (4\pi)^{-1} t^{\frac{q-7}{2}}\int\limits_{\R}\bigg(\iint\limits_{\R^2} |y-z|^2e^{-\frac{|x-y|^2 + |x-z|^2}{4t}}u(y,0) u(z,0)\,dy\,dz\bigg)u^{q-2}(x,t)v(x,t)\,dx.
}
\begin{proof}[Proof of Lemma~\ref{PartialDerivativeBCT}]
To pass to the case of arbitrary dimension, consider the expression
\eq{
R(t_j) = t_j^{\frac{q-1}{2}} \int\limits_{\R}u^{q}(x,\vec{t}\,)v(x,\vec{t}\,)\,dx_j,
}
where all the~$t_i$ except~$t_j$ are fixed; this quantity also depends on all the~$x_i$ except~$x_j$. Note that
\eq{\label{ExpressingQInTermsofR}
Q_q(t) = \Big(\prod\limits_{i\ne j}t_i\Big)^{\frac{q-1}{2}}\int\limits_{\R^{d-1}} R(t_j)\, dx_1\, dx_2\ldots dx_{j-1}\, dx_{j+1}\ldots dx_d. 
}
By application of~\eqref{eq3127},
\mlt{
\frac{\partial R}{\partial t_j} = \frac{q-1}{8}\cdot(4\pi)^{-1}\cdot t_j^{\frac{q-7}{2}}\times \\ 
\times\int\limits_{\R}\bigg(\iint\limits_{\R^2}|y-z|^2 e^{-\frac{|x_j-y|^2 + |x_j - z|^2}{4t_j}}u(x_1,\ldots, x_{j-1},y, x_{j+1},\ldots, x_d; t_1,\ldots, t_{j-1},0,t_{j+1},\ldots, t_d)\times \\
\times u(x_1,\ldots, x_{j-1},z, x_{j+1},\ldots, x_d; t_1,\ldots, t_{j-1},0,t_{j+1},\ldots, t_d)\,dy\,dz\bigg)\cdot u^{q-2}(x,\vec{t}\,)v(x,\vec{t}\,)\,dx_j\\
\Eeqref{SemigroupProperty}\frac{q-1}{8}\cdot(4\pi)^{-1}\cdot t_j^{\frac{q-7}{2}}\cdot \prod_{i\ne j}(4\pi t_i)^{-1}\times \\
\times\int\limits_\R\bigg(\int\limits_{\R^{2d}}|y_j - z_j|^2e^{-\sum_{1}^d\frac{|x_i - y_i|^2 + |x_i - z_i|^2}{4t_i}} u(y,0)u(z,0)\,dy\,dz\bigg)\cdot u^{q-2}(x,\vec{t}\,)v(x,\vec{t}\,)\,dx_j.
}
Plugging this into~\eqref{ExpressingQInTermsofR}, we obtain~\eqref{LongPartialDerivative}.
\end{proof}
\begin{proof}[Proof of Proposition~\ref{StrengtheningOfSemiinvariant}]
By Lemma~\ref{DilationOfInvariantCone}, it suffices to obtain~\eqref{DilationOfInvariantConeFormula}. Note that
\eq{
S_q(t) = Q_q(t^{a_1}, t^{a_2},\ldots, t^{a_d})\quad \text{and} \quad \frac{\partial S_q}{\partial t} = \sum\limits_{j=1}^d a_j t^{a_j - 1}\frac{\partial Q_q}{\partial t_j}.
}
Thus,
\mlt{
S_q'(1) = \frac{q-1}{8}\cdot(4\pi)^{-d}\times \\
\times \int\limits_\R\bigg(\int\limits_{\R^{2d}}\Big(\sum\limits_{j=1}^da_j |y_j - z_j|^2\Big)e^{-\sum_{1}^d\frac{|x_i - y_i|^2 + |x_i - z_i|^2}{4}} u(y,0)u(z,0)\,dy\,dz\bigg)u^{q-2}(x,\one\,)v(x,\one\,)\,dx.
}
The principal idea of the proof is that the quantities above for arbitrary anisotropy~$a$ and the standard isotropic~$a=\one$ are comparable. This allows to reduce the problem to the latter particular case already considered in~\cite{Stolyarov2022}. By using~\eqref{DispersionFormula} once again, 
\eq{
S_q'(1) \gtrsim \int\limits_{\R^d}\Disp(x- Y_x)(\mu_{x,1}(\R))^qv(x,1)\,dx,
}
where~$Y_x$ is the vectorial random variable~$y\in \R^d$ on the same probability space. Thus, the desired inequality~\eqref{DilationOfInvariantConeFormula} reduces to
\mlt{\label{VoidInequality}
\int\limits_{\R^d}\Big(\int\limits_{\R^d} |\mass(x)-y|^2e^{-\frac{|x-y|^2}{4}}\,d\mu(y)\Big)\Big(\int\limits_{\R^d} e^{-\frac{|x-y|^2}{4}}\,d\mu(y)\Big)^{q-1}G(x)\,dx\\
 \geq \delta q \int\limits_{\R^d}\Big(\int\limits_{\R^d} e^{-\frac{|x-y|^2}{4}}\,d\mu(y)\Big)^qG(x)\,dx,
}
where~$G$ denotes~$v(\fdot,1)$ and
\eq{
\mass(x) = \frac{\int_{\R^d} y\, e^{-\frac{|x-y|^2}{4}}\,d\mu(y)}{\int_{\R^d}e^{-\frac{|x-y|^2}{4}}\,d\mu(y)}.
}
Proposition~$3.10$ and the proof of Theorem~$3.1$ in~\cite{Stolyarov2022} say that if this inequality is violated (there does not exist~$\nu$ such that~\eqref{VoidInequality} holds true for all~$\mu\in \M$), then~$\M$ contains a delta measure. Since this is not the case, we have obtained~\eqref{DilationOfInvariantConeFormula}.
\end{proof}
\begin{Def}
Let~$\M$ be an invariant cone of measures, let~$C \geq 1$,~$\theta > 0$. Define
\eq{
\kappa(\M,\theta, C) = \sup\Set{\nu}{ \forall \mu \in \M, w\in \Smooth(\theta,C), t \in (0,1) \quad \eqref{StrengtheningOfSemiinvariantFormula} \ \text{holds true}}.
}
\end{Def}
\begin{Rem}
Proposition~\ref{StrengtheningOfSemiinvariant} may be restated:~$\kappa(\M,\theta,C) > 0$ if and only if~$\M$ does not contain delta measures. Note that the supremum in the definition above is attained.
\end{Rem}


\subsection{A more convenient inequality}\label{s32}
We will need a technical definition.
\begin{Def}\label{HeatStableCones}
Let~$\M$ be an invariant cone of measures. We say that it is heat stable, provided for all~$\mu\in \M$ and~$\vec{s} \in (\R_+)^d$ we also have~$\Heat[\mu](\fdot, \vec{s}) \in \M$.
\end{Def}
Recall the numbers~$K$ and~$L$ defined in~\eqref{DefinitionOfK}.
\begin{St}\label{Convenient}
Let~$\M$ be an invariant heat stable cone of measures that does not contain delta measures, let also~$\theta > 0$ and~$C \geq 1$ be fixed. There exists~$\nu > 0$ such that for any sufficiently small~$t > 0$ and
any non-negative solution to the multiparametric heat equation~\eqref{MultiparametricHeatEquations} on~$\prod_{j=1}^d [t^{Ka_j},1]$ such that
\eq{\label{eq321}
u\big(\fdot; t^{Ka_1}, t^{Ka_2},\ldots, t^{Ka_d}\big) \in \M,
}
the inequality~\eqref{StrengtheningOfSemiinvariantFormula} in the form
\eq{
\Big\|u\big(\fdot;t^{a_1},t^{a_2},\ldots, t^{a_d}\big)\Big\|_{L_q\big(\Heat[w](\fdot;\frac{1-t^{a_1}}{q},\frac{1-t^{a_2}}{q},\ldots, \frac{1-t^{a_d}}{q})\big)} \leq t^{-\frac{\alpha}{2} + \nu}\|u(\fdot;\one)\|_{L_q(w)}
}
holds true, provided~$w \in \Smooth(\theta,C)$.
\end{St}
We will derive this proposition from Proposition~\ref{StrengtheningOfSemiinvariant}, which is translated into PDE language in the same manner as Proposition~\ref{MonotoniityFormula}.
\begin{St}
Let~$u\colon \R^d \times [0,1]^d\to \R$ be a non-negative solution to the multiparametric heat equation, let also~$u(\fdot;0)\in \M$. Let~$v$ be a non-negative solution to~\eqref{eq2115}, let also~$v(\fdot,\one)\in \Smooth(\theta,C)$. Then,
\eq{
\int\limits_{\R^d}u^q(x; t^{a_1},t^{a_2},\ldots, t^{a_d}) v(x; t^{a_1},t^{a_2},\ldots, t^{a_d})\,dx \leq t^{-\frac{d(q-1)}{2} + \kappa q} \int\limits_{\R^d} u^q(x; \one)v(x;\one)\,dx,
}
where~$\kappa = \kappa(\M,\theta,C)>0$.
\end{St}
Applying the same dilations as we used to derive Proposition~\ref{RescaledPDECorollary} from Corollary~\ref{PDECorollary}, we obtain the following corollary.
\begin{Cor}\label{StrengthenedRescaledPDECorollary}
Fix~$s > 0$. Let~$u\colon \R^d \times \prod_j [0,s^{a_j}]\to \R$ be a non-negative solution to the multiparametric heat equation, let also~$u(\fdot;0)\in \M$. Let~$v$ be a non-negative solution to~\eqref{eq2115} on the same domain, let also~$v(\fdot,s^{a_1},s^{a_2},\ldots,s^{a_d})\in \Smooth(\theta,C)$. Then,
\mlt{
\int\limits_{\R^d}u^q(x; t^{a_1},t^{a_2},\ldots, t^{a_d}) v(x; t^{a_1},t^{a_2},\ldots, t^{a_d})\,dx\\
 \leq \Big(\frac{t}{s}\Big)^{-\frac{d(q-1)}{2} + \kappa q} \int\limits_{\R^d} u^q\big(x; s^{a_1},s^{a_2},\ldots,s^{a_d}\big)v\big(x;s^{a_1},s^{a_2},\ldots,s^{a_d}\big)\,dx,
}
where~$\kappa = \kappa(\M,\theta,C)>0$.
\end{Cor}

In the corollary above, we may replace the parallelepiped~$\prod_j [0,s^{a_j}]$ with~$\prod_j [1-s^{a_j},1]$.

\begin{figure}[h!]
\centerline{
\includegraphics[height=7cm]{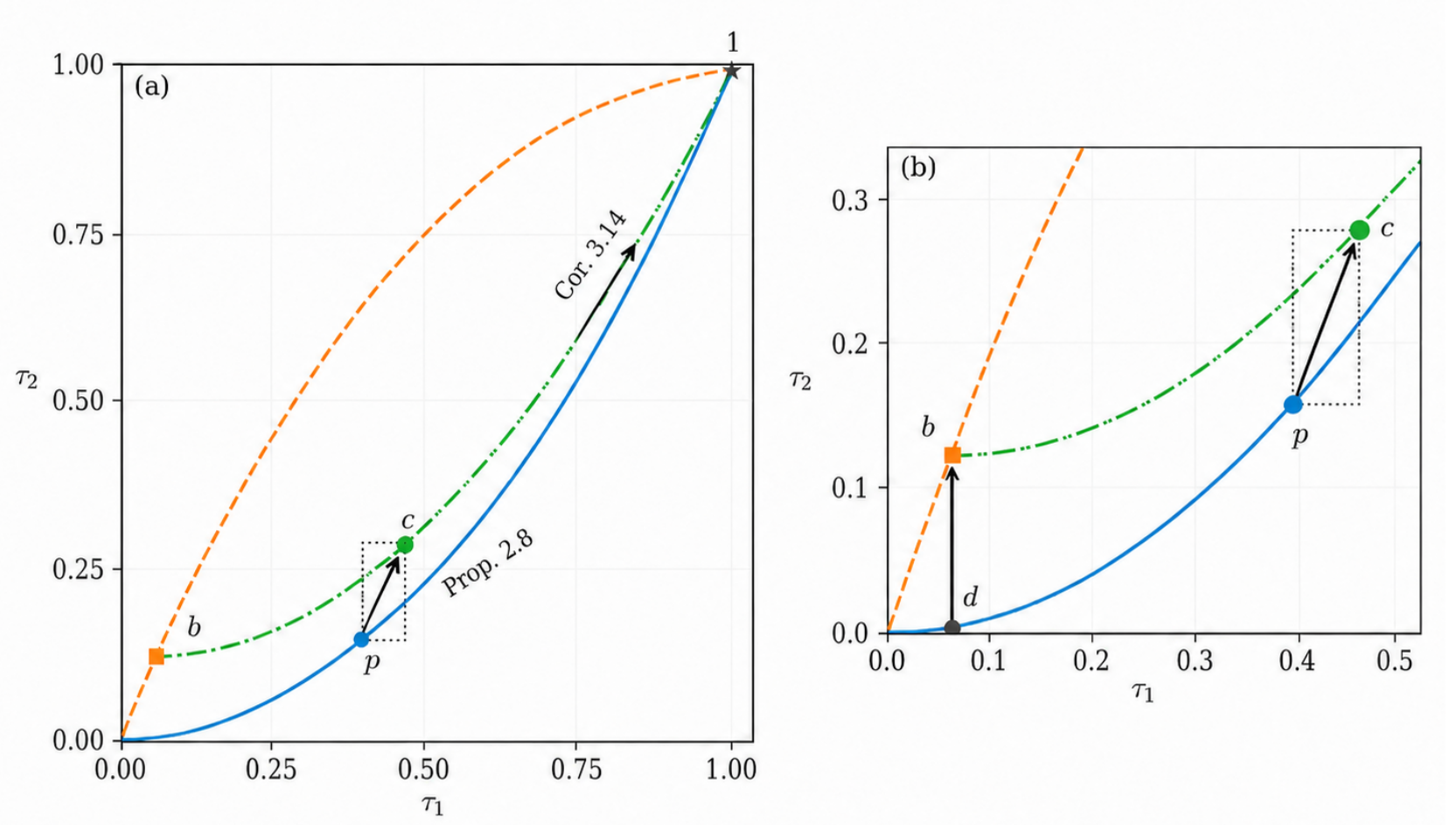}
}
\caption{Illustration to the proof of Proposition~\ref{Convenient}. Here~$d=2$,~$a_1 = 2/3$,~$a_2 = 4/3$, and~$K=3$. The blue curve is~$(\tau^{a_1},\tau^{a_2})$, which is the parabola in this case. The orange curve is~$(1-\sigma^{a_1},1-\sigma^{a_2})$. The point~$d$ is~$(t^{Ka_1}, t^{Ka_2})$ for some fixed small~$t$. The point~$b$ is~$(1-s^{a_1}, 1-s^{a_2})$ defined by~\eqref{ChoiceOfs}. The green curve is~$\tau \mapsto (1-s^{a_1} + \tau^{a_1},1-s^{a_2} + \tau^{a_2})$, and the black arrow from~$p = (t^{a_1},t^{a_2})$ to~$c$ signifies the application of Proposition~\ref{RescaledPDECorollary}. After that, the application of Corollary~\ref{StrengthenedRescaledPDECorollary} corresponds to the move from~$c$ to~$(1,1)$ along the green curve.}
\label{Figure311}
\end{figure}

\begin{proof}[Proof of Proposition~\ref{Convenient}]
We pick~$t$ and choose a number~$s$ close to~$1$ such that
\eq{\label{FirstAssumptionOns}
\forall j \qquad t^{Ka_j} \leq 1-s^{a_j};
}
there will be further requirements on this number. Then, by the assumptions that~$\M$ is heat stable and~\eqref{eq321}, we have
\eq{
u(\fdot; 1-s^{a_1},1-s^{a_2},\ldots, 1-s^{a_d})\in \M.
}
Then, with the standard notation~$v(x,\vec{\tau}) = \Heat[w](x;\frac{1-\tau_1}{q},\frac{1-\tau_2}{q},\ldots, \frac{1-\tau_d}{q})$, we have
\mlt{
\int\limits_{\R^d}u^q(x; t^{a_1},t^{a_2},\ldots, t^{a_d}) v(x; t^{a_1},t^{a_2},\ldots, t^{a_d})\,dx \\
\Lref{\text{\tiny Prop.~\ref{RescaledPDECorollary}}} \Big(\prod\limits_{j=1}^d\frac{t^{a_j}}{t^{a_j}+1-s^{a_j}}\Big)^{-\frac{q-1}{2}}\int\limits_{\R^d}u^q\big(x; t^{a_1}+1-s^{a_1},\ldots, t^{a_d}+1-s^{a_d}\big) v\big(x; t^{a_1}+1-s^{a_1},\ldots, t^{a_d}+1-s^{a_d}\big)\,dx \\
\Lref{\text{\tiny Cor.~\ref{StrengthenedRescaledPDECorollary}}} \Big(\prod\limits_{j=1}^d\frac{t^{a_j}}{t^{a_j}+1-s^{a_j}}\Big)^{-\frac{q-1}{2}} \Big(\frac{t}{s}\Big)^{-\frac{d(q-1)}{2}+\kappa q}\int\limits_{\R^d}u^q(x;\one) v(x;\one)\,dx,
}
where~$\kappa = \kappa(\M,\theta,C)$.
Thus, it remains to justify
\eq{
 \Big(\prod\limits_{j=1}^d\frac{t^{a_j}}{t^{a_j}+1-s^{a_j}}\Big)^{-\frac{q-1}{2}} \Big(\frac{t}{s}\Big)^{-\frac{d(q-1)}{2}+\kappa q} \leq t^{-\frac{d(q-1)}{2}+\frac{\kappa q}{2}},
}
since we may set~$\nu := 1/2\, \kappa(\M,\theta,C)$ in Proposition~\ref{Convenient}. The latter inequality is equivalent to
\eq{
\prod\limits_{j=1}^d\Big(1+ \frac{1-s^{a_j}}{t^{a_j}}\Big)^{\frac{q-1}{2}} s^{\frac{d(q-1)}{2}-\kappa q} \leq t^{-\frac{\kappa q}{2}}.
}
We claim that for any sufficiently small~$t$, there exists a choice of~$s = s(t)$ that fulfills~\eqref{FirstAssumptionOns} and for which the left-hand side of the inequality above is bounded (uniformly bounded with respect to~$t$); clearly, the claim yields the inequality above since the right hand side blows up as~$t\to 0$. We choose~$s$ in such a way that
\eq{\label{ChoiceOfs}
t = \min\limits_j (1-s^{a_j})^{\frac{1}{K a_j}}.
}
Let the minimum be attained at~$j=i$. The variable~$r = 1-s$ is more convenient. In this case,
\eq{
t = \big(a_i r+ o(r)\big)^{\frac{1}{K a_i}}.
}
The boundedness of the left hand side reduces to the boundedness of
\eq{
\frac{1-s^{a_j}}{(1-s^{a_i})^{\frac{a_j}{Ka_i}}} \asymp \frac{a_j r + o(r)}{\big(a_i r + o(r)\big)^{\frac{a_j}{Ka_i}}} = O(r^{1-\frac{a_j}{Ka_i}}),
}
which is bounded, provided~$K \geq a_j/a_i$ for any~$j$. The latter inequality is true by the definition of~$K$ in~\eqref{DefinitionOfK}.
\end{proof}
\begin{Rem}\label{Rem312}
By a better choice of~$\nu$, we may show the following: For any~$\nu \in (0,\kappa(\M,\theta,C))$ the inequality~\eqref{StrengtheningOfSemiinvariantFormula} holds true, provided~$t$ is sufficiently small, the required smallness of~$t$ depends on~$\nu$.
\end{Rem}


\section{End of the proof}\label{S4}

\subsection{A compactness argument}\label{s41}
Given a translation invariant closed subspace~$\W\subset \Sw'(\R^d; \R^\ell)$ that is also dilation invariant, one naturally constructs the set
\eq{
\M^\W = \Set{\mu \in \Sw'(\R^d)}{\mu \geq 0, \exists a\in \R^\ell \setminus \{0\} \quad a\otimes \mu \in \W}.
}
The set~$\M^\W$ is an invariant cone of measures in the sense of Definition~\ref{InvariantConeOfMeasuresDefinition}. By Proposition~\ref{ConvolutionTranslationInvariant} in the appendix, it is also a heat stable cone, see Definition~\ref{HeatStableCones}.
We wish to find a strengthening of Proposition~\ref{Convenient} that applies to vector-valued distributions in the sense that the assumption~\eqref{eq321} is replaced with~$f\in \W$ and the flatness assumption is the same as in our main Definition~\ref{Convex/flat}. It appears that one cannot do this in the very straightforward way, and an additional assumption that the corresponding function or measure is localized, is needed. We provide an example showing this necessity in Subsection~\ref{SNecessity} of the appendix and now state the localization assumption. Let~$u$ be another weight. We assume that it decays slower than~$w$ at infinity. The said localization assumption reads as
\eq{\label{Localization}
\|f_K\|_{L_1(u)} \leq C \|f_{K}\|_{L_1(\Heat[w](\fdot; 1 - A^{-2Ka_1}, 1 - A^{-2Ka_2}, \ldots, 1 - A^{-2Ka_d}))}.
}
This is the localization assumption for the atom~$(0,0)$, later we will transfer it to an arbitrary atom via dilations and translations. The choice of the scale~$K$ is dictated by Proposition~\ref{Convenient}. Lemmas~\ref{Lemma41} and~\ref{Lemma42} say that~$\Heat[w](\fdot; 1 - A^{-2Ka_1}, 1 - A^{-2Ka_2}, \ldots, 1 - A^{-2Ka_d})$ may be replaced with simply~$w$ in this definition; the constant~$C$ will change as well. 

We now fix the weights. We set
\eq{\label{OurWeights}
\begin{aligned}
w(x) &= \frac{(1+|x|^2)^{-\frac{\theta_1}{2}}}{\sum_{j\in \Z^d} (1+|x-j|^2)^{-\frac{\theta_1}{2}}};\\
u(x) &= (1+|x|^2)^{-\frac{\theta_2}{2}}; \\
v(x) &= (1+|x|^2)^{-\frac{\theta_3}{2}}.
\end{aligned}
}
A more complicated formula for~$w$ is prescribed by~\eqref{SumOneForWeight}. The role of the weight~$v$ will become clear slightly later.
We assume
\alg{
\label{eq414}& 2d < \theta_2 < \theta_1;\\
\label{eq415}& q\theta_1 < \theta_3.
}
The second inequality in~\eqref{eq414} says~$w$ is more concentrated than~$u$ as required by our understanding of the localization assumption~\eqref{Localization}. The new weight~$v$ decays at infinity so fast that~$\|f_1\|_{L_q(v)}$ is controlled by~$\| f_3\|_{L_1(w)}$ according to Lemma~\ref{Lemma44}. Note that our choice for~$w$ satisfies the requirement~\eqref{eq2216} and~$w\in \Smooth(\theta_1, C_1)$,~$u \in \Smooth(\theta_2,C_2)$,~$v\in \Smooth(\theta_3,C_3)$ for sufficiently large constants~$C_1$,~$C_2$, and~$C_3$ (recall Definition~\ref{SmoothnessOfWeightsDef} and formula~\eqref{TypicalWeightSmoothness}).  We need yet another weighted lemma. If~$\Omega \subset \R^d$ and~$g \colon \Omega \to \R$ is a function, we define its Lipschitz constant by the formula
\eq{
\|g\|_{\Lip(\Omega)} = \sup\Set{\frac{|g(x) - g(y)|}{|x-y|}}{x,y\in \Omega, x\ne y}.
}
\begin{Le}\label{LipschitzBound}
Assume~$w$ satisfies the bound~\eqref{eq2216} and~$R$ is fixed. Then,
\eq{
\big\|\Heat[f](\fdot, \vec{t}\,)\big\|_{\Lip(B_R(0))}\lesssim \|f\|_{L_1(w)}, \qquad \forall j \quad t_j \in \Big[\frac12,2\Big].
}
\end{Le}
We omit the tedious proof of Lemma~\ref{LipschitzBound}.

\begin{Th}\label{Theorem41}
Assume~$\delta_0 \notin \M^\W$ and fix~$\nu \in (0,\kappa(\M^\W,\theta_1, C_1))$. For every~$C \geq 1$ and every sufficiently large~$A$ there exists~$\eps > 0$ such that the flatness condition
\eq{\label{FlatnessInThm41}
\|f_L\|_{L_1(\tilde{w})} \leq (1+\eps) \|f_0\|_{L_1(w)},\qquad \tilde{w} = \Heat [w]\big(\fdot; 1 - A^{-2a_1L},  1 - A^{-2a_2L}, \ldots,  1 - A^{-2a_dL}\big),
}
and the localization condition~\eqref{Localization} yield
\eq{\label{eq426}
\|f_1\|_{L_q(\tilde{v})} \leq A^{\alpha - 2\nu} \|f_0\|_{L_q(v)},\qquad \tilde{v} = \Heat[v]\Big(\fdot; \frac{1- A^{-2a_1}}{q},\frac{1- A^{-2a_2}}{q},\ldots, \frac{1- A^{-2a_d}}{q}\Big).
}
\end{Th} 
\begin{proof}
Assume the contrary: Let there exist a sequence of functions~$f^n \in \W$ such that
\alg{
\label{flatness1n}\|f_L^n\|_{L_1(\tilde{w})} &\leq \frac{n+1}{n}\|f_0^n\|_{L_1(w)},\\
\label{localization}\|f^n_K\|_{L_1(u)} &\leq C \|f^n_{K}\|_{L_1(\Heat[w](\fdot; 1 - A^{-2Ka_1}, 1 - A^{-2Ka_2}, \ldots, 1 - A^{-2Ka_d}))}, \\
\label{antigrowth}\|f^n_1\|_{L_q(\tilde{v})} &\geq A^{\alpha- 2\nu} \|f^n_0\|_{L_q(v)}.
}
Without loss of generality, we may assume
\eq{
\|f_0^n\|_{L_1(w)} = 1.
}
By~\eqref{flatness1n},~$\|f_L^n\|_{L_1(\tilde{w})} \leq 2$. By Lemma~\ref{LipschitzBound}, this yields the functions~$f_K^n$ are uniformly Lipschitz on every bounded subset of~$\R^d$. By this and~\eqref{localization}, the sequence~$\{f_K^n\}_n$ is precompact in the space
\eq{
L_1\big(\Heat[w](\fdot; 1-A^{-2Ka_1}, 1-A^{-2Ka_2},\ldots , 1-A^{-2Ka_d})\big);
}
formally, we may cite Lemma~$12$ in~\cite{Stolyarov2022}. Without loss of generality, let this sequence converge to a function~$F$. Then, since Lemma~\ref{BasicMonotonicityLemma} yields
\eq{
\|f_0^n\|_{L_1(w)} \leq \|f^n_{K}\|_{L_1(\Heat[w](\fdot; 1 - A^{-2Ka_1}, 1 - A^{-2Ka_2}, \ldots, 1 - A^{-2Ka_d}))} \leq \|f_L^n\|_{L_1(\tilde{w})} \Leqref{flatness1n} \frac{n+1}{n}\|f_0^n\|_{L_1(w)},
}
we have
\eq{
\|F\|_{L_1(\Heat[w](\fdot; 1 - A^{-2Ka_1}, 1 - A^{-2Ka_2}, \ldots ,1 - A^{-2Ka_d}))} = 1.
}
On the other hand, by Lemma~\ref{BasicMonotonicityLemma},
\eq{
f_0^n \longrightarrow \Heat[F]\Big(\fdot; 1 - A^{-2Ka_1}, 1 - A^{-2Ka_2}, \ldots ,1 - A^{-2Ka_d}\Big) \qquad \text{in}\quad L_1(w),\quad n\to \infty.
}
Thus, the above leads to
\mlt{
\Big\|\Heat[F]\Big(\fdot; 1 - A^{-2Ka_1}, 1 - A^{-2Ka_2}, \ldots ,1 - A^{-2Ka_d}\Big)\Big\|_{L_1(w)}\\
 = \|F\|_{L_1(\Heat[w](\fdot; 1 - A^{-2Ka_1}, 1 - A^{-2Ka_2}, \ldots ,1 - A^{-2Ka_d}))},
}
since both sides are equal to one.
By Lemma~\ref{ZeroFlatnessLemma},~$F = a\otimes h$ with~$h \geq 0$ and~$a \in \R^\ell \setminus \{0\}$. Since~$F\in \W$, we have~$h \in \M^\W$. By Lemma~\ref{Lemma47},
\eq{
\begin{aligned}
&f^n_0 \longrightarrow a \otimes \Heat[h](\fdot, 1 - A^{-2Ka_1},\ldots, 1 - A^{-2Ka_d})& \quad & \text{in } L_q(v),\ n \to \infty;\\
&f^n_1 \longrightarrow a\otimes \Heat[h](\fdot, A^{-2a_1} - A^{-2Ka_1},\ldots, A^{-2a_d} - A^{-2Ka_d})& \quad & \text{in } L_q(\tilde{v}),\ n \to \infty.
\end{aligned}
}
Consequently,~\eqref{antigrowth} yields
\mlt{
\|\Heat[h](\fdot, A^{-2a_1} - A^{-2Ka_1},A^{-2a_2} - A^{-2Ka_2},\ldots, A^{-2a_d} - A^{-2Ka_d})\|_{L_q(\tilde{v})}\\
 \geq A^{\alpha - 2\nu} \|\Heat[h](\fdot, 1 - A^{-2Ka_1},1 - A^{-2Ka_2},\ldots, 1 - A^{-2Ka_d})\|_{L_q(v)}.
}
This contradicts Proposition~\ref{Convenient} and Remark~\ref{Rem312} (we set~$t := A^{-2}$ and
\eq{
u(x,\vec{\tau}) = \Heat[h](\fdot; t_1 - A^{-2Ka_1},\ldots, t_d - A^{-2Ka_d}),\qquad x\in \R^d,\ \vec{\tau} \in \prod\limits_{j=1}^d [A^{-2Ka_j},1],
} 
when we apply Proposition~\ref{Convenient}).
\end{proof}
The next proposition contains the anisotropic analog of the second half of Theorem~$5$ in~\cite{Stolyarov2022}. It might be thought of as a generalization of Lemma~\ref{ZeroFlatnessLemma}. That lemma says, in particular, that the existence of a~$0$-flat atom yields that~$f$ is a rank-one measure. Now we wish to find a more robust version that somehow describes a similar phenomenon for~$\eps$-flat atoms. The same philosophy that led us to introduction of the assumption~\eqref{Localization} in Theorem~\ref{Theorem41} says there should be some concentration assumption in this robust version of Lemma~\ref{ZeroFlatnessLemma} as well. We will use the same one for convenience. 
\begin{St}\label{Prop45}
Let the atom~$(0,0)$ be~$\eps$-flat in the sense that~\eqref{FlatnessInThm41} holds true. Let it be localized in the sense that~\eqref{Localization} holds true. Then,
\eq{\label{NonzeroAnisotropic}
\|f_K\|_{L_1(\Heat[w](\fdot, 1 - A^{-2Ka_1}, 1 - A^{-2Ka_2}, \ldots, 1 - A^{-2Ka_d}))} \lesssim \|f_0\|_{L_q(Q_{0,0})}
}
for any~$q$, provided~$\eps$ is sufficiently small; the constant in the latter inequality is independent of~$A$.
\end{St}
Note that the quantity~$\|f_0\|_{L_1(w)}$ is not controlled by~$\|f_0\|_{L_q(Q_{0,0})}$ in general. We will derive Proposition~\ref{Prop45} from its isotropic version, Theorem~$5$ in~\cite{Stolyarov2022}. In the isotropic case we have~$a = \one$,~$K = 2$, and~$L=3$. Let us replace the scaling parameter~$A$ by~$B$ for a while. The flatness condition~\eqref{FlatnessInThm41} reads as
\eq{\label{IsotropicFlatness}
\|\heat[f](\fdot, B^{-6})\|_{L_1(\heat[w](\fdot, 1 - B^{-6}))} \leq (1+\eps)\|\heat[f](\fdot, 1)\|_{L_1(w)}.
}
Here we use the classical heat extension~\eqref{ClassicalHeat}. The concentration condition becomes
\eq{\label{IsotropicConcentration}
\|\heat[f](\fdot, B^{-4})\|_{L_1(u)} \leq C \|\heat[f](\fdot, B^{-4})\|_{L_1(\heat[w](\fdot, 1- B^{-4}))}.
}
Theorem~$5$ in~\cite{Stolyarov2022} says that these two conditions together with~$f\in \W$ yield the bound
\eq{
\|\heat[f](\fdot, B^{-4})\|_{L_1(\heat[w](\fdot, 1 - B^{-4}))} \lesssim \|f_0\|_{L_q(Q_{0,0})},
}
provided~$B$ is sufficiently large; the multiplicative constant in the inequality is independent of~$B$. To be more precise, Theorem~$5$ of~\cite{Stolyarov2022} contains isotropic versions of both Theorem~\ref{Theorem41} and Proposition~\ref{Prop45} as its first and second parts.
\begin{Rem}
We note that the second part of Theorem~$5$ in~\cite{Stolyarov2022}  requires neither~$\delta_0 \notin \M^\W$, nor~$f\in \W$. While this is not stated explicitly in that paper, a direct inspection shows: The proof of the theorem starts with the justification of its second part, and this justification occupies five lines only.
\end{Rem}

\begin{proof}[Proof of Proposition~\ref{Prop45}] We choose a number~$B > 1$ such that
\eq{
\forall j \in [1\twodots d] \qquad B^2 \leq A^{Ka_j},\ B^3 \leq A^{La_j}.
}
It will be convenient for us to assume that one of these inequalities turns into equality,~$B = \min_j A^{\frac{La_j}{3}}$; since~$K \geq 2$ and~$L=K+1$, the second group of inequalities is stronger than the first one. 
In particular,~$B$ is a certain power of~$A$. We may also assume~$B$ is sufficiently large. 

We apply Lemma~\ref{BasicMonotonicityLemma} together with~\eqref{SemigroupProperty}:
\eq{
\|\heat[f](\fdot, B^{-6})\|_{L_1(\heat[w](\fdot, 1- B^{-6}))} \leq \|f_L\|_{L_1(\tilde{w})} \leq (1+\eps)\|f_0\|_{L_1(w)},
}
and, thus, verify the flatness assumption~\eqref{IsotropicFlatness}. Similarly,
\mlt{
\|\heat[f](\fdot, B^{-4})\|_{L_1(u)}\\
 \leq \|f_K\|_{L_1(\Heat[u](\fdot, B^{-4} - A^{-2Ka_1},B^{-4} - A^{-2Ka_2},\ldots ,B^{-4} - A^{-2Ka_d}))} \Lsref{\text{\tiny Lem.\ref{Lemma41}}} \|f_K\|_{L_1(u)}.
}
Therefore,
\eq{
\|\heat[f](\fdot, B^{-4})\|_{L_1(u)} \Leqref{Localization}C' \|f_L\|_{L_1(\tilde{w})} \Leqref{FlatnessInThm41}  C'(1+\eps) \|\heat[f](\fdot, B^{-4})\|_{L_1(\heat[w](\fdot, 1- B^{-4}))},
}
and the isotropic concentration assumption~\eqref{IsotropicConcentration} is also verified. Thus, we have reduced the proposition to the isotropic case and may conclude that
\eq{
\|\heat[f](\fdot; B^{-4})\|_{L_1(\heat[w](\fdot; 1- B^{-4}))} \lesssim \|f_0\|_{L_q(Q_{0,0})},
}
which, by the flatness assumption yields~\eqref{NonzeroAnisotropic}.
\end{proof}

We conclude this section by adjusting Theorem~\ref{Theorem41} and Proposition~\ref{Prop45} to an arbitrary atom. For that, we consider the weights
\eq{
u_{k,j} = \Dil^{A^{-k}}[u(\fdot - j)];\qquad v_{k,j} = \Dil^{A^{-k}}[v(\fdot - j)], \qquad k \in \{0\}\cup \N,\ j\in \Z^d.
}
The concentration condition for the atom~$(k,j)$ then reads as
\mlt{\label{ConcentrationAtom}
\big\|f_{k+K}\big\|_{L_1(u_{k,j})} \leq C\big\| f_{k+K}\big\|_{L_1[\tilde{w}_{k,j}]},\\
\tilde{w}_{k,j} = \Heat[w_{k,j}]\big(\fdot; A^{-2ka_1} - A^{-2(k+K)a_1},\ldots,  A^{-2ka_d} - A^{-2(k+K)a_d}\big).
}
\begin{Th}\label{Theorem41bis}
Assume~$\delta_0 \notin \M^\W$,~$\nu \in (0,\kappa(\M^\W,\theta_1, C_1))$. For any~$C \geq 1$ and any sufficiently large~$A$ there exists~$\eps$ such that for any~$k \in \{0\} \cup \N$,~$j\in \Z^d$, and~$f\in \W$ if the atom~$(k,j)$ is~$\eps$-flat and satisfies the concentration assumption~\eqref{ConcentrationAtom}, then 
\mlt{
\|f_{k+1}\|_{L_q(\tilde{v}_{k,j})} \leq A^{\alpha - 2\nu} \|f_k\|_{L_q(v_{k,j})},\\
\tilde{v}_{k,j} = \Heat[v_{k,j}]\Big(\fdot; \frac{A^{-2ka_1} - A^{-2(k+1)a_1}}{q},\ldots,  \frac{A^{-2ka_d} - A^{-2(k+1)a_d}}{q}\Big).
}
\end{Th}
\begin{St}\label{LqCubeArbitraryAtom}
Assume the atom~$(k,j)$ is~$\eps$-flat and fulfills the concentration assumption~\eqref{ConcentrationAtom}. Then,
\eq{
\|f_{k+K}\|_{L_1(\Heat[w_{k,j}](\fdot, A^{-2ka_1} - A^{-2(k+K)a_1},\ldots, A^{-2ka_d} - A^{-2(k+K)a_d}))}\lesssim A^{-\alpha k}\|f_{k}\|_{L_q(Q_{k,j})},
}
provided~$\eps$ is sufficiently small. The multiplication constant in this inequality is independent of~$A$,~$k$, and~$j$.
\end{St}
\begin{Cor}\label{CubeEstimateCorollary}
Assume~$\delta_0 \notin \M^\W$,~$\nu \in (0,\kappa(\M^\W,\theta_1, C_1))$. For any~$C > 1$ and any sufficiently large~$A$ there exists~$\eps > 0$ such that for any~$k \in \{0\} \cup \N$,~$j\in \Z^d$, and~$f\in \W$ if the atom~$(k,j)$ is~$\eps$-flat and satisfies the concentration assumption~\eqref{ConcentrationAtom}, then
\mlt{
\|f_{k+1}\|_{L_q(\tilde{v}_{k,j})} \lesssim A^{\alpha - 2\nu}\|f_{k}\|_{L_q(Q_{k,j})}, \\
 \tilde{v}_{k,j} = \Heat[v_{k,j}]\Big(\fdot; \frac{A^{-2ka_1} - A^{-2(k+1)a_1}}{q},\ldots,  \frac{A^{-2ka_d} - A^{-2(k+1)a_d}}{q}\Big).
}
The multiplication constant in this inequality is independent of~$A$,~$k$, and~$j$.
\end{Cor}
\begin{proof}
We combine Theorem~\ref{Theorem41bis} with the estimate
\eq{
A^{\alpha - 2\nu} \|f_k\|_{L_q(v_{k,j})} \lesssim A^{\alpha - 2\nu + \alpha k} \|f_{k+K}\|_{L_1(\Heat[w_{k,j}](\fdot, A^{-2ka_1} - A^{-2(k+K)a_1},\ldots, A^{-2ka_d} - A^{-2(k+K)a_d}))},
} 
which follows from Corollary~\ref{Cor44}, and complete the bounds with the help of Proposition~\ref{LqCubeArbitraryAtom} 
\end{proof}
\begin{Rem}
Considerations in the style of Lemmas~\ref{Lemma41} and~\ref{Lemma42} allow to replace~$\tilde{v}_{k,j}$ with simply~$v_{k,j}$ in the above corollary.
\end{Rem}


\subsection{Combinatorial part}\label{s42}
Fix a real parameter~$\theta_4 >d$ to be chosen later.
\begin{Def}
Let~$k \in \{0\} \cup \N$ and let~$j\in \Z^d$. Set
\mlt{\label{LocalSizeDefinitionFormula}
f^*_{k,j} = \|f_{k+K}\|_{L_1(\tilde{w}_{k,j})}, \\
\tilde{w}_{k,j} = \Heat[w_{k,j}](\fdot, A^{-2ka_1} - A^{-2(k+K)a_1}, A^{-2ka_2} - A^{-2(k+K)a_2}, \ldots, A^{-2ka_d} - A^{-2(k+K)a_d})
}
to be the local size of the function~$f$. 
\end{Def}
The choice of the weight~$\tilde{w}_{k,j}$ in~\eqref{LocalSizeDefinitionFormula} is suggested by the inequality
\eq{
\|f_{k}\|_{L_1(w_{k,j})} \leq f^*_{k,j} \leq \|f_{k+L}\|_{L_1(\Heat[w_{k,j}](\fdot, A^{-2ka_1} - A^{-2(k+L)a_1},  \ldots, A^{-2ka_d} - A^{-2(k+L)a_d}))},
}
which will be important since the quantities on the left and right appear in our central definition of~$\eps$-flat atoms, Definition~\ref{Convex/flat}.
\begin{Def}
Define the maximal function
\eq{\label{DefinitionOfMaximalFunction}
\Max_{k,j}f = \sup\limits_{i\in\Z^d} (1+|i - j|)^{-\theta_4} f_{k,i}^*.
}
The atom~$(k,j)$ is called \emph{saturated}, provided
\eq{
\Max_{k,j}f = f_{k,j}^*.
}
\end{Def}
In other words, an atom~$(k,j)$ is saturated, provided
\eq{
f_{k,i}^* \leq (1+|i - j|)^{\theta_4} f_{k,j}^*
}
for any~$i\in \Z^d$. There is a slight difference in our definition of a saturated atom and the one in~\cite{Stolyarov2022}, the latter one included an additional parameter~$K$, which seems unnecessary. We will shortly show that a saturated atom fulfills the concentration condition~\eqref{ConcentrationAtom}. 

Note that if~$f \in L_1(\R^d)$, the sequence~$\{f_{k,j}^*\}_j$ is uniformly bounded for any~$k$ fixed and tends to zero at infinity. Thus, the supremum in~\eqref{DefinitionOfMaximalFunction} is attained at some~$i$. For fixed~$j$, choose one of these `maximal' points and call it~$\vec{j}_k$. These choices define an oriented graph~$\Gamma_k$. 
\begin{Def}
The set of vertices of~$\Gamma_k$ is~$\Z^d$. We draw an arrow from~$\vec{j}_k$ to~$j$ for each~$j$, provided~$\vec{j}_k \ne j$. 
\end{Def}
In~\cite{Stolyarov2022}, the graph~$\Gamma_k$ was defined in a slightly different, more sophisticated way. The definition included a parameter~$\lambda$; now this parameter is redundant.

\begin{Le}\label{ShortPathsLemma}
The graph~$\Gamma_k$ does not contain oriented paths of length greater than one.
\end{Le}
The proof may be found in Lemma~$13$ of~\cite{Stolyarov2022}, now we will provide a sketch. One argues by contradiction: Assume there is a path~$i\to 0 \to j$. In such a case,
\alg{
\Max_{0,j}[f] = (1+|j|)^{-\theta_4} f_{0,0}^*;\\
f_{0,0}^* \leq (1+|i|)^{-\theta_4} f_{0,i}^*.
}
Combining these two bounds and using the triangle inequality, we arrive at~$\Max_{0,j}[f] < (1+|i-j|)^{-\theta_4} f_{0,i}^*$, which contradicts the definition of~$\Max_{0,j}[f]$.
\begin{Cor}\label{SaturatedOutgoingCor}
Assume a vertex~$j$ in~$\Gamma_k$ has an outgoing edge. Then, the corresponding atom~$(k,j)$ is saturated.
\end{Cor}
\begin{proof}
By Lemma~\ref{ShortPathsLemma},~$j$ does not have incoming edges. By the very definition, this means~$(k,j)$ is saturated.
\end{proof}
The graph~$\Gamma_k$ is a disjoint union of `stars': Each such star consists of its center, which is a saturated atom, and other vertices that are subordinate to the center; a star may consist of its center only.
\begin{Le}\label{SaturatedtoConcentrated}
Assume~$\theta_2 > \theta_4 + d$. There exists an absolute constant~$C$ such that any saturated atom in any graph~$\Gamma_k$ fulfills the concentration condition~\eqref{ConcentrationAtom}.
\end{Le}
This lemma is identical to Lemma~$15$ in~\cite{Stolyarov2022}. We will need yet another weighted lemma and postpone its proof till Subsection~\ref{SWeights} of the appendix. This lemma differs from previous lemmas about weights because here we measure the~$L_q$-norm of~$f_1$ on a cube of generation~$0$, not~$1$.
\begin{Le}\label{Lemma48}
Assume~$w$ satisfies the bound~\eqref{eq2216}. Then,
\eq{
\|f_1\|_{L_q(Q_{0,0})} \lesssim A^{\alpha} \|f_K\|_{L_1(w)}.
}
\end{Le}
\begin{Le}\label{SubordinationLemma}
 Assume there is an arrow~$j \to i$ in the graph~$\Gamma_k$. Then,
 \eq{\label{ShiftedEstimateOneLayer}
 \|f_{k+1}\|_{L_q(2Q_{k,i})} \lesssim A^{\alpha(k+1)} (1+|i-j|)^{-\theta_4} f_{k,j}^*.
 }
 The multiplication constant in this inequality is independent of~$A$,~$k$, and~$j$.
\end{Le}
\begin{proof}
Without loss of generality, we may assume~$k=0$. Then,
\eq{
 \|f_{1}\|_{L_q(2Q_{0,i})} \lesssim A^\alpha \|f_K\|_{L_1(\Heat[w_{0,i}](\fdot, 1 - A^{-2a_1K},1 - A^{-2a_2K}, \ldots, 1 - A^{-2a_dK}))} = A^{\alpha}f_{0,i}^*
}
by Lemma~\ref{Lemma48}; we have used Lemma~\ref{Lemma42} to justify that the weight
\eq{
\Heat[w_{0,0}](\fdot, 1 - A^{-2a_1K},1 - A^{-2a_2K}, \ldots, 1 - A^{-2a_dK}))
}
fulfills the bound~\eqref{eq2216}; here we apply the lemma to shifted weights. The obtained quantity~$A^{\alpha}f_{0,i}^*$ does not exceed~$A^{\alpha}\Max_{0,i} f$, which equals the right hand side of~\eqref{ShiftedEstimateOneLayer}:
\eq{
A^{\alpha}\Max_{0,i}f = A^{\alpha}(1+|i-j|)^{-\theta_4} f_{0,j}^*.
}
\end{proof}
\begin{Th}\label{CubeToCubeHorizontalTheorem}
Assume~$\delta_0 \notin \M^\W$. For any number~$\theta_5 \in (d,\theta_4)$ there exists~$\nu^* > 0$  with the following property. For any sufficiently large~$A$ there exists~$\eps$ such that if the atom~$(k,j)$ is~$\eps$-flat and there is an arrow~$j\to i$ in the graph~$\Gamma_k$, then
\eq{
\|f_{k+1}\|_{L_q(2Q_{k,i})} \lesssim A^{\alpha - \nu^*} (1+|i-j|)^{-\theta_5} \|f_k\|_{L_q(Q_{k,j})}.
}
The multiplication constant in this inequality is independent of~$A$,~$k$, and~$j$.
\end{Th}
\begin{proof}
Without loss of generality, let~$k=0$ and~$j=0$. The desired inequality will follow from the estimates
\alg{
\label{eq4214}\|f_1\|_{L_q(2Q_{0,i})} &\lesssim A^\alpha (1+|i|)^{-\theta_4} \|f_0\|_{L_q(Q_{0,0})};\\
\label{eq4215}\|f_1\|_{L_q(2Q_{0,i})} &\lesssim A^{\alpha - 2\nu} (1+|i|)^{\frac{\theta_3}{q}} \|f_0\|_{L_q(Q_{0,0})}.
}
Here~$\nu\in (0,\kappa(\M^\W, \theta_1,C_1))$. The inequality~\eqref{eq4214} follows from Proposition~\ref{LqCubeArbitraryAtom} and Lemma~\ref{SubordinationLemma}. The concentration condition required in Proposition~\ref{LqCubeArbitraryAtom} is implied by the fact that~$(0,0)$ is a saturated atom via Lemma~\ref{SaturatedtoConcentrated}. The atom~$(0,0)$ is saturated by Corollary~\ref{SaturatedOutgoingCor}. 

Let us prove~\eqref{eq4215}. We start with the bounds
\mlt{
\|f_1\|_{L_q(2Q_{0,i})}^q \Lsref{\text{\tiny Lem.\ref{Lemma42}}} (1+|i|)^{\theta_3}\int\limits_{2Q_{0,i}} |f_1(x)|^q \Heat[v]\Big(x, \frac{1 - A^{-2a_1}}{q}, \frac{1 - A^{-2a_2}}{q},\ldots, \frac{1 - A^{-2a_d}}{q}\Big)\,dx\\
 \leq (1+|i|)^{\theta_3} \|f_1\|^q_{L_q(\Heat[v](\fdot, \frac{1 - A^{-2a_1}}{q}, \frac{1 - A^{-2a_2}}{q}, \ldots, \frac{1 - A^{-2a_d}}{q}))},
}
which reduce~\eqref{eq4215} to Corollary~\ref{CubeEstimateCorollary}.
\end{proof}
Recall Definition~\ref{TreeStructure} of the tree structure on the set of atoms.
\begin{Def}\label{DefinitionOfGamma}
Define the oriented graph~$\GGamma$ that describes the vertical interaction of atoms. Set~$V(\GGamma) = \Fl$. There is an arrow from~$(k,j)$ to~$(k+1,i)$ if one of the two possibilities below occur:
\begin{itemize}
\item $(k,j)$ is a saturated atom and also the parent of~$(k+1,i)$ in~$\GenTree$;
\item Some other atom~$(k,j') \in \Fl$ is the parent of~$(k+1,i)$ in~$\GenTree$ and there is the arrow~$j\to j'$ in~$\Gamma_k$.
\end{itemize}
\end{Def}
At most one of the two possibilities in the definition above can occur by the construction of the graphs~$\Gamma_k$. Note that by Corollary~\ref{SaturatedOutgoingCor} only saturated atoms might have outgoing edges in~$\GGamma$. The graph~$\GGamma$ does not have cycles, so, this is a forest, i.e., a disjoint union of maximal  trees  with respect to inclusion. These trees may be also defined as the connectivity components of~$\GGamma$. Let the collection of all the obtained trees be denoted by~$\TREE$.
\begin{Le}\label{KidsBoundsWithConstantLemma}
Assume~$\delta_0 \notin \M^\W$. Let~$(k,j)\in\Fl$ be a saturated atom and let~$\{(k+1,i)\}_{i\in J}$ be the collection of all its children in~$\GGamma$, i.e., all~$i\in\Z^d$ such that~$(k,j) \to (k+1,i)$ in~$\GGamma$. Then,
\eq{\label{KidsBoundsWithConstant}
\|f_{k+1}\|_{L_q(\bigcup\limits_{i\in J} Q_{k+1,i})} \lesssim A^{\alpha - \nu^*} \|f_k\|_{L_q(Q_{k,j})},
} 
where~$\nu^* > 0$ is a constant independent of~$A$.
\end{Le}
\begin{proof}
Without loss of generality, assume~$k=0$ and~$j=0$. Let~$J'\subset \Z^d$ be the set of all points~$j'$ such that~$0 \to j'$ in~$\Gamma_0$. Then,
\eq{
\bigcup_{i\in J} Q_{1,i} \subset \Big(\bigcup_{j' \in J'} 2Q_{0,j'}\Big) \cup 2Q_{0,0},
}
and, by the triangle inequality,
\eq{
\|f_{1}\|_{L_q(\bigcup\limits_{i\in J} Q_{1,i})} \leq \sum\limits_{j'\in J'} \|f_{1}\|_{L_q(2Q_{0,j'})} + \|f_{1}\|_{L_q(2Q_{0,0})}.
}
The bound
\eq{
\|f_{1}\|_{L_q(2Q_{0,0})} \lesssim A^{\alpha - \nu^*} \|f_0\|_{L_q(Q_{0,0})} 
}
follows from Corollary~\ref{CubeEstimateCorollary}. To estimate the first sum, we employ Theorem~\ref{CubeToCubeHorizontalTheorem}:
\eq{
\sum\limits_{j'\in J'} \|f_{1}\|_{L_q(2Q_{0,j'})} \lesssim A^{\alpha - \nu^*} \sum\limits_{j'\in J'}(1+|j'|)^{-\theta_5} \|f_0\|_{L_q(Q_{0,0})} \lesssim A^{\alpha - \nu^*}  \|f_0\|_{L_q(Q_{0,0})} 
}
since~$\theta_5 > d$.
\end{proof}
The inequality~\eqref{KidsBoundsWithConstant} says there exists an absolute constant~$C$ (that might depend on the choice of~$\W$,~$\alpha$, or the parameters~$\theta_i$,~$i=1,\ldots, 5$, but not on~$A$ or~$\eps$) such that
\eq{\label{NonStrictInequality}
\|f_{k+1}\|_{L_q(\bigcup\limits_{i\in J} Q_{k+1,i})} \leq C A^{\alpha - \nu^*} \|f_k\|_{L_q(Q_{k,j})}
}
whenever all the assumptions of Lemma~\ref{KidsBoundsWithConstantLemma} are satisfied. By choosing sufficiently large~$A$, we deduce the bound
\eq{\label{StrictInequality}
\|f_{k+1}\|_{L_q(\bigcup\limits_{i\in J} Q_{k+1,i})} \leq A^{\alpha - \nu^*/2} \|f_k\|_{L_q(Q_{k,j})}
}
and also set~$\nu^{**} = \nu^*/2$. From now on, we fix the parameter~$A$. The latter inequality is suitable for induction, and we obtain the following lemma.
\begin{Le}\label{OneLayerOfTheTreeEstimate}
Let~$\Tree\in \TREE$ be a tree in~$\GGamma$ with the root at~$(k,j)$. Then,
\eq{
\|f_{m}\|_{L_q(\!\!\! \bigcup\limits_{(m,i)\in \Tree}\!\!\! Q_{m,i})} \leq A^{(\alpha - \nu^{**})(m-k)} \|f_k\|_{L_q(Q_{k,j})}
}
for any~$m \geq k$.
\end{Le}
Here and in what follows we use a slight abuse of notation: By~$(m,i)\in \Tree$ we mean~$(m,i)\in V(\Tree)$.
\begin{Def}\label{TreeLevels}
Let~$\Tree\in\TREE$ be a tree in~$\GGamma$ and let~$m\in \{0\}\cup \N$. Set
\eq{
\TS_m = \bigcup\limits_{(m,i)\in \Tree}\!\!\! Q_{m,i}.
}
\end{Def}
\begin{Cor}\label{TreeBoundLebesgueCorollary}
Let~$\Tree\in \TREE$ be a tree in~$\GGamma$ with the root at~$(k,j)$. Then,
\eq{\label{TreeBoundLebesgueCorollaryFormula}
\sum\limits_m A^{-\alpha m} \|f_{m}\|_{L_q(\TS_m)} \leq A^{-\alpha k} \|f_k\|_{L_q(Q_{k,j})}.
}
\end{Cor}
\begin{proof}
We use Lemma~\ref{OneLayerOfTheTreeEstimate} and the estimate for the geometric series:
\mlt{
\sum\limits_m A^{-\alpha m} \|f_{m}\|_{L_q(\TS_m)} \lesssim \sum\limits_{m \geq k} A^{-\alpha m} A^{(\alpha - \nu^{**})(m-k)} \|f_{k}\|_{L_q(Q_{k,j})}\\
 =\Big(\sum\limits_{m \geq k}A^{-\nu^{**} m}\Big) A^{-(\alpha - \nu^{**})k}\|f_k\|_{L_q(Q_{k,j})} \lesssim A^{-\alpha k} \|f_k\|_{L_q(Q_{k,j})}.
}
\end{proof}
This corollary says that the bound for the sum over a tree is reduced to the bound for the local quantity over its root. The lemma below describes the parents in the tree of all atoms (see Definition~\ref{TreeStructure}) of the roots of the trees in~$\TREE$. It is a direct consequence of definitions, the formal proof may be found in Lemma~$18$ in~\cite{Stolyarov2022}.
\begin{Le}
Let~$(k,j)$ be the root of a tree~$\Tree \in \TREE$, let also~$k \geq 1$. Either~$(k,j)$ is a child of a convex atom~$(k-1,j')$ in~$\GenTree$ or it is a child of some atom~$(k-1,j')$ in~$\GenTree$ that is subordinate to a convex atom~$(k-1,j^\uparrow)$ in~$\Gamma_{k-1}$.
\end{Le}
In the first case, when~$(k,j)$ is a child of a convex atom,
\mlt{
 A^{-\alpha k}\|f_k\|_{L_q(Q_{k,j})} \Lsref{\text{\tiny Cor.~\ref{Cor44}}} \|f_{k+K}\|_{L_1(\tilde{w}_{k-1,j'})} \Lsref{\scriptscriptstyle (k-1,j')\in \CO} \Big(\|f_{k+K}\|_{L_1(\tilde{w}_{k-1,j'})} - \|f_{k-1}\|_{L_1(w_{k-1,j'})}\Big),\\
  \tilde{w}_{k-1,j'} = \Heat[w_{k-1,j'}]\Big(\fdot, A^{-2(k-1)a_1} - A^{-2(k-1+L)a_1}, \ldots, A^{-2(k-1)a_d} - A^{-2(k-1+L)a_d}\Big).
}
Similar bounds lead to the estimate
\mlt{\label{HugeBoundSecondCase}
A^{-\alpha k}\|f_k\|_{L_q(Q_{k,j})} \\
\Lsref{\text{\tiny Cor.~\ref{Cor44}}} \|f_{k-1+K}\|_{L_1(\Heat[w_{k-1,j'}](\fdot; A^{-2(k-1)a_1} - A^{-2(k-1+K)a_1}, \ldots, A^{-2(k-1)a_d} - A^{-2(k-1+K)a_d}) )} \\
\Lref{\scriptscriptstyle j^\uparrow \stackrel{\scriptscriptstyle\Gamma_{k-1}}{\longrightarrow} j'}(1+|j^\uparrow-j'|)^{-\theta_4} \|f_{k-1+K}\|_{L_1(\Heat[w_{k-1,j^\uparrow}](\fdot; A^{-2(k-1)a_1} - A^{-2(k-1+K)a_1}, \ldots, A^{-2(k-1)a_d} - A^{-2(k-1+K)a_d}) )} \\
\Lsref{\scriptscriptstyle (k-1,j^\uparrow)\in \CO}  (1+|j^\uparrow-j'|)^{-\theta_4} \Big(\|f_{k+K}\|_{L_1(\tilde{w}_{k-1,j^\uparrow})} - \|f_{k-1}\|_{L_1(w_{k-1,j^\uparrow})}\Big),\\
  \tilde{w}_{k-1,j^\uparrow} = \Heat[w_{k-1,j^\uparrow}]\Big(\fdot, A^{-2(k-1)a_1} - A^{-2(k-1+L)a_1}, \ldots, A^{-2(k-1)a_d} - A^{-2(k-1+L)a_d}\Big)
}
in the second case. We will informally assume that the first bound is a particular case of the second one, i.e., think that a saturated atom is subordinate to itself.

\begin{proof}[Proof of Theorem~\ref{MainTheoremBesovLorentzScale}]
According to Remark~\ref{DiscretizationRemark} and Lemma~\ref{TruncationLemma}, it suffices to prove the bound
\eq{
\sum\limits_{k \geq 0} A^{-\alpha k} \|f_k\|_{L_{q}} \lesssim \|f\|_{L_1},\qquad f\in \W,
}
for some number~$A > 1$ independent of~$f$. This number is defined by Theorem~\ref{Theorem41}, Proposition~\ref{Prop45}, and the passage from~\eqref{NonStrictInequality} to~\eqref{StrictInequality}. We also choose sufficiently small~$\eps$ after we have fixed~$A$, the smallness of~$\eps$ is also specified in Theorem~\ref{Theorem41} and Proposition~\ref{Prop45}. Let now~$\Omega_k$ be the union of the~$Q_{k,j}$ that correspond to convex atoms~$(k,j)$ of generation~$k$, as defined in~\eqref{OmegakDefinition}. By the triangle inequality,
\eq{\label{eq4231}
\sum\limits_{k \geq 0} A^{-\alpha k} \|f_k\|_{L_{q}} \leq \sum\limits_{k \geq 0} A^{-\alpha k} \|f_k\|_{L_{q}(\Omega_k)} + \sum\limits_{k \geq 0} A^{-\alpha k} \|f_k\|_{L_{q}(\R^d \setminus \Omega_k)}.
}
The first sum is bounded with~$\|f\|_{L_1}$ in Corollary~\ref{ConvexAtomsEstimates}. To bound the second sum, we construct the graph~$\Gamma$ and split it into trees~$\Tree$; note that we have already fixed~$A$ and~$\eps$. Each tree~$\Tree$ generates its own collection of sets~$\{\TS_k\}_k$ via Definition~\ref{TreeLevels}. Since each flat atom belongs to some tree, we have the decomposition
\eq{
\R^d \setminus \Omega_k  = \bigcup_{\Tree\in \TREE} \TS_k,\qquad k\in \{0\}\cup \N.
}
We decompose the second sum on the right hand side of~\eqref{eq4231} further using the triangle inequality,
\eq{
\|f_k\|_{L_{q}(\R^d \setminus \Omega_k)} \leq \sum\limits_{\Tree\in \TREE} \|f_k\|_{L_{q}(\TS_k)}. 
}
Note that we cannot decompose further, see Subsection~\ref{NecessityForNonSplit} in the appendix. Let~$(k(\Tree), j(\Tree))$ be the root of a tree~$\Tree$. We interchange the orders of summation and collect the previous estimates:
\mlt{
\sum\limits_{k \geq 0} \sum\limits_{\Tree \in \TREE} A^{-\alpha k}\|f_k\|_{L_{q}(\TS_k)}\\
 = \sum\limits_{\Tree\in \TREE} \sum\limits_{m \geq k(\Tree)} A^{-\alpha m} \|f_m\|_{L_{q}(\TS_m)} \Lsref{\text{\tiny Cor.\ref{TreeBoundLebesgueCorollary}}} \sum\limits_{\Tree\in \TREE}A^{-\alpha k(\Tree)} \|f_{k(\Tree)}\|_{L_q(Q_{k(\Tree),j(\Tree)})}\\
\LseqrefTwo{TreeBoundLebesgueCorollaryFormula}{HugeBoundSecondCase} \sum\limits_{\Tree\in \TREE} (1+|j^\uparrow (\Tree)-j'(\Tree)|)^{-\theta_4} \Big(\|f_{k(\Tree)+K}\|_{L_1(\tilde{w}_{k(\Tree)-1,j^\uparrow(\Tree)})} - \|f_{k(\Tree)-1}\|_{L_1(w_{k(\Tree)-1,j^\uparrow(\Tree)})}\Big).
}
In the latter formula, we use the convention that in the case~$k(\Tree) = 0$, we replace the difference between weighted~$L_1$ norms with simply~$\|f_{K}\|_{L_1(w_{0,j^\uparrow(\Tree)})}$. We wish to bound the latter sum with the right hand side of~\eqref{DefOfTildeWeights}. This will follow, provided we show for any~$k$ and~$i$ that
\eq{
\sum\limits_{\genfrac{}{}{0pt}{-2}{\Tree\colon k(\Tree)=k,}{j^\uparrow(\Tree)=i}} (1+|j^\uparrow (\Tree)-j'(\Tree)|)^{-\theta_4}  \lesssim 1.
}
This estimate is true since each atom~$(k,j')$ has at most~$C_dA^{d}$,~$C_d$ being a dimensional constant, children, and~$\theta_4 > d$.
\end{proof}


\subsection{Reflection and possible further development}\label{s43}
In this subsection, we provide a context for several aspects of the proof. The reasoning is more informal.

\paragraph{Relationship between flat atoms and the polar decomposition of charges.} Let~$\mu$ be a finite charge on~$\R^d$. The Besicovitch differentiation theorem says that since~$\mu$ is absolutely continuous with respect to its total variation~$|\mu|$, there exists the density function~$\vn \colon \R^d \to S^{\ell-1}$ attaining its values in the unit sphere and such that~$\mu = \vec{n} |\mu|$. What is more, 
\eq{
\lim\limits_{r\to 0} \Big(|\mu|(B_r(x))\Big)^{-1}\int\limits_{B_r(x)} \big|\vn(x) - \vn(y)\big|\,d|\mu|(y) = 0
}
for~$|\mu|$-almost all~$x\in \supp \mu$; see Remark $2.15(3)$ in~\cite{Mattila1995} for the Besicovitch differentiation theorem and the latter limit relation. In other words,
\eq{
\int\limits_{B_r(x)} \big|\vn(x) - \vn(y)\big|\,d|\mu|(y) \leq \eps |\mu|(B_r(x))
}
for sufficiently small~$r$, whenever~$\eps$ and~$x$ are fixed. By the triangle inequality, the left hand side is bounded away from zero by
\mlt{
\Big|\int\limits_{B_r(x)}\vn(x)\,d|\mu|(y) - \int\limits_{B_r(x)}\vn(y)\,d|\mu|(y)\Big|\\
 = \Big|\vn(x) |\mu|(B_r(x)) - \mu(B_r(x))\Big| \geq |\mu|(B_r(x)) - |\mu(B_r(x))|.
}
In particular, 
\eq{\label{FlatnessBalls}
|\mu|(B_r(x)) - |\mu(B_r(x))| \leq \eps |\mu|(B_r(x))
}
for~$|\mu|$-almost every~$x$ and all~$\eps > 0$, provided~$r$ is sufficiently small. 

Now let us turn to~$\eps$-flat atoms. Until the end of this subsection we work with the classical isotropic homogeneity, i.e.,~$a = (1,1,\ldots, 1)$. Recall that in this case~\eqref{DefinitionOfK} says~$K=2$ and~$L=3$. Therefore,
\mlt{
\|f_{k+3}\|_{L_1(\tilde{w}_{k,j})} - \|f_k\|_{L_1(w_{k,j})} = \int\limits_{\R^d}  |f_{k+3}(x) | \heat [w_{k,j}]\big(x; A^{-2k} - A^{-2k-6}\big)\,dx - \int\limits_{\R^d}  |f_{k}(x) | w_{k,j}(x)\,dx\\
\Eeqref{ReproducingThefk} \int\limits_{\R^d} \Big(\heat\big[|f_{k+3}|\big]\big(x; A^{-2k} - A^{-2k-6}\big) - \Big| \heat\big[f_{k+3}\big]\big(x; A^{-2k} - A^{-2k-6}\big)\Big|\Big)w_{k,j}(x)\,dx.
}
Thus, Definition~\ref{Convex/flat} of a flat atom says that on average 
\mlt{
\heat\big[|f_{k+3}|\big]\big(x; A^{-2k} - A^{-2k-6}\big) - \Big| \heat\big[f_{k+3}\big]\big(x; A^{-2k} - A^{-2k-6}\big)\Big|\\
 \leq \eps \Big| \heat\big[f_{k+3}\big]\big(x; A^{-2k} - A^{-2k-6}\big)\Big| \leq \eps\heat\big[|f_{k+3}|\big]\big(x; A^{-2k} - A^{-2k-6}\big),
}
when~$x$ is close to~$A^{-k}j$. This might be thought of as a Gaussian version of~\eqref{FlatnessBalls}. On the other hand, we have used several formalizations of the reverse principle: The presence of an~$\eps$-flat atom~$(k,j)$ ensures that in a neighborhood of~$A^{-k}j$ the charge or function in question is close to a rank-one measure on the scale~$A^{-k}$. To summarize the informal discussion, the flat/convex decomposition allows to split a charge or a function into a sum of approximate rank-one measures and elementary `convex' atoms that are easy to analyze. The difficulty comes from the fact that the condition~$f\in \W$ is unfriendly to standard splitting procedures such as truncation or multiplication, so we need to work with weighted norms rather than split~$f$ into parts directly.

\paragraph{The role of dimension.} 
Let~$\mu$ be a locally finite charge on~$\R^d$. Its lower Hausdorff dimension at~$x\in \supp \mu$ is defined as
\eq{
\ldH \mu(x) = \varliminf_{r\to 0} \frac{\log |\mu|(B_r(x))}{\log r}.
}
The lower Hausdorff dimension of a charge is given by
\eq{
\ldH \mu = \sup \Set{\gamma \geq 0}{\ldH \mu (x) \geq \gamma \text{ for $|\mu|$-almost all } x}.
}
This quantity measures the `maximal singularity' of a charge, see Chapter~$10$ in~\cite{Falconer1997} for equivalent definitions and properties of this notion. It appears that larger values of~$\nu$ in Proposition~\ref{StrengtheningOfSemiinvariant} lead to better lower bounds for~$\ldH \mu$. Here we also prefer to work with the classical isotropic homogeneity~$a = (1,1,\ldots,1)$.
\begin{St}\label{GorbunovProposition}
Fix~$\gamma \in (0,d)$ and~$p > 1$. Let~$\mu$ be a finite measure on~$\R^d$ such that
\eq{
\|\heat [\mu](\fdot;t)\|_{L_p(\R^d)} \lesssim t^{- \frac{(d-\gamma)(p-1)}{2p}}
}
for all~$t > 0$ sufficiently small. Then,~$\ldH \mu \geq \gamma$.
\end{St}
We will not prove this proposition\footnote{Proposition~\ref{GorbunovProposition} was suggested by Leonid Gorbunov.}.  The main instrument of the proof is a strengthening of the Frostman lemma going back to~\cite{StolyarovWojciechowski2014}. See~\cite{Dobronravov2024} for more details on these techniques. A version of Proposition~\ref{GorbunovProposition} with~$p=\infty$ was used in~\cite{Stolyarov2023} to obtain bounds for the lower Hausdorff dimension of charges~$\mu$ with Fourier constraints. We wish to prove the following result, which does not provide any explicit bounds on the dimension.
\begin{Th}\label{DimensionEstimate}
Let~$\W \subset \Sw'(\R^d,\R^\ell)$ be a closed translation and dilation invariant subspace. Assume~$\W$ does not contain vectorial delta measures. Then, there exists~$\eta > 0$ such that~$\ldH \mu \geq \eta$ for any charge~$\mu \in \W \cap \M (\R^d,\R^\ell)$. 
\end{Th}
\begin{proof}
Let~$\mu \in \W \cap \M (\R^d,\R^\ell)$ be an excessively singular charge:~$\ldH \mu < \eta$. Our aim is to obtain non-trivial lower bounds for~$\eta$, a curious reader may look up~\eqref{LowerBoundForEta} below. Theorem~$1.2$ and Proposition~$2.2$ in~\cite{Stolyarov2023} say that in such a case there exists a sequence of shifts and dilations of~$\mu$ that converge in~$\Sw'(\R^d, \R^\ell)$ to a non-trivial rank-one charge~$a\otimes \m\in \W$ that satisfies the following requirements:
\begin{enumerate}[1)]
\item $\m \geq 0$;
\item $\ldH\m (0) < \eta$;
\item $\m(B_R(0))\lesssim R^\eta$ for~$R > 1$.
\end{enumerate}
The latter condition, in particular, yields the finiteness of the integral
\eq{
\int\limits_{\R^d} (1+|x|)^{-2\eta}\,d\m(x).
}
This might be interpreted as~$d\m \in L_1((1+|\fdot|)^{-2\eta})$. Lemma~\ref{Lemma44} then implies
\eq{
\int\limits_{\R^d} (\heat[\m](x,1))^p (1+|x|)^{-2p\eta}\,dx < \infty.
}
By Proposition~\ref{StrengtheningOfSemiinvariant}, we then have
\eq{\label{eq4312}
\|\heat [\m](\fdot;t)\|_{L_p((1+|x|)^{-2p\eta})} \lesssim t^{- \frac{d(p-1)}{2p} + \nu},
}
where~$\nu \in (0 ,\kappa(\M^\W, 2p\eta, 1))$; the parameter~$1$ in the latter formula appears from~\eqref{TypicalWeightSmoothness}.

Now we wish to relate the heat extension bounds with the second item,~$\ldH\m (0) < \eta$. This local dimension bound, in particular, leads to the estimate~$\m(B_{r_j}(0)) \geq r_j^\eta$ for some sequence of radii~$r_j$ tending to zero. Thus,
\eq{
\heat[\m](x,r^2) = \big(4\pi r^2\big)^{-d/2}\int\limits_{\R^d} e^{-\frac{|x-y|^2}{4r^2}}\,d\m(y) \gtrsim r^{\eta -d},\qquad x\in B_{r}(0),
}
here~$r$ is one of the~$r_j$. Therefore,
\eq{
\|\heat [\m](\fdot;r^2)\|_{L_p((1+|x|)^{-2p\eta})} \gtrsim r^{\eta -d + d/p} = r^{-d\frac{p-1}{p} + \eta}.
}
Since~$r_j$ is arbitrarily small, this bound together with~\eqref{eq4312} implies
\eq{
-d\,\frac{p-1}{p} + 2\nu \leq -d\,\frac{p-1}{p} + \eta,
}
therefore,~$2\nu \leq \eta$. Fix~$p=2$ and note that~$\kappa(\M^\W, 2p\eta, 1) \geq \kappa(\M^\W, 4d, 1)$ by simple inclusions. Consequently,
\eq{\label{LowerBoundForEta}
\eta \geq 2 \kappa(\M^\W, 4d, 1).
}
\end{proof}
The bounds obtained in the proof above are rough. The argument is not sensitive to the choice of~$p$. The search for sharp bounds in a similar discrete problem from~\cite{ASW2021} included optimization with respect to~$p$ (the optimal~$p$ does not exist, and the optimal value is obtained as~$p$ approaches~$1$). Seemingly, the starting point for transferring the reasoning from~\cite{ASW2021} would be to obtain the strengthening of Theorem~$1.2$ in~\cite{Stolyarov2023} that, with the same assumptions, leads to the stronger conclusion~$\ldH\m < \eta$. Note that there are other approaches to the dimension problem, see, e.g.,~\cite{ARDHR2019} and~\cite{DePhilippisRindler2016}.

Theorem~\ref{DimensionEstimate} suggests the interpretation of the material of Section~\ref{S3}: If an invariant cone of measures does not contain delta measures, then there are non-trivial lower dimensional bounds for the measures that are elements of the cone.  
A similar effect had already been observed in~\cite{RoginskayaWojciechowski2006} in a similar setting of Fourier constrained spaces described in Subsection~\ref{sD2} of the appendix.

\paragraph{Relationship with~$\mathrm{DS}_\beta$ spaces.} 
The target space~$\R^\ell$ plays an important role in our study: For all interesting examples of~$\W$ that do not contain delta measures, we have~$\ell \geq 2$. It is therefore desirable to design spaces of scalar functions, measures, and distributions, that somehow have the properties similar to those provided by Proposition~\ref{StrengtheningOfSemiinvariant}. The paper~\cite{SS2024} suggests a scale of spaces~$\mathrm{DS}_\beta(\R^d)$. The space depends on a real parameter~$\beta\in [0,d]$. The definition is slightly involved and mimics the atomic definition of the real Hardy class~$\mathrm{H}_1$. We have~$\mathrm{DS}_0 = \M(\R^d)$,~$\mathrm{DS}_d = \mathrm{H}_1$, and the other spaces interpolate these two endpoints. The important property that distinguishes different spaces inside the scale is that~$\ldH \mu \geq \beta$ whenever~$\mu \in \mathrm{DS}_\beta$, and, moreover, for any~$\beta \in [0,d]$ there are plenty of measures~$\mu \in \mathrm{DS}_\beta$ for which~$\ldH \mu = \beta$. 

The definition originated from~\cite{HernandezSpector2024}, where it was proved that~$\W\cap \M(\R^d,\R^d)$ embeds into~$\mathrm{DS}_1$, where~$\W$ is given by divergence-free vector fields:
\eq{
\W = \Set{g\in \Sw'(\R^d,\R^d)}{\mathrm{div}\, g = 0}.
}
By the embedding here we mean that each coordinate of a solenoidal charge is an element of~$\mathrm{DS}_1$.
Later, in~\cite{SSS2026}, it was shown that in the case
\eq{
\W = \Set{\nabla f}{f \in \Sw'(\R^d)}
}
we have~$\W\cap \M(\R^d,\R^d)$ embedded into~$\mathrm{DS}_{d-1}$. The natural question is, given some translation and dilation invariant closed subspace~$\W \subset \Sw'(\R^d,\R^\ell)$, what is the largest possible~$\beta$ such that~$\W \cap \M(\R^d,\R^\ell) \hookrightarrow \mathrm{DS}_\beta$? If~$\W$ is defined by a Fourier constraint as in Subsection~\ref{sD2} in the appendix, can this optimal~$\beta$ be expressed explicitly in terms of the corresponding function~$\Omega$?
\appendix

\section{Technical lemmas}
\subsection{General facts}\label{AppendixA}
\begin{St}\label{ConvolutionTranslationInvariant}
Let~$\W$ be a closed translation invariant linear subspace of~$\Sw'(\R^d,\R^\ell)$. For any function~$\varphi\in \Sw(\R^d)$ and any distribution~$\zeta \in \W$, we have~$\varphi*\zeta\in \W$. 
\end{St}
\begin{Le}\label{WeaklyConvergentMeasuresAndDistribution}
Let~$\{\mu_n\}_n$ be a sequence of charges supported in a compact set~$K\subset \R^d$. Assume it converges in the weak-$*$ sense to a charge~$\mu$. In such a case,~$\mu_n*\varphi \to \mu*\varphi$ in the Schwartz class topology, provided~$\varphi\in \Sw(\R^d)$.
\end{Le}
\begin{proof}
We need to verify the limit relation
\eq{
\sup\limits_{x\in \R^d}(1+|x|)^N\big|\mu*\varphi(x) - \mu_n*\varphi(x)\big| \to 0\quad \text{as}\ n\to \infty;
}
similar limit relations for higher derivatives reduce to this one by replacing~$\varphi$ with the corresponding derivative. Here~$N\in \N$ is an arbitrary number. Since the charges~$\mu_n$ have uniformly bounded variations, the numerical sequences in question are bounded. Thus, it suffices to verify, given an arbitrary~$R$, that
\eq{
\sup\limits_{|x|\leq R}\big|\mu*\varphi(x) - \mu_n*\varphi(x)\big| \to 0\quad \text{as}\ n\to \infty.
}
For any~$x\in \R^d$, we have~$\mu_n*\varphi(x) \to \mu*\varphi(x)$. What is more, the function family~$\{\mu_n*\varphi\}_n$ is uniformly continuous:
\eq{
\big|\mu_n*\varphi(x) - \mu_n*\varphi(y)\big|  = \Big|\int\limits_{\R^d}\varphi(x-z)\,d\mu_n(z) - \int\limits_{\R^d}\varphi(y-z)\,d\mu_n(z)\Big| \lesssim \omega(\varphi; |x-y|)\|\mu\|;
}
the symbol~$\omega$ denotes the classical modulus of continuity. Therefore, the application of the Arzel{\`a}--Ascoli theorem finishes the proof.
\end{proof}
\begin{proof}[Proof of Proposition~\ref{ConvolutionTranslationInvariant}]
Consider the case~$\varphi\in C_0^\infty(\R^d)$ first. In this case, we may approximate~$\varphi$ by a sequence of charges
\eq{
\mu_n = \sum\limits_{j=1}^n a_j\delta_{x_j}
}
in the sense that~$\supp \mu_n\subset \supp \varphi$ and~$\mu_n$ tends to~$\varphi$ in the weak-$*$ topology. Since~$\W$ is translation invariant and~$\zeta\in \W$, we also have~$\mu_n*\zeta \in \W$. Thus, it remains to justify the limit relation
\eq{
\mu_n*\zeta \to \varphi*\zeta
}
in the topology of~$\Sw'(\R^d,\R^\ell)$. This follows from Lemma~\ref{WeaklyConvergentMeasuresAndDistribution} and the definition of topology in~$\Sw'(\R^d,\R^\ell)$.

To reduce the case of a general~$\varphi\in \Sw(\R^d)$ to the already considered, we may simply refer to the density of smooth compactly supported functions in~$\Sw(\R^d)$.
\end{proof}
\begin{proof}[Proof of Lemma~\ref{SeveralDilations}]
The value of the left hand side of~\eqref{SeveralDilationsFormula}, evaluated at~$(x,\vec{t}\,)$, is
\eq{
\prod\limits_{j=1}^d (4\pi t_j)^{-\frac12}\int\limits_{\R^d} e^{-\sum_{1}^d \frac{|x_j - y_j|^2}{4t_j}} \lambda^{-d} f\Big(\lambda^{-a_1}y_1,\lambda^{-a_2}y_2, \ldots, \lambda^{-a_d}y_d\Big)\,dy.
}
The evaluation of the right hand side at the same point equals
\eq{
\lambda^{-d} \prod\limits_{j=1}^d (4\pi \lambda^{-2a_j}t_j)^{-\frac12}\int\limits_{\R^d} e^{-\sum_{1}^d \frac{|\lambda^{-a_j}x_j - z_j|^2}{4\lambda^{-2a_j}t_j}} f(z)\,dz.
}
The substitution~$z = \Dil_{\lambda^{-1}} y$ yields their coincidence.
\end{proof}

\subsection{About Lorentz spaces}\label{AppendixLorentz}
The reader may find a good introduction to Lorentz spaces~$L_{q,r}$ in Subsection~$1.4.2$ of~\cite{Grafakos2008}. Though we will need the case~$r=1$ only, we prefer to keep the parameter~$r$ for a while. The norm in the space~$L_{q,r}$ is defined via the formula
\eq{\label{LorentzDef}
\|h\|_{L_{q,r}(\Omega)} = \bigg(q\int\limits_0^{\infty} t^r\big|\set{x\in \Omega}{|h(x)| \geq t}\big|^{\frac{r}{q}}\,\frac{dt}{t}\bigg)^{\frac1r}.
}
To be precise, this quantity does not define a norm in general. However, it is equivalent to a norm when~$q \in (1,\infty)$. Thus, there exists a constant~$C$ such that for any collection of functions~$h_1,h_2, h_3, \ldots$, the triangle inequality holds true:
\eq{\label{TriangleLorentz}
\Big\|\sum\limits_j h_j\Big\|_{L_{q,r}} \leq C \sum\limits_j \|h_j\|_{L_{q,r}}.
}

This may also be stated as a convolutional inequality
\eq{\label{ConvolutionLorentz}
\|f*g\|_{L_{q,r}}\lesssim \|f\|_{L_{q,r}}\|g\|_{L_1}.
}
When~$r=q$, the Lorentz space~$L_{q,q}$ coincides with~$L_q$ since the quantity~\eqref{LorentzDef} in this case is equal to~$\|h\|_{L_q(\Omega)}$ via the layer cake representation formula.

We will also need that the Lorentz space is an interpolation space between two Lebesgue spaces. In particular,
\eq{\label{LorentzMultiplicative}
\|f\|_{L_{q,r}} \lesssim \|f\|_{L_{q_1}}^{1-\theta}\|f\|_{L_{q_2}}^{\theta},\qquad \frac{1-\theta}{q_1} + \frac{\theta}{q_2} = \frac{1}{q},\quad 1 < q_1 < q < q_2 < \infty;
}
this follows from the interpolation relation
\eq{\label{LorentzInterpolation}
L_{q,r} = (L_{q_1},L_{q_2})_{\theta,r},
}
see Theorem~$5.2.1$ in~\cite{BerghLofstrom1976}, and Theorem~$3.1.2$ of the same book.

We also often encounter the dilations of functions and need to keep track of their Lorentz norm. The following formula will be useful:
\eq{\label{LorentzDilation}
\big\|\Dil_t[f]\big\|_{L_{q,r}} = t^{-\frac{q-1}{q}d} \|f\|_{L_{q,r}} =  t^{-\alpha} \|f\|_{L_{q,r}},\qquad t > 0.
}

\subsection{Anisotropic Besov--Lorentz spaces}\label{Besov--Lorentz}
The scale of Besov--Lorentz spaces is a straightforward generalization of a more classical Besov scale. These spaces fall into a more general setting of abstract Besov spaces introduced in~\cite{Peetre1976} and later developed in~\cite{HedbergNetrusov2007}. For the study of the specific Besov--Lorentz scale, see the recent paper~\cite{SeegerTrebels2019}. We provide a case study needed for our purposes; note that the three papers above work with isotropic spaces.

Fix some function~$\psi \in \Sw(\R^d)$ and define the seminorm
\eq{
\|f\|_{\dot{B}_{q,r}^{\beta,s}} = \Big(\sum\limits_{k\in\Z} A^{s\beta k}\big\|f*(\psi_k - \psi_{k-1})\big\|_{L_{q,r}}^s\Big)^{\frac{1}{s}},\qquad \psi_k = \Dil_{A^{-k}} [\psi].
}
Here~$A > 1$ is an auxiliary parameter;~$\beta\in \R$ and~$s \in [1,\infty)$ are the smoothness parameters of the norm;~$q\in [1,\infty]$ and~$r\in [1,\infty]$ are summability parameters. Note that the seminorm above crucially depends on the choice of the anisotropy~$a$. We restrict our considerations to the case~$s=1$:
\eq{\label{DefOfBesovNorm}
\|f\|_{\dot{B}_{q,r}^{\beta,1}} = \sum\limits_{k\in\Z} A^{\beta k}\big\|f*(\psi_k - \psi_{k-1})\big\|_{L_{q,r}}.
}
The choice~$r=q$ defines the anisotropic homogeneous Besov space~$\dot{B}^{\beta,1}_q$ since in this case the Lorentz space reduces to the classical Lebesgue space~$L_q$.

For~$q > 1$, the triangle inequality~\eqref{TriangleLorentz} yields the continuous embedding
\eq{\label{BesovToLorentz}
\dot{B}_{q,r}^{0,1} \hookrightarrow L_{q,r};
}
the details are similar to the proof of~\eqref{eq224} in Lemma~\ref{TriangleInequalityLemma}.

Let us assume now that~$\hat{\psi}$ is compactly supported and that~$\hat{\psi} \equiv 1$ in the neighborhood of the origin. Our first aim is to establish that our definition of the norm in~$\dot{B}_{q,r}^{\beta,1}$ is independent of the choice of~$\psi$. Let~$\varphi$ be another function or distribution. We ask  whether~$\varphi$ defines an equivalent norm via a similar formula; to be more precise, we wish to obtain some good conditions on~$\varphi$ that are sufficient for the inequality
\eq{\label{BesovLorentzEquiv}
\sum\limits_{k\in\Z} A^{\beta k}\big\|f*(\varphi_k - \varphi_{k-1})\big\|_{L_{q,r}} \lesssim\|f\|_{\dot{B}_{q,r}^{\beta,1}};
}
the dilated functions~$\varphi_k$ are defined accordingly,~$\varphi_k = \Dil_{A^{-k}}[\varphi]$.
\begin{Le}\label{BesovLocalizationLemma}
If there are numbers~$\beta_+$ and~$\beta_-$ such that~$\beta_+ < \beta < \beta_-$ and that the inequality
\eq{\label{BesovLorentzEquivBase}
\big\|f*(\varphi - \varphi_{-1})\big\|_{L_{q,r}} \lesssim \sum\limits_{\ell \geq 0} A^{\beta_+ \ell}\big\|f*(\psi_\ell - \psi_{\ell-1})\big\|_{L_{q,r}} + \sum\limits_{\ell < 0} A^{\beta_- \ell}\big\|f*(\psi_\ell - \psi_{\ell-1})\big\|_{L_{q,r}}
}
holds true, then,~\eqref{BesovLorentzEquiv} holds true as well.
\end{Le}
\begin{proof}
By dilation invariance~\eqref{LorentzDilation}, the assumptions yield the bound
\mlt{
A^{\beta k}\big\|f*(\varphi_k - \varphi_{k-1})\big\|_{L_{q,r}}\\ \lesssim \sum\limits_{\ell \geq k} A^{\beta_+ (\ell - k) + \beta k}\big\|f*(\psi_\ell - \psi_{\ell-1})\big\|_{L_{q,r}} + \sum\limits_{\ell < k} A^{\beta_- (\ell - k) + \beta k}\big\|f*(\psi_\ell - \psi_{\ell-1})\big\|_{L_{q,r}}.
}
If we sum these inequalities over~$k$, interchange the order of summation on the right hand side, and estimate the geometric series by its major term, we get the desired bound~\eqref{BesovLorentzEquiv}:
\mlt{
\sum\limits_{k\in\Z}\bigg(\sum\limits_{\ell \geq k} A^{\beta_+ (\ell - k) + \beta k}\big\|f*(\psi_\ell - \psi_{\ell-1})\big\|_{L_{q,r}} + \sum\limits_{\ell < k} A^{\beta_- (\ell - k) + \beta k}\big\|f*(\psi_\ell - \psi_{\ell-1})\big\|_{L_{q,r}}\bigg)\\ 
= \sum\limits_{\ell \in \Z} \big\|f*(\psi_\ell - \psi_{\ell-1})\big\|_{L_{q,r}} \Big(\sum\limits_{k \leq \ell} A^{\beta_+\ell + k(\beta - \beta_+)} + \sum\limits_{k > \ell} A^{\beta_-\ell + k(\beta - \beta_-)}\Big)\\
\lesssim \sum\limits_{\ell \in \Z} A^{\beta\ell}\big\|f*(\psi_\ell - \psi_{\ell-1})\big\|_{L_{q,r}}.
}
\end{proof}
When does the bound~\eqref{BesovLorentzEquivBase} hold true? We list several simple answers.

\begin{enumerate}
\item Assume~$\varphi$ is a summable function,~$\hat{\varphi} \equiv 1$ in a neighborhood of the origin, and~$\supp \hat{\varphi}$ is compact. Then,~\eqref{BesovLorentzEquivBase} is true, and, in fact, one needs only a finite number of summands on the right hand side. Indeed, we may write 
\eq{\label{TelescopingLP}
f*(\varphi - \varphi_{-1}) = \sum\limits_{\ell} f*(\varphi - \varphi_{-1})*(\psi_{\ell} - \psi_{\ell - 1})
}
since the sum on the right is finite, and notice that the~$L_1$-norm of the function~$\varphi - \varphi_{-1}$ is finite, which yields
\eq{
\Big\| f*(\varphi - \varphi_{-1})*(\psi_{\ell} - \psi_{\ell - 1})\Big\|_{L_{q,r}}\lesssim \Big\| f*(\psi_{\ell} - \psi_{\ell - 1})\Big\|_{L_{q,r}}
}
by~\eqref{ConvolutionLorentz}. In particular, the definition~\eqref{DefOfBesovNorm} of the Besov--Lorentz norm does not depend on the choice of~$\psi$: Any choice of~$\psi$ from the class of Schwartz functions with Fourier transform compactly supported and equal to~$1$ in a neighborhood of the origin leads to an equivalent norm via~\eqref{DefOfBesovNorm}.

\item Assume~$\varphi$ satisfies the same requirements as in the previous item. Let now~$B > 1$ be a number. We wish to prove the bound 
\eq{
\sum\limits_{k \in \Z} B^{\beta k}\Big\|f*\big(\Dil_{B^{-k}}[\varphi] - \Dil_{B^{-k-1}}[\varphi]\big)\Big\|_{L_{q,r}} \lesssim \|f\|_{B_{q,r}^{\beta,1}},
}
which generalizes the previous item. We mimic the same reasoning and for each~$k$ find~$m\in \Z$ such that~$B^k \sim A^m$ (the multiplicative constants are uniform with respect to~$k$). For example, we may set~$m = \lceil k \log B/ \log A \rceil$. The inequality above follows via the proof of Lemma~\ref{BesovLocalizationLemma} from the individual bound
\eq{
\bigg\|f*\Big(\Dil_{B^{-k} A^m}[\varphi] - \Dil_{B^{-k-1} A^m}[\varphi]\Big)\bigg\|_{L_{q,r}} \lesssim \sum\limits_{\ell\colon |\ell| \lesssim 1} \big\|f*(\psi_\ell - \psi_{\ell-1})\big\|_{L_{q,r}},
}
which, in its turn, holds by the same reasons as in the previous item: The Fourier transform of~$\Dil_{B^{-k} A^m}[\varphi] - \Dil_{B^{-k-1} A^m}[\varphi]$ is compactly supported outside the origin and the~$L_1$-norm of this function is uniformly bounded.

Thus, the definition~\eqref{DefOfBesovNorm} of the Besov--Lorentz norm does not depend on the choice of~$A$. This principle, in particular, allows us to vary~$A$ during the proof of Theorem~\ref{MainTheoremBesovLorentzScale}. The forthcoming lemma will later show that for Theorem~\ref{MainTheoremBesovLorentzScale}, the Besov--Lorentz scale improvement may be achieved with ease.
\begin{Le}\label{FromBesovToLorentzBesov}
Let~$0 < \alpha < \beta < d$,~$p = d/(d-\beta)$,~$q = d/(d-\alpha)$. Then,
\eq{\label{FromBesovToLorentzBesovFormula}
\|\I_\beta f\|_{\dot{B}_{p,1}^{0,1}} \lesssim \|\I_\alpha f\|_{\dot{B}_{q}^{0,1}}.
}
\end{Le}
\begin{proof}
By the definition of Besov-type norms, the semigroup properties of the Riesz potentials~\eqref{SemigroupForRieszPotentials}, and notation~$g = \I_\alpha f$,~\eqref{FromBesovToLorentzBesovFormula} follows from
\eq{\label{IndividualLorentz}
\Big\|\I_{\beta-\alpha}[g]*(\varphi_k - \varphi_{k-1})\Big\|_{L_{p,1}}  \lesssim \Big\|g*(\varphi_k - \varphi_{k-1})\Big\|_{L_q}.
}
By homogeneity, the case of general~$k$ reduces to the case~$k=0$. In that case we have stronger embeddings
\eq{
\Big\|\I_{\beta-\alpha}[g]*(\varphi - \varphi_{-1})\Big\|_{L_r}  \lesssim \Big\|g*(\varphi - \varphi_{-1})\Big\|_{L_q}
}
for any~$r \geq q$ (this is similar to the proof of~\eqref{IndividualBound}), and~\eqref{IndividualLorentz} follows from the interpolation estimate~\eqref{LorentzMultiplicative}.
\end{proof}
\begin{Rem}\label{LInftyRemark}
We may extend Lemma~\ref{FromBesovToLorentzBesov} to the case~$p=\infty$ without additional efforts: For any~$\alpha \in (0,d)$, the estimate
\eq{
\sum\limits_{k\in\Z} A^{-dk} \|f*(\varphi_k - \varphi_{k-1}) \|_{_{L_\infty}} \lesssim \|\I_\alpha f\|_{\dot{B}_{q}^{0,1}},
}
where~$q = d/(d-\alpha)$.
\end{Rem}

\item Assume that~$\beta > 0$ now. Consider the classical difference characterization to Besov norms as in~\cite{BIN1979}.  Denote the divided difference of order~$N$ with respect to~$j$-th coordinate by the symbol~$\Delta_j^{N}$:
\eq{
\Delta_j^{N}(h) f(x) = \sum\limits_{i=0}^N (-1)^{N-i} \binom{N}{i} f(x+ ihe_j), 
}
where~$e_j$ is the~$j$-th vector of the standard basis; the notation~$\binom{N}{i}$ is used for the binomial coefficient. The quantity
\eq{
\sum\limits_{j=1}^d\int\limits_0^\infty t^{-s_j} \|\Delta_j^{N}(t)f\|_{L_{q,r}}\,\frac{dt}{t}
}
defines the seminorm of the homogeneous anisotropic Besov space~$\dot{B}_{q,r}^{\vec{s},1}$, where~$\vec{s}= (s_1,s_2,\ldots, s_d)$ is a vector with positive coordinates. A folklore fact says this norm is equivalent to the norm of~$\dot{B}_{q,r}^{\beta,1}$ defined in~\eqref{DefOfBesovNorm}, where~$s_j = \beta /a_j$ and~$N$ is larger than any of the~$s_j$. We do not need the equivalence, we only wish to prove the one-sided inequality
\eq{\label{eq136}
\int\limits_0^\infty t^{-s_j} \|\Delta_j^{N}(t)f\|_{L_{q,r}}\,\frac{dt}{t} \lesssim \|f\|_{\dot{B}_{q,r}^{\beta,1}}, \qquad j=1,2,\ldots,d.
}
For the proof of~\eqref{eq136} in the case of the classical Lebesgue norms and anisotropic homogeneity, see Remark~$5.7$ in~\cite{Triebel2006}. We present the proof in the Lorentz case below, it mimics the known proofs for the Lebesgue scale.
We start similarly to the proof of Lemma~\ref{BesovLocalizationLemma}: It suffices to show that
\mlt{
\int\limits_{A^{-(k+1)a_j}}^{A^{-ka_j}} t^{-s_j} \|\Delta_j^{N}(t)f\|_{L_{q,r}}\,\frac{dt}{t}\\
 \lesssim \sum\limits_{\ell \geq k} A^{\beta_+(\ell - k) + \beta k}\|f*(\psi_\ell - \psi_{\ell-1})\|_{L_{q,r}} + \sum\limits_{\ell \leq k} A^{\beta_-(\ell - k) + \beta k}\|f*(\psi_\ell - \psi_{\ell-1})\|_{L_{q,r}},
}
where the smoothness parameters satisfy the bounds~$\beta_+ < \beta < \beta_-$. This, in its turn, reduces to 
\eq{\label{PointwiseDividedDifferenceBound}
\|\Delta_j^{N}(t)f\|_{L_{q,r}} \lesssim  \sum\limits_{\ell \geq 0} A^{\beta_+ \ell}\big\|f*(\psi_\ell - \psi_{\ell-1})\big\|_{L_{q,r}} + \sum\limits_{\ell < 0} A^{\beta_- \ell}\big\|f*(\psi_\ell - \psi_{\ell-1})\big\|_{L_{q,r}},
}
for any~$t \in [A^{-a_j},1]$. The reduction is also similar to Lemma~\ref{BesovLocalizationLemma}, the only additional ingredient is the dilation formula
\eq{\label{DilationAndDividedDifference}
\Delta_j^N(t)\big[\Dil_{A^{-k}} [g]\big](x) = \Dil_{A^{-k}}\Big[\Delta_j^N(A^{a_jk} t)[g]\Big].
}
The estimate~\eqref{PointwiseDividedDifferenceBound} would, in its turn, follow from the bounds
\alg{
\Big\|\Delta_j^{N}(t)f*(\psi_\ell - \psi_{\ell-1})\Big\|_{L_{q,r}} \lesssim   A^{\beta_+ \ell}\big\|f*(\psi_\ell - \psi_{\ell-1})\big\|_{L_{q,r}} ,\qquad & \ell \geq 0;\\
\label{eqC021}\Big\|\Delta_j^{N}(t)f*(\psi_\ell - \psi_{\ell-1})\Big\|_{L_{q,r}} \lesssim   A^{\beta_- \ell}\big\|f*(\psi_\ell - \psi_{\ell-1})\big\|_{L_{q,r}} ,\qquad & \ell \leq 0.
}
The first inequality holds with~$\beta_+ = 0$: The divided difference is a linear combination of shifts of~$f$. Here we have used the assumption~$\beta > 0$. To prove the second inequality, consider a function~$\Phi \in \Sw(\R^d)$ whose Fourier transform is compactly supported outside the origin and equals one on the support of~$\hat{\psi} - \hat{\psi}_{-1}$. It suffices to show that
\eq{
\big\|\Delta_j^{N}(t) \Phi_\ell\big\|_{L_1}\lesssim A^{\beta_-\ell},\qquad \ell \leq 0;\qquad \Phi_\ell = \Dil_{A^{-\ell}}[\Phi],
}
this bound yields~\eqref{eqC021} via~\eqref{ConvolutionLorentz}.
Using~\eqref{DilationAndDividedDifference}, we transform this inequality into
\eq{
\int\limits_{\R^d} \big|\Delta_j^{N}(A^{a_j\ell}t)\Phi(x)\big|\,dx \lesssim A^{\beta_- \ell}.
}
This bound may be obtained from the integral representation
\eq{
\Delta_j^{N}(\tau)\Phi(x) = \int_0^\tau \int_0^\tau\ldots \int_0^\tau\frac{\partial^N \Phi}{\partial x_j^N}\Big(x+ (s_1 + s_2+ \ldots + s_N)e_j\Big)\,ds_1\,ds_2\ldots\,ds_N,
}
taking into account~$Na_j > \beta$ (we use the notation~$\tau = A^{a_j\ell}t$ in the estimates below):
\mlt{
\int\limits_{\R^d} \big|\Delta_j^{N}(A^{a_j\ell}t)\Phi(x)\big|\,dx\\
 \lesssim \int_0^\tau \int_0^\tau\ldots \int_0^\tau\int\limits_{\R^d}\Big|\frac{\partial^N \Phi}{\partial x_j^N}\Big(x+ (s_1 + s_2+ \ldots + s_N)e_j\Big)\Big|\,dx\,ds_1\,ds_2\ldots\,ds_N\\
  \leq \tau^{N}\|\nabla^N \Phi\|_{L_1} \lesssim A^{Na_j \ell} \lesssim A^{\beta_- \ell}.
}
\end{enumerate}

\subsection{Proofs of weighted lemmas}\label{SWeights}
\begin{proof}[Proof of Lemmas~\ref{Lemma41} and~\ref{Lemma42}]
We will be using the following elementary inequality:
\eq{\label{eqE01}
(1+|x|)^{-\theta}(1+|y|)^{-\theta} \leq (1+|x-y|)^{-\theta} \leq (1+|x|)^{-\theta}(1+|y|)^\theta,\qquad x,y\in \R^d,
}
which follows from the triangle inequality and also yields~\eqref{TypicalWeightSmoothness}.
We start with the proof of Lemma~\ref{Lemma41}:
\mlt{
\Heat[G](x,\vec{t}\,) =  \Big(\prod\limits_{j=1}^d(4\pi t_j)\Big)^{-\frac12}\int\limits_{\R^d}G(x-y)e^{-\sum_{1}^d\frac{y_j^2}{4t_j}}\,dy\\
 \leq C \Big(\prod\limits_{j=1}^d(4\pi t_j)\Big)^{-\frac12}\int\limits_{\R^d}(1+|x-y|)^{-\theta}e^{-\sum_{1}^d\frac{y_j^2}{4t_j}}\,dy\\
\Leqref{eqE01} C(1+|x|)^{-\theta} \Big(\prod\limits_{j=1}^d(4\pi t_j)\Big)^{-\frac12}\int\limits_{\R^d}(1+|y|)^{\theta}e^{-\sum_{1}^d\frac{y_j^2}{4t_j}}\,dy.
}
Thus, Lemma~\ref{Lemma41} is reduced to the bound
\eq{\label{Eeq3}
\Big(\prod\limits_{j=1}^d(4\pi t_j)\Big)^{-\frac12}\int\limits_{\R^d}(1+|y|)^{\theta}e^{-\sum_{1}^d\frac{y_j^2}{4t_j}}\,dy \lesssim 1, \qquad \forall j\in [1\twodots d] \quad t_j \in [0,2].
}
Similarly, Lemma~\ref{Lemma42} is reduced via~\eqref{eqE01} to
\eq{\label{Eeq4}
\Big(\prod\limits_{j=1}^d(4\pi t_j)\Big)^{-\frac12}\int\limits_{\R^d}(1+|y|)^{-\theta}e^{-\sum_{1}^d\frac{y_j^2}{4t_j}}\,dy \gtrsim 1, \qquad \forall j\in [1\twodots d] \quad t_j \in [0,2].
}
We may justify~\eqref{Eeq4} by the substitution~$z_j = y_j/\sqrt{t_j}$:
\eq{
\Big(\prod\limits_{j=1}^d(4\pi t_j)\Big)^{-\frac12}\int\limits_{\R^d}(1+|y|)^{-\theta}e^{-\sum_{1}^d\frac{y_j^2}{4t_j}}\,dy 
=(4\pi)^{-\frac{d}{2}}\int\limits_{\R^d}\Big(1+\big(\sum_{1}^d t_j^2 z_j^2\big)^{\frac12}\Big)^{-\theta}e^{-\frac{|z|^2}{4}}\,dz
}
and note that even the part of the latter integral over the unit ball is bounded away from zero since~$t_j \leq 2$. The estimate~\eqref{Eeq3} follows from the bound~$(1+|y|)^\theta \lesssim e^{|y|^2/10}$:
\eq{
\Big(\prod\limits_{j=1}^d(4\pi t_j)\Big)^{-\frac12}\int\limits_{\R^d}(1+|y|)^{\theta}e^{-\sum_{1}^d\frac{y_j^2}{4t_j}}\,dy \lesssim \Big(\prod\limits_{j=1}^d(4\pi t_j)\Big)^{-\frac12}\int\limits_{\R^d}e^{-\sum_{1}^dy_j^2 (\frac{1}{4t_j} - \frac{1}{10})}\,dy,
}
and it remains to notice that~$\frac{1}{4t_j} - \frac{1}{10}$ and~$\frac{1}{4t_j}$ are comparable since~$t_j \leq 2$:
\eq{
\Big(\prod\limits_{j=1}^d(4\pi t_j)\Big)^{-\frac12}\int\limits_{\R^d}e^{-\sum_{1}^dy_j^2 (\frac{1}{4t_j} - \frac{1}{10})}\,dy \lesssim \prod\limits_{j=1}^d \Big(\frac{1}{4t_j} - \frac{1}{10}\Big)^{\frac12}\int\limits_{\R^d}e^{-\sum_{1}^dy_j^2 (\frac{1}{4t_j} - \frac{1}{10})}\,dy=\pi^{\frac{d}{2}}.
}
\end{proof}
\begin{proof}[Proof of Lemma~\ref{Lemma43}] This is completely similar to the proof above: We need to replace the inequality~\eqref{eqE01} with
\eq{
C^{-1} (1+|y|)^{-\theta}G(x) \leq G(x-y) \leq C(1+|y|)^{\theta}G(x)
}
and use the bounds~\eqref{Eeq3} and~\eqref{Eeq4}.
\end{proof}
\begin{proof}[Proof of Lemma~\ref{Lemma44}]
As usual for~$L_1$-estimates,~\eqref{eq416} is reduced to the case~$f=\delta_x$,~$x\in \R^d$, in which it reads as follows:
\eq{
\Big(\prod\limits_{j=1}^d(4\pi t_j)\Big)^{-\frac12}\Big(\int\limits_{\R^d}v(y)e^{-q\sum_{1}^d\frac{|x_j-y_j|^2}{4t_j}}\,dy\Big)^{\frac1q}\lesssim u(x).
}
Since the~$t_j$ are bounded away from zero and infinity, the first factor can be absorbed into the constant. The second factor does not exceed
\eq{
\Big(C_v\int\limits_{\R^d}(1+|x-y|)^{-\theta_v}e^{-q\sum_{1}^d\frac{|y_j|^2}{4t_j}}\,dy\Big)^{\frac1q} \lesssim (1+|x|)^{-\frac{\theta_v}{q}}
}
by the reasoning presented in the previous proof (we may formally cite Lemma~\ref{Lemma41} and use that~$t_j/q < 2$). The lemma follows since~$\theta_v \geq q\theta_u$ by~\eqref{ThetavThetau}.
\end{proof}

The proof of Lemma~\ref{Lemma47} is similar to the proof of Lemma~\ref{Lemma44}; the only difference is that we use Lemma~\ref{Lemma43} instead of Lemma~\ref{Lemma41}.
\begin{proof}[Proof of Lemma~\ref{Lemma48}]
By~\eqref{ReproducingThefk}, the inequality in question reduces to
\mlt{
\Big\| \Heat[g](\fdot,\vec{s}\,)\Big\|_{L_q(Q_{0,0})}\lesssim A^{\alpha}\|g\|_{L_1(w)},\\
 \vec{s} = \Big(A^{-2a_1} - A^{-2a_1K},A^{-2a_2} - A^{-2a_2K},\ldots, A^{-2a_d} - A^{-2a_dK}\Big).
}
Using dilations, this may be rewritten as
\mlt{
\Big\| \Heat[h](\fdot,\vec{t}\,)\Big\|_{L_q(\Dil_A Q_{0,0})}\lesssim \|h\|_{L_1(\Dil^{A} w)},\\
 \vec{t} = \Big(1 - A^{-2a_1(K-1)},1 - A^{-2a_2(K-1)},\ldots, 1 - A^{-2a_d(K-1)}\Big).
}
This inequality follows from Lemma~\ref{Lemma47} since
\eq{
\chi_{Q_{0,0}}(\Dil_{A^{-1}} x)\lesssim \Dil^{A} v(x),
}
where~$v(x) = (1+|x|)^{-\theta}$ for sufficiently large~$\theta$ and the latter weight meets the assumptions of the said lemma.
\end{proof}

\section{Examples}\label{s13}
\subsection{Anisotropic Sobolev spaces.}
The first series of examples comes from the classical generalizations of the Gagliardo--Nirenberg embedding~$\dot{W}_1^1 (\R^d)\hookrightarrow L_{d/(d-1)}$, going back to~\cite{Gagliardo1959} and~\cite{Nirenberg1959}. In these generalizations, one replaces the derivative~$\frac{\partial f}{\partial x_j}$ with a  higher order pure derivative. If the orders of these derivatives are different, the problem becomes anisotropic. The sharpest results in this direction were obtained by V. I. Kolyada in~\cite{Kolyada1993} (see~\cite{BesovIlin1969} and~\cite{Solonnikov1972} for earlier results). Here we mean the sharpness with respect to the function space scale: one wishes to embed into the narrowest space possible, preserving the invariance properties of the problem. We cite Theorem~$4$, case~$p=1$, from~\cite{Kolyada1993}; we also adjust notation.
\begin{Th}[Theorem~$4$ in~\cite{Kolyada1993}]\label{KolyadaThm}
Let~$d \geq 2$, let~$r_1,r_2,\ldots,r_d$ be natural numbers, and let also~$q\in (1,d/(d-r))$. Set
\eq{\label{eq131}
r = d\Big(\sum\limits_{j=1}^d \frac{1}{r_j}\Big)^{-1},\qquad b_j = r_j\Big(1 - \frac{d}{r}\cdot \frac{q-1}{q}\Big),\qquad j=1,2,\ldots, d.
}
For any smooth function~$f$, the inequality
\eq{\label{eq132}
\sum\limits_{j=1}^d \int\limits_0^\infty t^{-b_j - 1} \|\Delta_j^{r_j}(t) f\|_{L_{q,1}}\,dt \lesssim \sum\limits_{j=1}^d \big\|\partial_j^{r_j} f\big\|_{L_1}
}
holds true with a uniform constant.
\end{Th}
We have used the notation~$\partial_j$ for~$\frac{\partial}{\partial x_j}$. The number~$r$ in~\eqref{eq131} is often called the \emph{mean smoothness}. Kolyada's theorem above is dilation invariant, provided we choose the anisotropy
\eq{\label{eqD13}
a_j = \frac{d}{r_j}\Big(\sum\limits_{i=1}^d \frac{1}{r_i}\Big)^{-1},\qquad j=1,2,\ldots,d.
}
As we will see later, the mean smoothness~$r$ may be interpreted as the order of the~$a$-homogeneous differential operator~$(\partial_1^{r_1},\partial_2^{r_2},\ldots, \partial_d^{r_d})$.
The norm on the left-hand side of~\eqref{eq132}, in its turn, may be interpreted as a norm in a Besov--Lorentz space, see the discussion after formula~\eqref{eq136}. For that we introduce another mean smoothness~$b$:
\eq{
b = d\Big(\sum\limits_{j=1}^d \frac{1}{b_j}\Big)^{-1},
}
similar to the definition of~$r$. The expression on the left hand side of~\eqref{eq132} is equivalent to the~$\dot{B}_{q,1}^{b,1}$-norm of~$f$, since
\eq{
a_j b_j = r - \frac{d(q-1)}{q} = b
}
for all~$j=1,2,\ldots, d$. We will not prove this equivalence, only one-sided bound needed to derive Theorem~\ref{KolyadaThm} from our results. The corresponding inequality will be justified slightly later.

We define the space~$\W$ by the formula
\eq{
\W = \Set{g\in \Sw'(\R^d;\R^d)}{\exists f\in \Sw'(\R^d) \quad \forall j \quad g_j = \partial_j^{r_j} f}.
}
If~$d\geq 2$, this space does not contain delta measures. This follows from the description
\eq{
\W = \set{g \in \Sw'(\R^d; \R^d)}{\forall i \ne j\qquad  \partial_j^{r_j} g_i = \partial_i^{r_i} g_j}.
}
If~$a\otimes \delta_0 \in \W$, then~$\partial_j^{r_j}[a_i \delta_0] =\partial_i^{r_i} [a_j\delta_0]$ for any~$i\ne j$; this immediately yields~$a_i=a_j=0$. Note that~$\delta_0 \in \W$ in the case~$d=1$.

Then, Theorem~\ref{MainTheoremBesovLorentzScale} and Lemma~\ref{FromBesovToLorentzBesov} imply
\eq{
\|g\|_{\dot{B}_{q,1}^{-\alpha,1}} \asymp \|\I_\alpha g\|_{\dot{B}_{q,1}^{0,1}} \lesssim \|g\|_{L_1}, \quad g\in \W,
}
where~$\alpha$ is defined by~$q$,~$\alpha = d(1-1/q)$. This would yield~\eqref{eq132} via~\eqref{eq136}, provided we show
\eq{
\|g\|_{\dot{B}_{q,1}^{-\alpha,1}} \asymp \|f\|_{\dot{B}_{q,1}^{b,1}}, \qquad \forall j \quad g_j = \partial_j^{r_j} f.
}
This bound, by the very definition~\eqref{DefOfBesovNorm}, is reduced to
\eq{\label{eqD19bis}
A^{-\alpha k}\|g*(\psi_k - \psi_{k-1})\|_{L_{q,1}} \asymp  A^{bk} \|f*(\psi_k - \psi_{k-1})\|_{L_{q,1}}, \quad k\in \Z.
}
Note that~$\hat{g}_j(\xi) = (2\pi i \xi_j)^{r_j} \hat{f}(\xi)$ for all~$j$, which also yields
\eq{\label{eqD19}
\hat{f}(\xi) = \sum\limits_{j=1}^d \frac{(-2\pi i \xi_j)^{r_j}\hat{g}_j(\xi)}{\sum_{k} |2\pi \xi_k|^{2r_k}}.
}
These formulas reduce~\eqref{eqD19bis} to the bounds
\eq{
\Big\|\mathcal{F}^{-1}\big[(\xi_j)^{r_j} \hat{\Psi}_k(\xi)\big]\Big\|_{L_1} \lesssim A^{(\alpha + b)k};\qquad \bigg\|\mathcal{F}^{-1}\Big[\frac{(-2\pi i \xi_j)^{r_j}\hat{\Psi}_k(\xi)}{\sum_{i} |2\pi i \xi_i|^{2r_i}}\Big]\bigg\|_{L_1} \lesssim A^{-(\alpha + b)k},
}
where~$\Psi$ is a Schwartz function whose Fourier transform is compactly supported outside the origin and equals one on the support of~$\hat{\psi} - \hat{\psi}_{-1}$;~$\Psi_k$ is defined in the usual way,~$\Psi_k = \Dil_{A^{-k}}[\Psi]$. Fix~$j$. For the case~$k=0$, the bounds are clearly true, and the question is whether they are uniform in~$k$. In fact, they are dilation invariant, and we only need to verify that the orders of homogeneity on the left and right hand sides coincide. The functions on the left hand sides are homogeneous of the order~$a_jr_j$ and~$-a_jr_j$, correspondingly. Thus, we need to check that~$a_jr_j = \alpha + b$. By~\eqref{eqD13},~$a_jr_j = r$ and we always have~$\alpha = d - d/q$; we now need~$r = b + d - d/q$, which follows from~\eqref{eq131}. Alternatively,~$r$ is the order of the operator~$(\partial_1^{r_1},\partial_2^{r_2},\ldots, \partial_d^{r_d})$,~$b$ is the order of smoothness on the left hand side~\eqref{eq132}, and thus their difference should coincide with the order of the operator that transforms the right hand side into the left hand side, which is exactly~$\alpha$.

\subsection{Fourier constrained spaces.}\label{sD2}
Since we will be working with the Fourier transform, it will be convenient to switch to complex scalars. Let~$\ell = 2l$ and~$k < l$. By~$G(l,k)$ we denote the complex Grassmannian, the collection of all~$k$-dimensional~$\Co$-linear subspaces of the space~$\Co^l$ equipped with smooth structure. Let~$\Omega\colon S^{d-1}\to G(l,k)$ be a smooth function. It naturally generates a smooth vector bundle~$\sqcup_{\zeta \in S^{d-1}} \Omega(\zeta)$. Consider the function space
\eq{
W_1^\Omega = \Set{f\in L_1(\R^d, \Co^l)}{\forall \xi \in \R^d \setminus\{0\} \quad \hat{f}(\xi) \in \Omega\big(\Dil_{1/\rho(\xi)}(\xi)\big)},
}
where~$\rho$ is the anisotropic 'norm' defined in~\eqref{AnisotropicNorm}. In other words, the function~$\Omega$ defines an~$a$-homogeneous bundle on~$\R^d\setminus \{0\}$ and we restrict our attention to the functions~$f$ whose Fourier transforms are sections of this bundle. We say that the function~$\Omega$ defines the Fourier constraints and also that the elements~$f\in W_1^\Omega$ are subordinate to~$\Omega$. 
We will shortly explain how these spaces are related to Sobolev spaces. By definition,~$W_1^\Omega$ is dilation and translation invariant. It is also a closed subspace of~$L_1(\R^d,\Co^l)$.
\begin{Ex}\label{GradientExample}
Let~$l = d$ and~$k=1$. Assume~$d \geq 2$ to avoid technical issues. Consider the function~$\Omega(\zeta) = \zeta\cdot \Co$. If~$a = (1,1,\ldots,1)$, i.e., we consider the classical isotropic homogeneity, then
\eq{\label{eqD23}
W_1^\Omega = \Set{f\in L_1(\R^d,\Co^l)}{ \exists g\in \Sw'(\R^d)\quad \nabla g = f}.
}
Let us justify the identity above. The inclusion~$\supset$ follows from standard distribution theory: if a continuous function~$\hat{f}$ satisfies~$\hat{f}(\xi) = 2\pi i \xi \cdot \hat{g}(\xi)$ in the sense of distributions, then~$\hat{f}(\xi) \parallel \xi$ outside the origin in the classical pointwise sense. The reverse inclusion~$\subset$ requires to define a distribution~$g\in \Sw'(\R^d)$ for every~$f\in W_1^\Omega$. We choose
\eq{
\hat{g}(\xi) = \frac{\sum_{j=1}^d 2\pi i \xi_j \hat{f}_j(\xi)}{-4\pi^2 |\xi|^2}
}  
and get a locally summable function~$\hat{g}$ since~$d \geq 2$. This finishes the proof of~\eqref{eqD23}. The equation says that~$W_1^\Omega = \nabla \dot{W}_1^1$ for the choice~$\Omega(\zeta) = \zeta\cdot \Co$. 
\end{Ex}
\begin{Ex}
Let~$d=2$,~$l = 2$, and~$k=1$. Set~$a_1 = 3/2$,~$a_2 = 1/2$, and~$\Omega(\zeta) = (\zeta_1, - 4\pi^2 \zeta_2^3)\cdot \Co$. In such a case,
\eq{
\hat{f}(\xi) \in \Omega\big(\Dil_{1/\rho(\xi)}(\xi)\big) = \big(\xi_1,-4\pi^2\xi_2^3\big)\cdot \Co.
} 
Similarly to the previous example,
\eq{
W_1^\Omega = \Set{f\in L_1(\R^2, \Co^2)}{ \exists g\in \Sw'(\R^2)\quad f=(\partial_1 g,\partial_2^3 g)}.
}
\end{Ex}
\begin{Ex}\label{CurrentsExample}
Let~$a = (1,1,1,\ldots,1)$ and let~$\ell = \binom{d}{p}$, where~$p=0,1,\ldots, d-1$. We enumerate the basic vectors in~$\R^\ell$ with subsets of~$[1\twodots d]$ of cardinality~$p$. This describes the natural identification of~$\R^\ell$ with the space of exterior~$p$-forms~$\Lambda^p(\R^d)$. Set
\eq{
\Omega(\zeta) = \sset{\{v_I\}_{\#I = p}}{\forall J \subset [1\twodots d], \#J = p+1\qquad \sum\limits_{j\in J} \sign(j, J\setminus \{j\}) \zeta_j v_{J\setminus \{j\}} = 0}.
}
The notation~$\sign(j, J\setminus \{j\})$ is used to denote the sign of the permutation needed to re-order the string~$(j, J\setminus \{j\})$ alphabetically. In other words,
\eq{
\Omega(\zeta) = \Set{v \in \Lambda^p(\R^d)}{\zeta \wedge v = 0}.
}
The corresponding space~$W_1^\Omega$ is the space of summable closed differential~$p$-forms on~$\R^d$.
\end{Ex}
We may also generalize the concept of the gradient of a~$\BV$ function or of a divergence-free measure to general Fourier constraints~$\Omega$:
\eq{
\BV^\Omega = \Set{\mu\in \M(\R^d, \Co^l)}{\forall \xi \in \R^d \setminus\{0\} \quad \hat{\mu}(\xi) \in \Omega\big(\Dil_{1/\rho(\xi)}(\xi)\big)}.
}
In the case of the function~$\Omega$ described in Example~\ref{GradientExample}, we have~$\BV^\Omega = \nabla \dot{\BV}$, where~$\dot{\BV}$ is the homogeneous version of the space of functions of bounded variation. In the case considered in Example~\ref{CurrentsExample}, the space~$\BV^\Omega$ is the space of measure-valued closed~$p$-forms. 

Now we wish to define the space of distributions subordinate to~$\Omega$. This requires some work. The natural approach we survey below was suggested in~\cite{AyoushWojciechowski2017}. Consider the function~$\Omega^\perp\colon S^{d-1} \to G(l,l-k)$ that maps a point~$\zeta\in S^{d-1}$ to the subspace of~$\Co^l$ orthogonal to~$\Omega(\zeta)$; clearly, the function~$\Omega^\perp$ obtained this way is smooth. We also pick an auxiliary function~$H\in \Sw(\R^d)$ that attains positive values outside the origin and vanishes to infinite order at the origin:~$\forall N \in \mathbb{N}$ we have~$H(x) = O(|x|^N)$ as~$x\to 0$. Set
\eq{
\WW = \Set{f\in \Sw'(\R^d, \Co^l)}{\pi_{\Omega^\perp(\Dil_{1/\rho(\xi)}(\xi))} [f] \cdot H = 0}.
}
This space is translation and dilation invariant. It is also a closed subspace of~$\Sw'(\R^d,\Co^l)$. By standard distribution theory techniques,
\eq{\label{CorrespondenceOfDistrAndMeasures}
\WW \cap L_1(\R^d,\Co^l) = W_1^\Omega\quad \text{and}\quad \WW \cap \M(\R^d,\Co^l) = \BV^\Omega.
}
Thus, Corollary~\ref{LorentzCorollary} leads to the result below.
\begin{Cor}\label{CorollaryForOmega}
If the space~$\BV^\Omega$ does not contain measures of the type~$a\otimes \delta_0$,~$a\in \Co^l \setminus \{0\}$, then~$I_\alpha \colon W_1^\Omega \to L_{d/(d-\alpha),1}$ continuously.
\end{Cor}
We conclude this example with a simple observation that describes the functions~$\Omega$ for which the corresponding space~$\BV^\Omega$ does not contain delta measures. The lemma below follows from the fact that the Fourier transform of the charge~$a\otimes \delta_0$ equals~$a$ identically.
\begin{Le}\label{AlgebraicCancelling}
The space~$\BV^\Omega$ does not contain measures of the type~$a\otimes \delta_0$,~$a\in \Co^l \setminus \{0\}$, if and only if
\eq{
\bigcap_{\zeta\in S^{d-1}} \Omega(\zeta) = \{0\}.
}
\end{Le}
This cancellation condition appeared in~\cite{RoginskayaWojciechowski2006} and~\cite{VanSchaftingen2013} independently.

\subsection{Differential operators.}
Consider a differential operator~$A(\partial)$. We assume it is linear and has constant coefficients. Assume it is~$a$-homogeneous:
\eq{
A(\partial) u = \sum\limits_{\scalprod{\gamma}{a} = m} c_{\gamma}\partial^\gamma u.
}
Here~$\gamma \in \Z^{d}_+$ and~$\partial^\gamma = \partial_1^{\gamma_1}\partial_2^{\gamma_2}\ldots \partial_d^{\gamma_d}$,~$|\gamma| = \sum_j \gamma_j$. The number~$m$, which is the~$a$-order of the operator, need not be integer: See Example~\ref{SimplestAnistropicGradient} below. We will be working with vectorial differential operators. So, here~$u\in \Sw(\R^d, \Co^{\kappa})$ and~$A(\partial) u \in \Sw(\R^d, \Co^l)$. Thus, the coefficients~$c_\gamma$ are~$l\times \kappa$ matrices with complex entries. To each differential operator, we assign its \emph{symbol}:
\eq{
A(\xi) = \sum\limits_{\scalprod{\gamma}{a} = m} (2\pi i)^{|\gamma|}c_{\gamma}\xi^\gamma.
}
We say that the operator~$A$ is \emph{injectively elliptic}, provided~$A(\xi)$ is an injective linear operator for any~$\xi\in \R^d \setminus \{0\}$. We will say that~$A$ is a \emph{differential operator of constant rank}, provided the rank of~$A(\zeta)$ does not depend on~$\zeta\in S^{d-1}$. The simplest example of an injectively elliptic operator is the gradient~$\nabla$. One may also consider its anisotropic version
\eq{\label{AnisotropicGradient}
A(\partial) = \Big(\partial_1^{r_1},\partial_2^{r_2},\ldots, \partial_{d}^{r_d}\Big).
}
Here we define the anisotropy by~\eqref{eqD13}. This operator is also injectively elliptic. Another simple example of an elliptic operator is the Laplacian~$\Delta$. The first example of a constant rank operator is the divergence. If~$d=3$ and we add the curl operator, i.e., consider the vector-valued operator~$(\mathrm{div}, \curl)$ mapping vector fields on~$\R^3$ into~$\Co^4$, we obtain an elliptic operator. The differential~$d$ acting on the space of~$p$-forms is injectively elliptic if and only if~$p=0$. The pair~$(d,\partial)$ is always injectively elliptic. See Example~\ref{HodgeDeRhamExample} below.

Each differential operator~$A$ of rank~$k$ defines the function~$\Omega\colon S^{d-1}\to G(l,k)$ in a natural way:
\eq{\label{OmegaForDifferentialOperators}
\Omega(\zeta) = \Image A(\zeta),\qquad \zeta \in S^{d-1}.
}
This function is smooth. One may wonder when the estimate
\eq{
\|\partial^s u\|_{L_q} \lesssim \|A(\partial) u\|_{L_1}
}
holds. Here~$s\in \Z^{d}_+$ is some fixed vector.  The parameters~$q$ and~$s$ are linked by the homogeneity conditions:
\eq{\label{HomogeneityForDifOperators}
q = d/(d-\alpha),\quad \text{where}\quad \alpha = m - \scalprod{s}{a},\qquad \alpha \in (0,d).
}
If the operator~$A$ is injectively elliptic and~$q \in (1,\infty)$, then, by a version of the H\"ormander--Mikhlin multiplier theorem, for example, the one in~\S 1.III of~\cite{FabesRiviere1966},
\eq{\label{AnisotropicMikhlin}
\|\partial^s u\|_{L_{q,1}} \lesssim \|\I_\alpha [A(\partial) u]\|_{L_{q,1}}.
}
The cited multiplier theorem claims~$L_q$ continuity, the~$L_{q,1}$ continuity follows by interpolation via~\eqref{LorentzInterpolation}.
\begin{Th}\label{JeanTheorem}
Assume~$A$ is an~$a$-homogeneous constant coefficient linear injectively elliptic operator of~$a$-order~$m$. For any~$s$ such that~$\scalprod{a}{s} < m$, the estimate
\eq{\label{cancelestimate}
\|\partial^s u\|_{L_{q,1}} \lesssim \|A(\partial) u\|_{L_1}
}
holds true, provided the homogeneity conditions~\eqref{HomogeneityForDifOperators} are fulfilled and the operator~$A$ is cancelling:
\eq{\label{Canceling}
\bigcap_{\zeta\in S^{d-1}} \Image A(\zeta) = \{0\}.
}
\end{Th}
\begin{proof}
Consider the corresponding spaces~$W_1^\Omega$ and~$\BV^\Omega$ with~$\Omega$ given by~\eqref{OmegaForDifferentialOperators}. By Lemma~\ref{AlgebraicCancelling}, the space~$\BV^\Omega$ does not contain delta measures. Therefore, by~\eqref{CorrespondenceOfDistrAndMeasures}, the space~$\WW$ of distributions subordinate to~$\Omega$ does not contain vectorial delta measures as well. Thus, we apply Corollary~\ref{CorollaryForOmega} together with the bound~\eqref{AnisotropicMikhlin} and obtain the desired result.
\end{proof}
We have shown that~\eqref{Canceling} implies~\eqref{cancelestimate}. Note that the reverse implication is immediate. For recent rearrangement-invariant extensions of the isotropic theory of canceling and co-canceling operators, see also~\cite{BCS2026}.
In the case where the operator~$A$ is not injectively elliptic, we are still able to obtain some bounds by the same reasoning. Note that without the ellipticity condition, the estimate~\eqref{AnisotropicMikhlin} may be false. 
\begin{Th}\label{ConstantRankTheorem}
Assume~$A$ is an~$a$-homogeneous constant coefficient linear differential operator of~$a$-order~$m$ and constant rank. Assume it is cancelling:~\eqref{Canceling} holds true. Then,
\eq{
\|\I_\alpha [A(\partial) u]\|_{L_{q,1}} \lesssim \|A(\partial) u\|_{L_1},
}
where~$q = d/(d-\alpha)$ and~$\alpha \in (0,d)$.
\end{Th}
\begin{Rem}
One may replace the Lorentz space~$L_{q,1}$ with the Besov space~$\dot{B}_{q}^{0,1}$ in Theorems~\ref{JeanTheorem} and~\ref{ConstantRankTheorem}, or with even narrower space~$\dot{B}_{q,1}^{0,1}$.
\end{Rem}

\begin{Ex}\label{SimplestAnistropicGradient}
We may apply Theorem~\ref{JeanTheorem} to the case of 'anisotropic gradient' described in Theorem~\ref{KolyadaThm}. We set~$k=1$,~$l= d$, and define the operator by~\eqref{AnisotropicGradient}. Further, set~$d=2$,~$r_1 = 1$, and~$r_2 = 2$. In other words, we wish to have the operator~$(\partial_1,\partial_2^2)$ on the right hand side of our inequality. Formula~\eqref{eqD13} suggests the anisotropy~$a = (4/3,2/3)$. The chosen operator is injectively elliptic, cancelling, and has order~$m=4/3$. Note that the order of such an innocent operator is already non-integer. Let us try to bound the~$L_q$ norm of the function~$f$ itself. Then, by~\eqref{HomogeneityForDifOperators}, we need~$\alpha = m = 4/3$. By the same formula, we have~$q=3$ and arrive at the bound
\eq{\label{PartEx1}
\|f\|_{L_3} \lesssim \|\partial_1f\|_{L_1} + \|\partial_2^2 f\|_{L_1}
}
for Schwartz functions in two variables. Note that this inequality does not formally follow from Theorem~\ref{KolyadaThm} since the inequality~$q < d/(d-r)$ is strict in the latter theorem.
\end{Ex}
\begin{Ex}
Let now~$k=1$,~$l = 2$,~$d \geq 2$, and
\eq{
a_1 = \frac{2d}{d+1},\ a_2=a_3 = \ldots = a_d = \frac{d}{d+1}.
}
Set
\eq{
A(\partial)  = \Big(\partial_1, \partial_2^2 + \partial_3^2 + \ldots + \partial_d^2\Big) = \Big(\partial_1,\Delta_{2,3,\ldots, d}\Big).
}
The symbol of this operator is~$(2\pi i \xi_1, -4\pi^2\sum_{2}^d \xi_j^2)$ and the order is~$2d/(d+1)$. In particular, this operator is injectively elliptic and cancelling. Here~$\alpha= 2d/(d+1)$,~$q = (d+1)/(d-1)$, and we obtain
\eq{
\|f\|_{L_{\frac{d+1}{d-1}}} \lesssim \|\partial_1 f\|_{L_1} + \Big\|\Delta_{2,3,\ldots, d} f\Big\|_{L_1}.
}
In the case~$d=2$ it reduces to~\eqref{PartEx1}. In the case~$d=3$ it reads as
\eq{
\|f\|_{L_2}\lesssim \|\partial_1f\|_{L_1} + \Big\|(\partial_2^2 + \partial_3^2)f\Big\|_{L_1};
}
the latter inequality was obtained in~\cite{KMS2015}, see~\cite{Stolyarov2021bibis} as well.
\end{Ex}
\begin{Ex}\label{HodgeDeRhamExample}
This example generalizes Example~\ref{CurrentsExample} and adjusts the Hodge--de Rham complex to the anisotropic setting. We refer the reader to~\cite{Lee2013} for background. Choose a vector~$\vec{r} = (r_1,r_2,\ldots, r_d)$ with natural entries. Define the operator~$\dr$ that maps differential~$p$-forms into~$(p+1)$-forms by the rule
\eq{
\dr \omega = \sum\limits_{\genfrac{}{}{0pt}{-2}{I \subset [1\twodots d]}{\# I = p}} \sum\limits_{j\notin I} \partial_j^{r_j}\omega_I dx^j \wedge dx^I,\qquad \omega = \sum\limits_{\genfrac{}{}{0pt}{-2}{I \subset [1\twodots d]}{\# I = p}} \omega_I dx^I.
}
Here~$\omega_I\in C^{\infty}(\R^d)$ are the coefficients of the~$p$-form~$\omega$ in the standard basis
\eq{
dx^I = dx^{i_1}\wedge dx^{i_2}\wedge \ldots \wedge dx^{i_p}, \quad I = (i_1,i_2,\ldots, i_p).
} 
The symbol of~$\dr$ is
\eq{
\Lambda^p(\R^d)\ni \omega \mapsto v(\xi)\wedge \omega,\qquad v(\xi) = \big((2\pi i\xi_1)^{r_1},(2\pi i\xi_2)^{r_2},\ldots, (2\pi i\xi_d)^{r_d}\big),\ \xi \in \R^d.
}
In the case~$\vec{r} = (1,1,\ldots, 1)$, this operator coincides with the classical exterior derivative~$d$. We may also consider a version of the codifferential~$\codr$ as the operator with the symbol:
\eq{
\Lambda^p(\R^d)\ni \omega \mapsto \iota_{\overline{v(\xi)}}\,\omega,\quad  \xi \in \R^d.
}
One may see that the operators~$\dr$ and~$\codr$ are adjoint in the sense that
\eq{\label{AdjointExtDer}
\scalprod{\dr \omega}{\eta} = \scalprod{\omega}{\codr \eta},\qquad \omega \in C^{\infty}(\R^d,\Lambda^p(\R^d)) \text{ and } \eta \in C^{\infty}(\R^d, \Lambda^{(p+1)}(\R^d)),
}
at least one of these forms is compactly supported, and we use the standard scalar product on~$\Lambda^p(\R^d)$ and spaces of differential forms. 

We use the anisotropy defined by~\eqref{eqD13} and see that both~$\dr$ and~$\codr$ are~$a$-homogeneous differential operators of order~$r$. The latter symbol denotes the mean smoothness of~$\vec{r}$ defined in~\eqref{eq131}. The operator~$\dr$ is of constant rank since the corresponding function~$\Omega$ is defined by
\eq{
\Omega(\zeta) = \Set{v(\zeta)\wedge \omega}{ \omega\in \Lambda^p(\R^d)} = \set{\eta \in \Lambda^{p+1}(\R^d)}{v(\zeta) \wedge \eta = 0},\qquad \zeta \in S^{d-1}.
}
The dimension of the latter set is~$\binom{d-1}{p}$. By~\eqref{AdjointExtDer},~$\codr$ is also of constant rank.

We claim that the pair~$(\dr,\codr)$ that maps a~$p$-form into a pair of~$(p-1)$ and~$(p+1)$-forms, is an injectively elliptic operator. For that, we need to prove that for any~$\zeta \in S^{d-1}$ fixed the equalities~$v(\zeta)\wedge \omega = 0$ and~$\iota_{\overline{v(\zeta)}}\,\omega = 0$ imply~$\omega = 0$. This follows from the identity
\eq{
\big|v\wedge \omega\big|^2 + \big|\iota_{\overline{v}}\,\omega\big|^2 = |v|^2|\omega|^2.
}
Next, we claim that~$\dr$ is cancelling, provided~$p < d-1$. We need to verify that
\eq{
\bigcap_{\zeta \in S^{d-1}} \Set{v(\zeta)\wedge \omega}{\omega\in \Lambda^p(\R^d)} = \{0\},\qquad p = 0,1,\ldots, d-2.
}
Assume the contrary and let~$\eta\in \Lambda^{(p+1)}(\R^d)$ be a non-zero form lying in this intersection. Then,~$v(\zeta)\wedge \eta = 0$ for any~$\zeta\in S^{d-1}$. In particular,~$dx^j \wedge \eta = 0$ for any~$j \in [1 \twodots d]$. This means~$\eta = c \det$ for~$c \ne 0$, i.e.,~$\eta \in \Lambda^d(\R^d)$, which contradicts our assumptions about~$p$. Thus,~$\dr$ is cancelling, provided~$p<d-1$. Consequently,~$\codr$ is cancelling, provided~$p>1$.  Therefore, the pair~$(\dr,\codr)$ is cancelling if and only if~$1 < p < d-1$. Theorem~\ref{JeanTheorem} then delivers the following result. In the classical case~$r_j = 1$ it was obtained by Lanzani and Stein in~\cite{LanzaniStein2005} for the Lebesgue scale, see~\cite{HRS2023} and~\cite{VanSchaftingen2010} as well.
\begin{Cor}
If~$d \geq 4$,~$p \in [2\twodots d-2]$, and~$r < d$, then
\eq{\label{SLAnisotropic}
\|\omega\|_{L_{q,1}} \lesssim \|\dr \omega\|_{L_1} + \|\codr \omega\|_{L_1}, \qquad \omega \text{ is a~$p$-form},
}
where~$q = d/(d-r)$ and~$r$ is the mean smoothness of~$\vec{r}$ defined in~\eqref{eq131}.
\end{Cor}
The inequality~\eqref{SLAnisotropic} is false in the cases~$p=1$ and~$p=d-1$. However, here we recover the following result in the case~$p=d-1$.
\begin{Cor}
Let~$d \geq 3$ and~$r < d$. Then,
\eq{
\|\omega\|_{L_{q,1}} \lesssim \|\codr \omega\|_{L_1}, \ \text{provided~$\omega$ is a closed~$(d-1)$-form},
}
and~$q = d/(d-r)$.
\end{Cor}
\begin{proof}
Since~$d \geq 3$,~$\codr$ is a cancelling constant rank operator. Theorem~\ref{ConstantRankTheorem} says
\eq{
\|\I_\alpha \codr \omega\|_{L_{q,1}} \lesssim \| \codr\omega\|_{L_1}.
}
The inequality
\eq{
\|\omega\|_{L_{q,1}} \lesssim   \|\I_\alpha [\codr \omega]\|_{L_{q,1}} = \|\I_\alpha [\codr \omega]\|_{L_{q,1}} + \|\I_\alpha [\dr \omega]\|_{L_{q,1}}
}
follows from the ellipticity of the pair~$(\dr,\codr)$ and the bound~\eqref{AnisotropicMikhlin}.
\end{proof}
\end{Ex}

\section{Supplementary facts}

\subsection{The necessity of the concentration assumption in Theorem~\ref{Theorem41}}\label{SNecessity}
We wish to show the necessity of a concentration assumption in Theorem~\ref{Theorem41}. Throughout this section we work with~$\W$ generated by the gradient as described in Example~\ref{GradientExample}. In particular, we work with the isotropic homogeneity~$a= (1,1,\ldots, 1)$ and classical heat extension. We also use our common notation for the weights~\eqref{OurWeights}.
\begin{St}\label{NecessityOfConcentrationProp}
For any~$\eps > 0$ and any~$A$ sufficiently large there exists a vector field~$F\in \W$ such that the atom~$(0,0)$ is~$\eps$-flat for~$F$ and the scaling parameter~$A$, i.e.,
\eq{
\|F_3\|_{L_1(\heat[w](\fdot, 1 - A^{-6}))} \leq (1+\eps)\|F_0\|_{L_1(w)},
}
however,
\eq{
\|F_1\|_{L_q(\heat[v](\fdot, \frac{1- A^{-2}}{q}))}  >  A^{\alpha} \|F_0\|_{L_q(v)}.
  }
\end{St}
\begin{Le}\label{LemmaE1}
For any~$A$ sufficiently large there exists a smooth compactly supported vector field~$f\in \W$ such that
\eq{\label{LemmaF2Inequality}
\|f_1\|_{L_q(\heat[v](\fdot, \frac{1- A^{-2}}{q}))} > A^{\alpha} \|f_0\|_{L_q(v)},
}
the weight~$v$ is given by~\eqref{OurWeights}.
\end{Le}
\begin{proof}
Let~$\tf \in \W$ be a smooth compactly supported vector field that has vanishing moments up to order~$N$; here~$N$ is a sufficiently large natural number. Set~$f(x) = A^{d/q} \tf(Ax)$. We use the dilation that preserves the~$L_q$ norm. Then,
\mlt{
\|f_0\|_{L_q(v)} \leq \|f_0\|_{L_\infty} \leq \|\mathcal{F}[f_0]\|_{L_1} = \int\limits_{\R^d} |\hat{f}(\xi)| e^{-4\pi^2 |\xi|^2}\,d\xi\\
 = A^{-\alpha}\int\limits_{\R^d} \big|\mathcal{F}[\tilde{f}] (\xi/A)\big| e^{-4\pi^2 |\xi|^2}\,d\xi \lesssim A^{-\alpha}\int\limits_{\R^d} \Big|\frac{\xi}{A}\Big|^{N+1} e^{-4\pi^2 |\xi|^2}\,d\xi = O(A^{-\alpha-N})
} 
by the vanishing moments assumption. What is more,
\eq{
\|f_1\|_{L_q(\heat[v](\fdot, \frac{1- A^{-2}}{q}))} \sim \|f_1\|_{L_q} \sim 1
}
by construction. Thus,~$f$ indeed serves as the desired function, provided~$A$ is sufficiently large.
\end{proof}
In fact, the constant~$A^{\alpha}$ in~\eqref{LemmaF2Inequality} may be replaced with any power of~$A$.

\begin{Le}\label{LemmaE2}
For any~$A > 1$, there exists a sequence of vector fields~$g^n\in \W$ such that~$g^n = e\otimes h^n$, where~$e\in \R^d \setminus \{0\}$ is a fixed vector and~$h^n$ is a non-negative function, and, moreover,
\alg{
\|g_1^n\|_{L_q(\heat[v](\fdot, \frac{1- A^{-2}}{q}))} \longrightarrow 0,\qquad n\to \infty;\\
\|g_3^n\|_{L_1(\heat[w](\fdot, 1 - A^{-6}))} \longrightarrow \infty,\qquad n\to\infty.
}
\end{Le}
\begin{proof}
Let~$e = (1,0,\ldots,0)$ for simplicity of notation. We will also construct measures instead of functions and leave the tedious smoothing procedure to the reader. Set
\eq{
\bar{w}(s) = \int\limits_{\R^{d-1}} w(s,x')\,dx'; \quad \bar{v}(s) = \int\limits_{\R^{d-1}} v(s,x')\,dx',\qquad s\in \R.
}
Then,~$\bar{w}(s) \asymp s^{d-1-\theta_1}$ and~$\bar{v}(s) \asymp s^{d-1-\theta_3}$.
We will construct~$h^n$ as
\eq{
h^n = (\bar{w}(n))^{-1}\cdot d\mathcal{H}_{d-1}\big|_{\set{x\in \R^d}{x_1 = n}}.
}
Clearly,~$g^n = e \otimes h^n \in \W$ and
\alg{
\label{eqE07}\|h_1^n\|_{L_q(\heat[v](\fdot, \frac{1- A^{-2}}{q}))} &\sim \frac{\bar{v}^{\frac1q}(n)}{\bar{w}(n)} ;\\
\|h_3^n\|_{L_1(\heat[w](\fdot, 1 - A^{-6}))} &\sim 1.
}
Note that the constants in these inequalities do not depend on~$s$ and by~\eqref{eq415},~$\bar{v}^{\frac1q}(n)/\bar{w}(n) \to 0$ as~$n\to \infty$. Thus, we may redefine~$g^n := c_ng^n$, where~$c_n$ is a sequence that tends to infinity sufficiently slowly.
\end{proof}

\begin{proof}[Proof of Proposition~\ref{NecessityOfConcentrationProp}] We fix some large~$A$ and construct a function~$f$ with the help of Lemma~\ref{LemmaE1}. Construct the functions~$g^n$ as in Lemma~\ref{LemmaE2} and consider the function~$F = f+ g^n$. Then, since~$f$ is smooth and compactly supported,
\mlt{
\|F_3\|_{L_1(\heat[w](\fdot, 1 - A^{-6}))} = \|g_3^n\|_{L_1(\heat[w](\fdot, 1 - A^{-6}))} + O(1)\\
  = \|g_0^n\|_{L_1(w)} + O(1) = \|F_0\|_{L_1(w)} + O(1),
}
which means that for any given~$\eps > 0$, the function~$F$ is~$\eps$-flat, provided~$n$ is sufficiently large. Moreover,
\mlt{
\|F_1\|_{L_q(\heat[v](\fdot, \frac{1- A^{-2}}{q}))}  \geq \|f_1\|_{L_q(\heat[v](\fdot, \frac{1- A^{-2}}{q}))} + o(1)\\
  > A^{\alpha} \|f_0\|_{L_q(v)} + o(1) \geq A^{\alpha} \|F_0\|_{L_q(v)} + o(1).
}
\end{proof}


\subsection{Why it is impossible to split the~$L_q$ norm inside a tree}\label{NecessityForNonSplit}
In this section, we consider an example that shows the need for considering trees in the proof of Theorem~\ref{MainTheoremBesovLorentzScale} and running induction over trees in Subsection~\ref{s42}. We will show that the estimate
\eq{\label{TooMuchSplitting}
\sum\limits_{k \geq 0} A^{-\alpha k}\sum\limits_{j \in \Z^d} \|f_k\|_{L_q(Q_{k,j})} \lesssim \|f\|_{L_1},\qquad f\in \W,
}
is impossible for the most natural example of the space~$\W$. Namely, set~$d=2$ and~$a = (1,1)$, i.e., we consider the classical isotropic homogeneity in two variables. Set~$\W$ to be the space of gradients described in Example~\ref{GradientExample}.

Let~$f$  be the gradient of a characteristic function of~$Q_{0,0}$. This vector field is a charge, however, a principle similar to Lemma~\ref{SmoothingLemma} says that it suffices to disprove~\eqref{TooMuchSplitting} for the case where~$f$ is a charge, with the~$L_1$ norm replaced by the total variation of that charge. Let us fix~$A$ and consider large values of~$k$. We wish to study the behavior of~$f_k$ in a neighborhood of the right side of the square~$Q_{0,0}$; call this side~$L$ and denote by~$n_L$ the outward pointing unit normal to~$L$. We see that 
\eq{
f_k \sim A^{k} n_L \qquad \text{on a tubular neighborhood of~$L$ of thickness~$\sim\!\! A^{-k}$}. 
}
Call this neighborhood~$U_k$. There are~$\sim\!\! A^k$ squares~$Q_{k,j}$ that lie inside~$U_k$. For each such square, we have
\eq{\label{LqnormSmallCube}
 \|f_k\|_{L_q(Q_{k,j})}  \sim \big(A^{qk} A^{-2k}\big)^{\frac1q} \sim A^{(1-2/q)k}.
}
Therefore,
\eq{
\sum\limits_{j \in \Z^d} \|f_k\|_{L_q(Q_{k,j})}  \gtrsim A^k A^{(1-2/q)k} =A^{2\frac{q-1}{q} k}= A^{\alpha k},
}
and the series on the left hand side of~\eqref{TooMuchSplitting} diverges. 

Further analysis of this example shows that the cubes~$Q_{k,j}$ intersecting~$U_k$ correspond to flat atoms for any reasonable~$\eps$, because the vector field~$f_k$ is approximately  a scalar multiple of a fixed vector on~$U_k$: It points into the direction~$(1,0)$. They form a tree~$\Tree$, and the corresponding set~$\TS_k$ essentially coincides with~$U_k$:
\eq{
\|f_k\|_{L_q(U_k)} = \Big(\sum\limits_{j\colon Q_{k,j}\subset U_k} \|f_{k}\|_{L_q(Q_{k,j})}^q\Big)^{\frac1q} \Lseqref{LqnormSmallCube} \Big(A^k A^{(q-2)k}\Big)^{\frac1q} = A^{\frac{q-1}{q}k}.
}
Thus, we indeed observe a geometric decay in our series, exactly as in Corollary~\ref{TreeBoundLebesgueCorollary}:
\eq{
\sum\limits_{k} A^{-\alpha k} \|f_k\|_{L_q(U_k)} \lesssim \sum\limits_kA^{-\alpha k} A^{\alpha k/2}  \lesssim 1.
}

\bibliography{/Users/mac/Documents/Bib/Mybib_26_8}{}
\bibliographystyle{plain}

St. Petersburg State University, Department of Mathematics and Computer Science;

d.m.stolyarov at spbu dot ru.
\end{document}